\documentclass[12pt]{amsart}
\usepackage[utf8]{inputenc}
\usepackage{mathrsfs} 
\usepackage{amssymb}
\usepackage{bm}
\usepackage{graphicx}
\usepackage[centertags]{amsmath}
\usepackage{amsfonts}
\usepackage{amsthm}
\usepackage{enumerate}
\usepackage{tocvsec2}
\usepackage{xcolor}
\usepackage[margin=1in]{geometry}

\newtheorem{theorem}{Theorem}[section]
\newtheorem*{theorem*}{Theorem}

\newtheorem{corollary}[theorem]{Corollary}
\newtheorem{lemma}[theorem]{Lemma}
\newtheorem{rem}[theorem]{Remark}

\newtheorem{proposition}[theorem]{Proposition}

\newtheorem{example}{Example}[section]

\theoremstyle{definition}

\newcommand{\ee}{\varepsilon}

\begin{document}

\title{A Ramsey theorem for pairs in trees}

\begin{abstract} We prove a sharp structural result concerning finite colorings of pairs in well-founded trees. 

\end{abstract}

\author{R.M. Causey}
\address{Department of Mathematics, Miami University, Oxford, OH 45056, USA}
\email{causeyrm@miamioh.edu}

\author{C. Doebele}
\address{Department of Mathematics, Miami University, Oxford, OH 45056, USA}
\email{doebelce@miamioh.edu}

\thanks{2010 \textit{Mathematics Subject Classification}. Primary: 	05D10.}
\thanks{\textit{Key words}: Combinatorics, Ramsey theory, trees.}

\maketitle

\section{Introduction}

The two most classical theorems in Ramsey theory, found in \cite{R}, are the following: 

\begin{theorem}\begin{enumerate}[(i)]\item For any $n,k< \omega$ and any function $f:[\omega]^{n+1}\to k+1$, there exists a cofinal subset $M$ of $\omega$ such that $f|_{[M]^{n+1}}$ is constant. \item For any $p,n,k< \omega$, there exists a number $r=r(p,n,k)< \omega$ such that for any $r\leqslant N< \omega$ and any function $f:[N+1]^{n+1}\to k+1$, there exists a subset $M$ of $N+1$ with cardinality $p+1$ such that $f|_{[M]^{n+1}}$ is constant. \end{enumerate}

\label{Ramsey}
\end{theorem}

Here, for a given set $M$, $[M]^{n+1}$ denotes the set of $n+1$-element subsets of $M$. Since the results of \cite{R}, Ramsey theory has developed into a rich area of mathematics concerned with the following general type of question, stated by Erd\H{o}s and Rado \cite{ER}: If sufficiently many objects are distributed over not too many classes, then at least one  class contains many of these objects.   Varying the objects, the classes, and one's notion of ``many'' produces a broad class of questions. Above, the objects are $n+1$-element subsets of a given set, the classes are level sets of the function $f$. In $(i)$, ``many'' means all $n+1$-element subsets of an infinite set, while in $(ii)$, ``many'' means all $n+1$-element subsets of a set of cardinality $p+1$.

The seminal work of Erd\H{o}s and Rado began by considering sets and well-ordered sets and partitions of the $n+1$-element subsets, with the goal being to find subsets with the same cardinality or order type all of whose $n+1$-element subsets lie in the same member of the partition.   Other authors have studied such partition problems for sets with other or more general structure. For example, Baumgartner \cite{B} has shown that for any countable ordinal $\alpha$ and any partition of the topological space $\omega^{\omega^\alpha}$ into finitely many sets, one of the members of the partition must contain a subset homeomorphic to the entire space.  An additional topic is the study of partition relations on trees (that is, partially ordered sets $(P, \leqslant)$ such that for each $t\in P$, the set $\{s\in P: s\leqslant t\}$ is well-ordered).  An important result regarding partition relations for trees is the Halpern-L\"{a}uchli Theorem \cite{HL}, which concerns partitions of finite products of infinite trees.   The Halpern-L\"{a}uchli Theorem has produced many variations and generalizations, such as Milliken's theorem \cite{Mi} on finite partitions of strong trees, countable colorings of perfect trees \cite{L}, and a dual version \cite{TT}.  For a fuller discussion of the theorem and its variants, see \cite{DH} and the references therein.

In this work, the objects of interest to us will be well-founded trees. We will replace cardinality in the previous theorems with the rank of the tree. We will color the linearly ordered pairs in the tree and hope to find  subtrees of large rank with monochromatic pairs.     Our argument is, under certain conditions on the ordinals $\alpha_i$, to view a tree of rank $\alpha_0 \cdot\alpha_1$ as a ``tree of trees.''  We will think of partitioning subsets of this tree into subtrees of rank $\alpha_0$ in such a way that the members of the partition become nodes of a tree of rank $\alpha_1$.   More generally, given a tree of rank $\alpha_0 \cdot\alpha_1\cdot\ldots \cdot\alpha_n$, we will view this as being a ``tree of trees of $\ldots$ of trees.''    At the highest resolution, we see a single tree of rank $\alpha_0$.  Zooming out one level, we see a tree of rank $\alpha_0\cdot\alpha_1$. Zooming out two levels, we see a tree of rank $\alpha_0\cdot\alpha_1\cdot\alpha_2$, etc.  The lowest resolution is, of course, the entire tree of rank $\alpha_0 \cdot\ldots \cdot\alpha_n$.   We then define the separation of two linearly ordered nodes in the tree to be the lowest resolution which allows us to see both nodes at once.  That is, the separation of two nodes $s$ and $t$ of the tree is defined to be $0$ if both $s$ and $t$ live in the same subtree of rank $\alpha_0$.  The separation is $1$ if they live in separate subtrees of rank $\alpha_0$, but the same subtree of rank $\alpha_0\cdot\alpha_1$, etc.  Our primary structural result is that we can stabilize any coloring of pairs on a subtree of a tree of a certain rank in such a way that the color of a pair is determined by the separation of that pair.  In what follows, for a tree $P$ and a pair $(s,t)$ in the tree, $\varsigma_P(s,t)$ denotes the separation, which we formally define in Section $3$. Also, for $1\leqslant n< \omega$, we define  $$\Lambda_n(P)=\{(t_0, \ldots, t_{n-1})\in \Pi_{i<n}P : t_0<\ldots <t_{n-1}\}.$$  Our main theorem, Theorem \ref{main1},  completely elucidates the stabilization properties of finite colorings of pairs in well-founded trees, and provides an explicit description of the nature of subtrees with monochronatic pairs. 

\begin{theorem} Suppose $P$ is a tree and $k,l< \omega$.  Suppose also that there exist ordinals $\ee_0\geqslant \ldots \geqslant \ee_l$ such that $$\text{\emph{rank}}(P)=\omega^{\omega^{\ee_0}}\cdot\ldots \cdot\omega^{\omega^{\ee_l}}.$$      Then for any function $f:\Lambda_2(P)\to k+1$,  there exist a subtree $Q$ of $P$ and a function $F:l+1\to k+1$ such that $\text{\emph{rank}}(Q)=\text{\emph{rank}}(P)$ and for each $(s,t)\in \Lambda_2(Q)$,   $\varsigma_Q(s,t)=\varsigma_P(s,t)$ and $F(\varsigma_Q(s,t))=f(s,t)$. 

\label{main1}
\end{theorem}

The last inequality is the aforementioned fact that, on the subtree $Q$, the color a pair $(s,t)$ receives is completely determined by the separation $\varsigma_Q$, which depends on the relative positions of $s$ and $t$ in the tree $Q$. The property $\varsigma_Q=\varsigma_P$ means, in a sense that we make precise later,  that the pairs $(s,t)\in \Lambda_2(Q)$ have the same relative positions in $Q$ as they do in $P$.

We will also prove the sharpness of the preceding  result.

\begin{theorem} Fix  $k,l<\omega$. \begin{enumerate}[(i)]\item Suppose $\ee_0\geqslant \ldots \geqslant \ee_l$, $P$ is a tree with $\text{\emph{rank}}(P)= \omega^{\omega^{\ee_0}}\cdot\ldots \cdot\omega^{\omega^{\ee_l}}$, and $f:\Lambda_2(P)\to k+1$ is a function.  If $Q\subset P$ and $F:l+1\to k+1$ are as in Theorem \ref{main1}, $j\leqslant k$,  $A=\{i\leqslant l: F(i)=j\}$, and $\alpha =\omega^{\omega^{\ee_0}\cdot 1_A(0)}\cdot \ldots \cdot\omega^{\omega^{\ee_l}\cdot 1_A(l)},$ then there exists a subtree $R$ of $P$ with $\text{\emph{rank}}(R)=\alpha$  such that $f|_{\Lambda_2(R)}\equiv j$.   Here we obey the convention that if $A=\varnothing$, $\alpha=1$.     \item Let $F:l+1\to k+1$ be a function.   Fix  $\ee_0\geqslant \ldots \geqslant \ee_l$ and a tree $P$ with $\text{\emph{rank}}(P)=\omega^{\omega^{\ee_0}}\cdot\ldots \cdot\omega^{\omega^{\ee_l}}$. Define $f:\Lambda_2(P)\to k+1$ by $f(s,t)=F(\varsigma_P(s,t)).$  If $j\leqslant k$, $A=\{i\leqslant l: F(i)=j\}$, $\alpha=\omega^{\omega^{\ee_0}\cdot 1_A(0)}\cdot \ldots \cdot\omega^{\omega^{\ee_l}\cdot 1_A(l)},$ and $Q$ is any subtree of $P$ such that $f|_{\Lambda_2(Q)}\equiv j$, then $\text{\emph{rank}}(Q) \leqslant \alpha$. \end{enumerate}
\label{slop}

\end{theorem}

 In what follows, for $1\leqslant n< \omega$, let $\mathfrak{R}_n$ denote the class of all $\xi\in \textbf{Ord}$ such that $\xi>0$ and  if $P$ is any tree with $\text{rank}(P)=\xi$, $k< \omega$, and $f:\Lambda_n(P)\to k+1$ is any function, then there exists $Q\subset P$ with $\text{rank}(Q)=\text{rank}(P)$ such that $f|_{\Lambda_n(Q)}$ is constant.  The following Theorem \ref{main2} was proved in \cite{C}. Our proof of Theorem \ref{main1} gives a new proof of Theorem \ref{main2}, while providing new information not contained in the proof of Theorem \ref{main2} found in \cite{C}.  Our hope is that this new information will provide a new approach to the problem of finding  explicit descriptions of $\mathfrak{R}_3, \mathfrak{R}_4, \ldots$, which, to our knowledge, are unknown.

\begin{theorem} $\mathfrak{R}_2$ is the class of multiplicatively indecomposable ordinals.

\label{main2}
\end{theorem}

\section{The basics}

Throughout this work, $\textbf{Ord}$ will denote the class of ordinals. We will assume a basic knowledge of ordinals. Each of the facts we state about ordinals can be found in \cite{M}.

For us, a \emph{tree} will be a partially ordered set $(P, \leqslant )$ such that for each $t\in P$, the ancestor set $$A_P[t]:=\{s\in P: s\leqslant t\}$$ is well-ordered.  Throughout, given a tree $P$, we will assume each subset of $P$ is endowed with the same order as $P$, and is therefore also a tree.  Given a tree $P$, we let $\text{Roots}(P)$ denote the set of minimal members of $P$. That is, the set of $t\in P$ such that $A_P[t]=\{t\}$.  Of course, we refer to the members of $\text{Roots}(P)$ as the \emph{roots} of $P$.     Given $t\in P$, we let $$P[t]=\{s\in P: t\leqslant s\}$$ and $$P(t)=\{s\in P: t<s\}.$$  For convenience, for any $t\notin P$, $P[t]=P(t)=\varnothing$. We let $\text{Leaves}(P)$ denote the set of maximal members of $P$.  That is, $\text{Leaves}(P)$ denotes the set of those $t\in P$ such that $P(t)=\varnothing$. If $t\in \text{Leaves}(P)$, we say $t$ is a \emph{leaf} of $P$. We define $\Pi(P)=\{(s,t)\in P\times \text{Leaves}(P): s\leqslant t\}$. 

%Given a tree $P$ and $t\in P$, we let $|t|_P=|A_P[t]|$.  For $1\leqslant i\leqslant |t|_P$, we let $t|^P_i=s_i$, where $A_P[t]=\{s_1, \ldots, s_{|t|_P}\}$ and $s_1<\ldots <s_{|t|_P}$.   We recall from the introduction that $$\Lambda_2(P)=\{(s,t)\in P\times P: s<t\}.$$   We also define $\Pi(P)=\{(s,t)\in P\times \text{Leaves}(P): s\leqslant t\}$. 

Given a tree $P$, we let $P'=P\setminus \text{Leaves}(P)$. The tree $P'$ is called the \emph{derivative} of $P$.   We then define by transfinite induction the $\xi^{th}$ \emph{derivative} $P^\xi$ by $$P^0=P,$$ $$P^{\xi+1}=(P^\xi)',$$ and if $\xi$ is a limit ordinal, $$P^\xi=\bigcap_{\zeta<\xi}P^\zeta.$$  Of course, if $\zeta\leqslant \xi$, $P^\xi\subset P^\zeta$.   From this it follows that there must exist $\zeta\in \textbf{Ord}$ such that $P^\zeta=P^{\zeta+1}$, and evidently $P^\xi=P^\zeta$ for all $\xi\geqslant \zeta$.   We define $P^\infty$ to be equal to $P^\zeta$, where $\zeta$ is the minimum (equivalently, any) ordinal such that $P^\zeta=P^{\zeta+1}$.  We say $P$ is \emph{well-founded} if $P^\infty=\varnothing$, and \emph{ill-founded} otherwise.  

For a subset $M$ of $P$, the \emph{downward closure of} $M$ \emph{in} $P$ is the set $$\{s\in P: (\exists t\in M)(s\leqslant t)\}.$$  We say a subset $M$ of $P$ is \emph{downward closed in} $P$  if $M$ is equal to its downward closure in $P$.

\begin{proposition} Let $P$ be a tree. $P$ is ill-founded if and only if there exist $t_0<t_1<t_2<\ldots$ such that $t_n\in P$ for all $n<\omega$. In particular, if $P$ is well-founded and $s\in P$, then $P[s]\cap \text{\emph{Leaves}}(P)\neq \varnothing$ and $P$ is the downward closure of $\text{\emph{Leaves}}(P)$. 
\label{easy1}
\end{proposition}

\begin{proof} If $P$ is ill-founded, then there exists an ordinal $\xi$ such that $\varnothing\neq P^\infty=P^\xi=P^{\xi+1}$.  Since $P^\xi\neq \varnothing$, we may fix $t_0\in P^\xi$.    Now assume that $t_0<\ldots <t_n\in P^\xi$ have been chosen.  Since $t_n\in P^\xi=P^{\xi+1}$, $t_n$ is not a leaf in $P^\xi$, whence there exists $t_{n+1}\in P^\xi$ such that $t_n<t_{n+1}$. This completes the recursive process of choosing $t_0<t_1<\ldots$ such that $t_n\in P^\xi\subset P$ for all $n<\omega$.  Conversely, suppose there exist $t_0<t_1<\ldots$ such that $t_n\in P$ for all $n<\omega$.  We claim that for each $n<\omega$ and each ordinal $\xi$,  $t_n\in P^\xi$ and $t_n$ is not a leaf in $P^\xi$.    This is an easy induction on $\xi$. Thus $t_0\in P^\infty$, and $P$ is ill-founded.

If $P[s]\cap \text{Leaves}(P)=\varnothing$, then let $s_0=s\in P[s]$.  Now suppose that $s_0<\ldots<s_n$ have been chosen. Since $s=s_0<\ldots <s_n$, $s_n\in P[s]$, and $s_n\in P\setminus \text{Leaves}(P)$, there exists $s_{n+1}\in P$ such that $s_n<s_{n+1}$.  Note that $s_{n+1}\in P[s]$.  This completes the recursive choice of $s_0<s_1<\ldots$ with $s_n\in P$ for all $n<\omega$. We deduce that $P$ is ill-founded by the previous paragraph. By contraposition, if $P$ is well-founded and $s\in P$, $P[s]\cap \text{Leaves}(P)\neq \varnothing$. 

For the final statement, suppose $P$ is well-founded and fix $s\in P$. Since $P[s]\cap \text{Leaves}(P)\neq \varnothing$, we may fix a leaf $t\in P[s]$. Then since $s\leqslant t$, $s$ lies in the downward closure of $\text{Leaves}(P)$ in $P$. Since $s$ was arbitrary, $P$ is contained in, and therefore equal to, the downard closure of $\text{Leaves}(P)$ in $P$.

\end{proof}

\begin{rem}\upshape By our definition, a partially ordered set $(P, \leqslant )$ is a tree if for any $t\in P$, $A_P[t]$ is well-ordered.  However, if for any $t\in P$, $A_P[t]$ is infinite, then it has an initial segment order isomorphic to $\omega$. In this case, $P$ is ill-founded.  Therefore in all cases of interest to us in this work, our trees will have the stronger property that for every $t\in P$, $A_P[t]$ is finite and linearly ordered.   We will often use this fact throughout without repeatedly recalling it.

\end{rem}

Another way of stating the first equivalence in Proposition \ref{easy1} is that a tree $P$ is ill-founded if and only if it has a subset order isomorphic to $\omega$.  From this it it easy to see that the following corollary is equivalent to Theorem \ref{Ramsey}$(i)$.   This equivalence yields that stabilization questions are already completely solved in the ill-founded case, which is the  reason we only concern ourselves with the well-founded case in this work.

\begin{corollary} For any $n,k< \omega$, any ill-founded tree $P$, and $f:\Lambda_{n+1}(P)\to k+1$, there exists an ill-founded subtree $Q$ of $P$ such that $f|_{\Lambda_{n+1}(Q)}$ is constant.

\end{corollary}

For each tree $P$, we let $\text{rank}(P)=\{\zeta:P^\zeta\neq \varnothing\}$, which is a downward closed class of ordinals.    Then $P$ is well-founded if and only if $\text{rank}(P)$ is an ordinal, and in this case $\text{rank}(P)=\min\{\xi: P^\xi=\varnothing\}$.  Moreover, $P$ is ill-founded if and only if $\text{rank}(P)=\textbf{Ord}$. It is obvious that if $Q\subset P$, then $Q^\xi\subset P^\xi$ for all $\xi\in \textbf{Ord}$, whence $\text{rank}(Q)\leqslant \text{rank}(P)$.

We next provide some simple examples to show that trees of all orders exist. 

\begin{example}\upshape  For two ordinals $\alpha, \beta$, let $$I(\alpha, \beta)=\{(\xi_0, \ldots, \xi_n): n< \omega, \beta> \xi_0>\ldots >\xi_n\geqslant \alpha\}.$$   We define the order on $I(\alpha, \beta)$ by letting $(\zeta_0, \ldots, \zeta_m)\leqslant (\xi_0, \ldots, \xi_n)$ if $m\leqslant n$ and $\zeta_i=\xi_i$ for all $i\leqslant m$. Note that $I(\alpha, \beta)=\varnothing$ if and only if $\alpha\geqslant \beta$.  It is evident that $$\text{Leaves}(I(\alpha, \beta))=\{(\xi_0, \ldots, \xi_n)\in I(\alpha, \beta): \xi_n=\alpha\},$$ and $I(\alpha, \beta)'=I(\alpha+1, \beta)$.    This is the successor step of an easy induction argument which shows that  for any ordinal $\zeta$, $I(\alpha, \beta)^\zeta=I(\alpha+\zeta, \beta).$  Therefore for any ordinal $\xi$, $\text{rank}(I(0,\xi))=\xi$.

\end{example}

For a well-founded tree $P$ and $t\in P$, we let $\tau_P(t)=\sup\{\xi: t\in P^\xi\}$.  Note that for a fixed $t\in P$, $\{\xi: t\in P^\xi\}$ is a downward closed class of ordinals, and $\tau_P(t)\leqslant \text{rank}(P)$.  Furthermore, the supremum in the definition of $\tau_P(t)$ is a maximum.  To see this, note that if the supremum $\tau_P(t)$ is zero or a successor, it is obviously a maximum. If the supremum $\tau_P(t)$ is a limit ordinal, then since $\{\xi: t\in P^\xi\}$ is downard closed, $$t\in \bigcap_{\zeta<\tau_P(t)} P^\zeta=P^{\tau_P(t)}.$$      Since $\tau_P(t)$ is a maximum, it follows that $\tau_P(t)< \text{rank}(P)$ for all $t\in P$.  For later purposes, for an ordinal $\beta>0$, a tree $P$, and $t\in P$, we also define $$\tau_{P,\beta}(t)=\sup \{\xi: t\in P^{\beta\cdot \xi}\}.$$   It is easy to see that this supremum is also a maximum.

If $P$ is a tree, $M$ is a set, and for each $t\in M$, $P_t\subset P$ is given, we use the notation $\amalg_{t\in M}P_t$  to refer to the tree $\cup_{t\in M}P_t$ as well as the assertion that this union is a totally incomparable union.  That is, for distinct $s,t\in M$ and $a\in P_s$ and $b\in P_t$, $a\not\leqslant b$ and $b\not\leqslant a$.  For any tree, $$P=\amalg_{t\in \text{Roots}(P)} P[t].$$

\begin{proposition} Let $P$ be a tree. \begin{enumerate}[(i)] \item If $\xi$ is a limit ordinal and $M$ is a cofinal subset of $\xi$, then $$\bigcap_{\zeta\in M}P^\zeta =P^\xi.$$

\item For any ordinal $\zeta$, $P^\zeta$ is downward closed in $P$. 

\item For any ordinals $\beta, \gamma$, $(P^\beta)^\gamma=P^{\beta+\gamma}$. In particular, if $P^\beta\neq \varnothing$, then  $\text{\emph{rank}}(P)=\beta+\gamma$ if and only if $\text{\emph{rank}}(P^\beta)=\gamma$.  

\item If $P$ is well-founded and $s,t\in P$ with $s<t$, then $\tau_P(s)>\tau_P(t)$.   

\item For ordinals $\beta, \gamma$, $(P\setminus P^\beta)^\zeta=P^\zeta\setminus P^\beta$.   

\item If $\amalg_{t\in M}P_t\subset P$, then for any ordinal $\zeta$, $$\bigl(\amalg_{t\in M} P_t\bigr)^\zeta= \amalg_{t\in M}P_t^\zeta.$$ 

\item For any $s\in P$, $P(s)\cap \text{\emph{Leaves}}(P)= \text{\emph{Leaves}}(P(s))$ and  $P[s]\cap \text{\emph{Leaves}}(P)= \text{\emph{Leaves}}(P[s])$.     

\item For $t\in P$ and any ordinal $\zeta$, $P^\zeta(t)=P(t)^\zeta$ and  $P^\zeta[t]=P[t]^\zeta$. This implies that if $P$ is well-founded, $\tau_{P(s)}=\tau_P|_{P(s)}$, and    $t\in P^\zeta$ if and only if $\tau_P(t)\geqslant \zeta$ if and only if $\text{\emph{rank}}(P(t))\geqslant \zeta$ if and only if $\text{\emph{rank}}(P[t])>\zeta$. 

\item For an ordinal $\zeta$, $t\in \text{\emph{Leaves}}(P^\zeta)$ if and only if $\text{\emph{rank}}(P(t))=\zeta$ if and only if $\tau_P(t)=\zeta$. 

\item If $P$ is well-founded and $\text{\emph{rank}}(P)\geqslant \beta$, then $\text{\emph{rank}}(P\setminus P^\beta)=\beta$.  Moreover, if $\beta, \gamma$ are ordinals and $t\in P^{\beta+\gamma}$, then  for any ordinal $\zeta$, $[P(t)\cap (P^\beta\setminus P^{\beta+\gamma})]^\zeta=P(t)\cap (P^\beta\setminus P^{\beta+\gamma})^\zeta$ and $\text{\emph{rank}}(P(t)\cap (P^\beta\setminus P^{\beta+\gamma}))=\gamma$.

\item If $\zeta$ is an ordinal,  $R\subset P$ is a non-empty, well-founded subtree of $P$,  and for each $t\in \text{\emph{Leaves}}(R)$, there exists $R_t\subset P(t)$ such that $\text{\emph{rank}}(R_t)=\zeta$, then for each $\mu\leqslant \zeta$ and each ordinal $\eta$,  $$\Bigl(R\cup \bigl(\amalg_{t\in \text{\emph{Leaves}}(R)} R_t\bigr)\Bigr)^\mu= R\cup \bigl(\amalg_{t\in \text{\emph{Leaves}}(R)} R_t^\mu\bigr),$$ and $$\Bigl(R\cup \bigl(\amalg_{t\in \text{\emph{Leaves}}(R)} R_t\bigr)\Bigr)^{\zeta+\eta}= R^\eta. $$ In particular, $$\text{\emph{rank}}\Bigl(R\cup \bigl(\amalg_{t\in \text{\emph{Leaves}}(R)} R_t\bigr)\Bigr)=\zeta+\text{\emph{rank}}(R).$$

\end{enumerate}

\label{ted}
\end{proposition}

\begin{proof}$(i)$ Obviously $\cap_{\zeta\in M}P^\zeta\supset \cap_{\zeta<\xi}P^\zeta=P^\xi$. If $t\in P\setminus P^\xi$, then there exists $\eta<\xi$ such that $t\notin P^\eta$. Since $M$ is cofinal in $\xi$, there exists $\mu\in  M$ such that $\eta<\mu<\xi$. Since $\cap_{\zeta\in M} P^\zeta\subset P^\mu\subset P^\eta$ and $t\notin P^\eta$, $t\notin \cap_{\zeta\in M}P^\zeta$. This yields $(i)$.

$(ii)$ It is  a trivial induction on $\xi$ that if $s,t\in P$ and $s<t\in P^\xi$, then $s\in P^\xi$.

$(iii)$ The first statement can be seen by induction on $\gamma$ for $\beta$ held fixed. The $\gamma=0$ case is clear. The successor case follows by the definition of $P^\xi$ when $\xi$ is a successor together with associativity of ordinal addition. The limit ordinal case uses $(i)$ together with the fact that for a limit ordinal $\gamma$ and any ordinal $\beta$, $\beta+\gamma$ is a limit ordinal and $\{\beta+\delta: \delta<\gamma\}$ is a cofinal subset of $\beta+\gamma$.

Suppose that $\text{rank}(P)=\beta+\gamma$.  Then for any ordinal $\mu$, $(P^\beta)^\mu=P^{\beta+\mu}$ is empty if and only if $\beta+\mu\geqslant \beta+\gamma$ if and only if $\mu\geqslant \gamma$, so that $\text{rank}(P^\beta)=\gamma$.        Next, suppose that $P^\beta \neq \varnothing$ and  $\text{rank}(P^\beta)=\gamma$.  Then $P^{\beta+\gamma}=(P^\beta)^\gamma=\varnothing$, so $\text{rank}(P)\leqslant \beta+\gamma$.    But for $\mu<\gamma$, $\beta+\mu<\beta+\gamma$, and $P^{\beta+\mu}=(P^\beta)^\mu\neq \varnothing$, so $\text{rank}(P)> \beta+\mu$. Since this holds for any $\mu<\gamma$, and since the set of such $\mu$ is non-empty because $P^\beta=P^{\beta+0}\neq \varnothing$, we deduce that $\text{rank}(P)\geqslant \beta+\gamma$.

$(iv)$ We have already discussed that $\tau_P(t)$ is an ordinal and $t\in P^{\tau_P(t)}$. By $(ii)$,  $s<t\in P^{\tau_P(t)}$, which yields $s\in P^{\tau_P(t)+1}$. From this it follows that $\tau_P(s)>\tau_P(t)$.

$(v)$ It is obvious that $(P\setminus P^\beta)^\zeta\subset P^\zeta \cap (P\setminus P^\beta)= P^\zeta\setminus P^\beta$.    We prove by induction on $\zeta$ that $P^\zeta\setminus P^\beta\subset (P\setminus P^\beta)^\zeta$. By the first sentence of $(v)$, this is equivalent to proving that $P^\zeta\setminus P^\beta= (P\setminus P^\beta)^\zeta$.      For $\zeta=0$, this is clear.   If $\zeta$ is a limit ordinal, \begin{align*} P^\zeta \setminus P^\beta & = \Bigl(\bigcap_{\mu<\zeta}P^\mu\Bigr)\setminus P^\beta =\bigcap_{\mu<\zeta} (P^\mu\setminus P^\beta) = \bigcap_{\mu<\zeta}(P\setminus P^\beta)^\mu= (P\setminus P^\beta)^\zeta.  \end{align*}  Now suppose $P^\mu\setminus P^\beta\subset (P\setminus B^\beta)^\mu$. If $P^{\mu+1}\setminus P^\beta=\varnothing$, then obviously this is contained in $(P\setminus P^\beta)^{\mu+1}$. Now if $t\in P^{\mu+1}\setminus P^\beta$, then $t\in P^{\mu+1}$ and $t\notin P^\beta$. Fix $t<s\in P^\mu$ and note that, since $t\notin P^\beta$ and $P^\beta$ is downward closed in $P$, $s\notin P^\beta$.  Therefore $s\in P^\mu\setminus P^\beta=(P\setminus P^\beta)^\mu$. Since $t<s\in (P\setminus P^\beta)^\mu$, $t\in (P\setminus P^\beta)^{\mu+1}$.

$(vi)$  First we note that if $\{P_t: t\in M\}$ are pairwise incomparable, then so are $\{P_t^\zeta: t\in M\}$ for any ordinal $\zeta$.  Also, for any $s\in M$, $P_s^\zeta\subset (\amalg_{t\in M}P_t)^\zeta$ for any ordinal $\zeta$. Then $\amalg_{s\in M}P_s^\zeta\subset (\amalg_{t\in M}P_t)^\zeta$.    We prove the reverse inclusion by induction on $\zeta$.    The $\zeta=0$ is clear, and so is the limit case once we observe that intersection commutes with incomparable union. That is, if $\{P_t: t\in M\}$ are pairwise incomparable, $$\amalg_{t\in M} \bigcap_{\mu<\zeta}P_t^\mu=\bigcap_{\mu<\zeta}\amalg_{t\in M}P_t^\mu.$$   Indeed, $s\in \amalg_{t\in M}\cap_{\mu<\zeta}P^\mu_t$ if and only if there exists $t\in M$ such that $s\in \cap_{\mu<\zeta}P^\mu_t$, while $s\in \cap_{\mu<\zeta}\amalg_{t\in M}P^\mu_t$ if and only if for each $\mu<\zeta$, there exists $t_\mu\in M$ such that $s\in P^\mu_{t_\mu}$. The first condition implies the second by taking $t_\mu=t$ for all $\mu<\zeta$. The second condition implies the first, since if $s\in \amalg_{t\in M}P_t$, there exists unique $t\in M$ such that $s\in P_t$, and it must be the case that $t_\mu=t$ for all $\mu<\zeta$.     For the successor case, it is sufficient to prove the $\zeta=1$ case.   We note that $s\in \amalg_{t\in M}P_t$ fails to be a leaf in $\amalg_{t\in M}P_t$ if and only if there exists $u\in \amalg_{t\in M}P_t$ such that $s<u$. But since this is an incomparable union, $s,u\in P_t$ for the same $t\in M$, whence $s$ also fails to be a leaf in $P_t$. Therefore $(\amalg_{t\in M}P_t)'\subset \amalg_{t\in M}P_t'$.

$(vii)$ If $u\in P(s)\cap \text{Leaves}(P)$, there cannot exist $t\in P(s)$ such that $u<t$, otherwise $u\in P'$ and $u\notin \text{Leaves}(P)$. Since $u\in P(s)$ and there cannot exist $t\in P(s)$ such that $u<t$, $u\in \text{Leaves}(P(s))$.   If $u\in \text{Leaves}(P(s))\subset P(s)$, there cannot exist $t\in P$ such that $u<t$, otherwise $s<u<t$, and $t\in P(s)$. But this contradicts the maximality of $u$ in $P(s)$.  Therefore $u\in P(s)\cap \text{Leaves}(P)$.

$(viii)$ We prove the first statement by induction on $\zeta$.    The $\zeta=0$ case is clear. For the successor case, it is sufficient to prove the $\zeta=1$ case.   For this, we note that $s\in P'(t)$ if and only if $t<s$ and there exists $u\in P$ such that $s<u$ if and only if $t<s$ and there exists $u\in P(t)$ such that $s<u$ if and only if $s\in P(t)'$.   For the limit ordinal case, we note that $s\in P^\zeta(t)$ if and only if $t<s$ and $s\in P^\mu$ for all $\mu<\zeta$ if and only if for all $\mu<\zeta$, $t<s$ and $s\in P^\mu$ if and only if $s\in \cap_{\mu<\zeta}P^\mu(t)=\cap_{\mu<\zeta}(P(t))^\mu=P(t)^\zeta$.    The proof for $P[t]$ is similar, replacing $t<s$ with $t\leqslant s$.       

For the second statement, suppose that $P$ is well-founded.   For $t\in P(s)$ and an ordinal $\zeta$, $t\in P(s)^\zeta$ if and only if $t\in P^\zeta(s)$ if and only if $t\in P^\zeta$. Thus $\tau_P(t)=\max \{\zeta: t\in P^\zeta\}=\max\{\zeta: t\in P(s)^\zeta\}=\tau_{P(s)}(t)$.  

Still supposing $P$ is well-founded and $t\in P$, $t\in P^\zeta$ if and only if $\tau_P(t)\geqslant \zeta$ from the definition of $\tau_P$ and the downward closedness of $\{\mu: t\in P^\mu\}$.     Suppose $\tau_P(t) \geqslant \zeta$.    Then for any $\mu<\zeta$, $t\notin \text{Leaves}(P^\mu)$, otherwise $t\notin (P^\mu)'=P^{\mu+1}\supset P^\zeta$ and $\tau_P(t)<\zeta$.    Then any member of $P^\mu \cap P(t)$ lies in $P^\mu(t)$, whence $P^\mu(t)\neq \varnothing$.    Since this holds for any $\mu<\zeta$,  $\text{rank}(P(t))\geqslant \zeta$.  Next suppose that $\text{rank}(P(t))\geqslant \zeta$.    We claim that for any $\mu\leqslant \text{rank}(P(t))$,  $\{t\}\cup P(t)^\mu= P[t]^\mu$. We prove this by induction on $\mu$.    We prove the containments separately. We first show that $P[t]^\mu\subset \{t\}\cup P(t)^\mu$ for all $\mu$.  First, if $s\in P[t]^\mu$, then either $s=t\in \{t\}\cup P(t)^\mu$, or $t<s\in P[t]^\mu=P^\mu[t]$. In the latter case, $t<s$ and $s\in P^\mu$, whence $s\in P^\mu(t)=P(t)^\mu$.     We next prove the containment $\{t\}\cup P(t)^\mu\subset P[t]^\mu$ for all $\mu\leqslant \text{rank}(P(t))$.  It is clear that $P(t)^\mu\subset P[t]^\mu$ for all $\mu$.      We  show that $t\in P[t]^\mu$ for each $\mu\leqslant \text{rank}(P(t))$. If $\mu=\nu+1\leqslant \text{rank}(P(t))$, then $\nu<\mu\leqslant \text{rank}(P(t))$.   Therefore we may fix $s\in P(t)^\nu \subset P[t]^\nu$.     Since $t<s\in P[t]^\nu$, $t\in P[t]^{\nu+1}=P[t]^\mu$.   If $\mu\leqslant \text{rank}(P(t))$ is a limit ordinal, then for every $\nu<\mu\leqslant \text{rank}(P(t))$, $t\in P[t]^\nu$.    Intersecting over $\nu$ yields that $t\in P[t]^\mu$. This gives the claim on $\mu$.  Applying this with $\mu=\zeta$ yields that $t\in P[t]^\zeta\neq \varnothing$ and $\text{rank}(P[t])>\zeta$.   Now suppose that $\text{rank}(P[t])>\zeta$.    This implies that $P[t]^\zeta\neq \varnothing$. But since $P[t]^\zeta$ is downward closed in $P[t]$, $t\in P[t]^\zeta=P^\zeta[t]$. This yields that $t\in P^\zeta$.

$(ix)$ Suppose that $t\in \text{Leaves}(P^\zeta)$.  Then $P^\zeta(t)=\varnothing$, so $\text{rank}(P(t))\leqslant \zeta$. But for $\mu<\zeta$, $P(t)^\mu=P^\mu(t)\neq \varnothing$, otherwise the proof of $(viii)$ yields that $P^\mu[t]\subset \{t\}$ and $\varnothing =P^{\mu+1}[t]=P[t]^{\mu+1}$. In particular, $t\notin P^{\mu+1}\subset P^\zeta$, a contradiction of $t\in \text{Leaves}(P^\zeta)$.   So $P(t)^\mu\neq \varnothing$ for all $\mu<\zeta$, and $\text{rank}(P(t))=\zeta$.  Thus we have shown that if $t\in \text{Leaves}(P^\zeta)$, $\text{rank}(P(t))=\zeta$.     Now suppose that $\text{rank}(P(t))=\zeta$.  Then for every $\mu<\zeta$, $P(t)^\mu\neq \varnothing$, whence there exists $t<s\in P^\mu$.  Since $P^\mu$ is downward closed in $P$,  $t\in P^\mu$ for every $\mu<\zeta$, and $\tau_P(t)\geqslant \zeta$.   Since $P(t)^\zeta=\varnothing$, it must be the case that $t\notin P^{\zeta+1}$, otherwise there exists $t<s\in P^\zeta$, and $s\in P^\zeta(t)\neq \varnothing$.  This shows that if $\text{rank}(P(t))=\zeta$, then $\tau_P(t)=\zeta$.     Now suppose that $\tau_P(t)=\zeta$. Then $t\in P^\zeta$. Furthermore, if $\tau_P(t)=\zeta$, then $t\in \text{Leaves}(P^\zeta)$, otherwise $t\in (P^\zeta)'=P^{\zeta+1}$ and $\tau_P(t)>\zeta$.   This yields that $t\in \text{Leaves}(P^\zeta)$ if $\tau_P(t)=\zeta$.

$(x)$ We know from $(v)$ that $(P\setminus P^\beta)^\beta\subset P^\beta\setminus P^\beta=\varnothing$, so $\text{rank}(P\setminus P^\beta)\leqslant \beta$. But for any $\mu<\beta$, by well-foundedness, there exists $t\in \text{Leaves}(P^\mu)$. From this it follows that $t\notin P^\beta$, and $t\in P^\mu\setminus P^\beta$. Therefore $(P\setminus P^\beta)^\mu=P^\mu\setminus P^\beta \neq \varnothing$. Since this holds for any $\mu<\beta$, $\text{rank}(P\setminus P^\beta)\geqslant \beta$. This gives the first statement.  Now suppose that $t\in P^{\beta+\gamma}=(P^\beta)^\gamma$. Let $Q=P(t)^\beta=P^\beta(t)$.  Then $\text{rank}(Q)\geqslant \gamma$ by $(viii)$ applied to $P^\beta(t)$.    Now applying the first part of $(x)$, $\text{rank}(Q\setminus Q^\gamma)=\gamma$.  Next note  that $$P(t)\cap (P^\beta\setminus P^{\beta+\gamma})= Q\setminus Q^\gamma.$$   To see this, note that $s\in P(t)\cap (P^\beta\setminus P^{\beta+\gamma})$ if and only if $t<s$ and $s\in P^\beta$ and $s\notin P^{\beta+\gamma}$, and $s\in Q\setminus Q^\gamma$ if and only if $s\in Q=P^\beta(t)$ and $s\notin Q^\gamma=P^{\beta+\gamma}(t)$.  Indeed, if  $t<s$, $s\in P^\beta$, and $s\notin P^{\beta+\gamma}$, then the first two conditions imply that $s\in P^\beta(t)=Q$, while $t<s$ and $s\notin P^{\beta+\gamma}$ imply that $s\notin P^{\beta+\gamma}(t)=Q^\gamma$.  Thus $s\in P(t)\cap (P^\beta \setminus P^{\beta+\gamma})$ implies that $s\in Q\setminus Q^\gamma$.   Now if $s\in Q= P^\beta(t)$ and $s\notin Q^\gamma=P^{\beta+\gamma}(t)$, then the first condition implies $t<s$ and $s\in P^\beta$. The conditions $t<s$ and $s\notin P^{\beta+\gamma}(t)$ imply that $s\notin P^{\beta+\gamma}$. Therefore $s\in Q\setminus Q^\gamma$ implies that $s\in P(t)\cap (P^\beta\setminus P^{\beta+\gamma})$.  Now since $$P(t)\cap (P^\beta\setminus P^{\beta+\gamma})=Q\setminus Q^\gamma$$ and $\text{rank}(Q\setminus Q^\gamma)=\gamma$, we deduce that $$\text{rank}(P(t)\cap (P^\beta\setminus P^{\beta+\gamma}))=\gamma.$$  Next,  it follows from $(v)$ that for any ordinal $\zeta$, $$[P(t)\cap (P^\beta\setminus P^{\beta+\gamma})]^\zeta = (Q\setminus Q^\gamma)^\zeta = Q^\zeta \setminus Q^\gamma.$$   To finish $(x)$, it suffices to show that $$P^{\beta+\zeta}(t)\setminus P^{\beta+\gamma}(t)=Q^\zeta\setminus Q^\gamma=P(t)\cap (P^{\beta+\zeta}\setminus P^{\beta+\gamma}).$$   The first equality follows from $(iii)$ and $(viii)$.   For the second equality, we argue as above to deduce that $s\in Q^\zeta\setminus Q^\gamma$ if and only if $t<s$, $s\in P^{\beta+\zeta}$, and $s\notin P^{\beta+\gamma}$, if and only if $s\in P(t)\cap (P^{\beta+\zeta}\setminus P^{\beta+\gamma})$.

$(xi)$  Let $Q=R\cup \Bigl(\amalg_{t\in \text{Leaves}(R)} R_t\Bigr)$.  For $t\in \text{Leaves}(R)$, $R_t= Q(t)$, so $\text{rank}(Q(t))= \text{rank}(R_t)=\zeta$ and $t\in Q^\zeta$ by $(viii)$. This yields that $\text{Leaves}(R)\subset Q^\zeta$.  Now since $R$ is well-founded and $Q^\zeta$ is downward closed in $Q$,   the downward closure of $\text{Leaves}(R)$ in $Q$ is $R$ and is contained in $Q^\zeta$.  Therefore $$R\subset Q^\zeta \subset \bigcap_{\mu<\zeta}Q^\mu.$$   For  $s\in Q\setminus R$, there exists unique $t\in \text{Leaves}(R)$ such that $s\in R_t$. Since $R_t(s)=Q(s)$, $\tau_{R_t}(s)=\tau_Q(s)$, so that for any $\mu\leqslant \zeta$,  $s\in Q^\mu$ if and only if $s\in R_t^\mu$.   From this it follows that $$Q^\mu=R\cup \Bigl(\amalg_{t\in \text{Leaves}(R)} R_t^\mu\Bigr)$$ for all $0\leqslant \mu\leqslant \zeta$.   Therefore $Q^\zeta=R$ and for any $\eta$, $Q^{\zeta+\eta}=(Q^\zeta)^\eta=R^\eta$.

The second statement follows easily from $(iii)$ and  the last sentence of the previous paragraph.

\end{proof}

\section{Warming up}

We say a non-zero ordinal $\gamma$ is \emph{additively indecomposable} if  there do not exist $\alpha, \beta<\gamma$ such that $\alpha+\beta\geqslant \gamma$. This is known to be equivalent to the condition that $\alpha+\gamma=\gamma$ for all $\alpha<\gamma$, and to the condition that there exists an ordinal $\xi$ such that $\gamma=\omega^\xi$.   We say a non-zero ordinal $\delta$ is \emph{multiplicatively indecomposable} if there do not exist $\alpha, \beta<\gamma$ such that $\alpha\cdot \beta\geqslant \gamma$. This is known to be equivalent to the condition that $\alpha\cdot \gamma=\gamma$ for all $0<\alpha<\gamma$, and to the condition that $\gamma$ is either $1$, $2$, or $\omega^{\omega^\xi}$ for some ordinal $\xi$.

We recall that any non-zero ordinal $\xi$ has a \emph{Cantor normal form}, which is the unique representation $$\xi=\omega^{\alpha_0}\cdot n_0+\ldots +\omega^{\alpha_l}\cdot n_l,$$ where $l< \omega$,  $0<n_i< \omega$, and $\alpha_0>\ldots >\alpha_l$.  By replacing $\omega^\alpha\cdot n$ with $\omega^\alpha+\ldots +\omega^\alpha$, where the summand is repeated $n$ times,  we may uniquely represent any non-zero ordinal $\xi$ as $$\xi=\omega^{\ee_0}+\ldots +\omega^{\ee_l},$$ where $\ee_0\geqslant \ldots \geqslant \ee_l$.

Recall that $\mathfrak{R}_1$ is the class of non-zero ordinals $\xi$ such that for every tree $P$ with $\text{rank}(P)$, each $k<\omega$, and each function $f:\Lambda_1(P)\to k+1$ (equivalently, each function $f:P\to k+1)$, there exists a subtree $Q$ of $P$ with $\text{rank}(Q)=\text{rank}(P)$ such that  $f|_Q$ constant. Let us first prove that $\mathfrak{R}_1$ is a subclass of the additively indecomposable ordinals.    In showing that $\mathfrak{R}_1$ is a subclass of the additively indecomposable ordinals, we will demonstrate a relationship between sums and coloring the members of a tree.  Later, we shall see that this relationship is not coincidental. In fact, $\mathfrak{R}_1$ is precisely the class of additively indecomposable ordinals, and sums of indecomposable terms is the reason behind this.  Suppose that $0<\alpha, \beta<\gamma$ are such that $\gamma\leqslant \alpha+\beta$.   Let $P$ be a tree with $\text{rank}(P)=\gamma$.   Let $B=P^\alpha$ and let $A=P\setminus B=P\setminus P^\alpha$. Since $B^\beta=(P^\alpha)^\beta=P^{\alpha+\beta}=\varnothing$, $\text{rank}(B)\leqslant \beta<\gamma$.    By Proposition \ref{ted}$(x)$, $\text{rank}(A)=\alpha$.    Now define $f:P\to 2$ by letting $f=1_B=1-1_A$, where $1_A, 1_B$ denote the indicator functions of $A,B$, respectively.    Now if $Q\subset P$ is such that $f|_Q\equiv 0$, $Q\subset A$ and $\text{rank}(Q)\leqslant \text{rank}(A)= \alpha$.    If $Q\subset P$ is such that $f|_Q\equiv 1$, then $Q\subset B$ and $\text{rank}(Q)\leqslant \text{rank}(B)\leqslant \beta$.   This yields that $\gamma\notin \mathfrak{R}_1$, with the obvious identification of $P$ with $\Lambda_1(P)$.    Note that we have shown something which is formally stronger than $\gamma\notin \mathfrak{R}_1$. In order for a non-zero ordinal not to lie in $\mathfrak{R}_1$, there must exist some tree $P$, some $k<\omega$, and $f:P\to k+1$ such that there does not exist a subtree $Q$ of $P$ with $\text{rank}(Q)=\text{rank}(P)$ and $f|_Q$ constant.  However, what we have shown is that if $\gamma>0$ fails to be additively indecomposable, then for any tree $P$ with $\text{rank}(P)=\gamma$, there exist such $k$ and $f$, and in fact we may take $k=1$.  Generalizing the above example, suppose $\gamma=\omega^{\ee_0}+\ldots +\omega^{\ee_l}$ with $\ee_0\geqslant \ldots \geqslant \ee_l$ and $l>0$.  Let $\gamma_0=0$ and $\gamma_{i+1}=\omega^{\ee_0}+\ldots +\omega^{\ee_i}$ for $i=0, \ldots, l$. For any tree $P$ with $\text{rank}(P)=\gamma$, we may partition  $P$  into the sets $P^{\gamma_i}\setminus P^{\gamma_{i+1}}$, $i=0, \ldots, l$.  We then define $f:P\to l+1$ by letting $f|_{P^{\gamma_i}\setminus P^{\gamma_{i+1}}}\equiv i$. Any subtree $Q$ on which $f$ is constant must be a subset of $P^{\gamma_i}\setminus P^{\gamma_{i+1}}$ for some $i\leqslant l$, and therefore $$\text{rank}(Q)\leqslant \text{rank}(P^{\gamma_i}\setminus P^{\gamma_{i+1}})=\omega^{\ee_i}.$$   We shall see that, up to passing to subtrees which are compatible with the original tree $P$ in a way that we make precise later, this is essentially the only obstruction to monochromatic subtrees in the case of coloring trees.

Next we discuss coloring pairs in trees.   We  will  give an easy proof that $\mathfrak{R}_2$ is a subclass of the multiplicatively indecomposable ordinals.  The majority of the remainder of this work is concerned with the reverse inclusion. Our example here shows that an ordinal being multiplicatively decomposable provides an obstruction to monochromaticity of pairs.   Our later arguments will show that this obstruction to monochromaticity is essentially the only one.    Suppose that $0<\alpha, \beta<\gamma$ are such that $\alpha\cdot \beta\geqslant \gamma$.  Fix a tree $P$ such that $\text{rank}(P)=\gamma$. Define $f:\Lambda_2(P)\to 2$ by letting $f(s,t)=0$ if there exists an ordinal $\delta$ such that $s,t\in P^{\alpha \cdot \delta}\setminus P^{\alpha \cdot (\delta+1)}$, and $f(s,t)=1$ otherwise.   Heuristically, we view $P$ as a tree of trees, where the inner trees are those of the form $P^{\alpha \cdot \delta}\setminus P^{\alpha \cdot (\delta+1)}$, $\delta<\beta$.   We let $f(s,t)=0$ if $s,t$ lie in the same inner tree, and $f(s,t)=1$ otherwise.    We claim that if $Q\subset P$ is such that $f|_{\Lambda_2(Q)}$ is constant, then $\text{rank}(Q)\leqslant \max\{\alpha, \beta\}$.  This will yield that $\gamma\notin \mathfrak{R}_2$.  First suppose that $Q\subset P$ is such that $f|_{\Lambda_2(Q)}\equiv 0$.   We claim that $$Q=\amalg_{\delta<\beta} Q\cap (P^{\alpha\cdot \delta}\setminus P^{\alpha\cdot (\delta+1)}).$$    We first note that since $\text{rank}(P)=\gamma\leqslant \alpha \cdot \beta$, $P=\cup_{\delta<\beta} P^{\alpha\cdot \delta}\setminus P^{\alpha \cdot (\delta+1)}$ and $$Q=\bigcup_{\delta<\beta} Q\cap (P^{\alpha \cdot \delta}\setminus P^{\alpha \cdot (\delta+1)}).$$  To see that this union is an incomparable one, we note that if $\delta\neq \eta$, $s<t$, $s\in P^{\alpha \cdot \delta}\setminus P^{\alpha\cdot (\delta+1)}$, and $t\in P^{\alpha \cdot\eta}\setminus P^{\alpha \cdot (\eta+1)}$, then $f(s,t)=1$. This means either $s\notin Q$ or $t\notin Q$, and shows that $$\bigcup_{\delta<\beta} Q\cap (P^{\alpha\cdot \delta}\setminus P^{\alpha\cdot (\delta+1)}) = \amalg_{\delta<\beta} Q\cap (P^{\alpha\cdot \delta}\setminus P^{\alpha\cdot (\delta+1)}) .$$  From this and Proposition \ref{ted}$(vi)$ it follows that $$\text{rank}(Q)=\sup_{\delta<\beta} \text{rank}(Q\cap (P^{\alpha\cdot \delta}\setminus P^{\alpha\cdot (\delta+1)})) \leqslant  \sup_{\delta<\beta}\text{rank}(P^{\alpha\cdot \delta}\setminus P^{\alpha\cdot (\delta+1)})=\alpha.$$  Now suppose that $Q\subset P$ and $f|_{\Lambda_2(Q)}\equiv 1$.   We claim that for any ordinal $\delta$, $Q^\delta\subset P^{\alpha \cdot \delta}$.  We prove this by induction on $\delta$.  The $\delta=0$ and $\delta$ a limit ordinal cases are clear.   Suppose $Q^\delta \subset P^{\alpha \cdot \delta}$. If $s\in Q^{\delta+1}\setminus P^{\alpha \cdot (\delta+1)}$, then we could find some $s<t\in Q^\delta\subset P^{\alpha \cdot \delta}$.   But since $P^{\alpha\cdot \delta}$ and $P^{\alpha\cdot (\delta+1)}$ are downward closed in $P$, $s<t\in P^{\alpha \cdot \delta}$ implies that $s\in P^{\alpha \cdot \delta}$, while $P^{\alpha \cdot (\delta+1)}\not\ni s<t$ implies that $t\notin P^{\alpha \cdot (\delta+1)}$.   Therefore $s,t\in P^{\alpha \cdot \delta}\setminus P^{\alpha \cdot (\delta+1)}$, and $f(s,t)=0$. But this contradicts our assumptions on $f$ and $Q$ and completes the induction proof.  Since $\text{rank}(P)\leqslant \alpha \cdot \beta$, $$Q^\beta\subset P^{\alpha\cdot \beta}=\varnothing,$$ and $\text{rank}(Q)\leqslant \beta$. This completes the proof that $\mathfrak{R}_2$ is a subclass of the class of multiplicatively indecomposable ordinals.

In this paragraph, for an additively indecomposable ordinal $\gamma$, we define $\lambda(\gamma)$, and for a tree $P$ with $\text{rank}(P)=\gamma$, we define the \emph{separation} $\varsigma_P:\Lambda_2(P)\to \lambda(\gamma)$.      Let $\gamma$ be an additively indecomposable ordinal, which means $\gamma=\omega^\xi$ for some ordinal $\xi$.  First, if $\xi=0$ and if $P$ is a tree with $\text{rank}(P)=\gamma=\omega^0=1$,  $\Lambda_2(P)=\varnothing$. Then we define $\lambda(\gamma)=0$ and we define the separation function $\varsigma_P:\Lambda_2(P) \to \lambda(\gamma)$ in the only possible way, which is to let $\varsigma_P$ be the empty function. Now suppose $0<\xi$.    We may write $\xi=\omega^{\ee_0}+\ldots +\omega^{\ee_l}$ with $\ee_0\geqslant \ldots \geqslant \ee_l$.   Then $$\gamma=\omega^{\omega^{\ee_0}}\cdot \ldots \cdot\omega^{\omega^{\ee_l}},$$ and this representation is unique.   We define $\lambda(\gamma)=l+1$. We also define $\alpha_0=\omega^{\omega^{\ee_0}}$ and $\alpha_i=\alpha_{i-1}\cdot \omega^{\omega^{\ee_i}}$ for $i=1, \ldots, l$.  If $P$ is a tree with $\text{rank}(P)=\gamma$, we  define $\varsigma_P:\Lambda_2(P)\to l+1$ by letting $\varsigma_P(s,t)$ be the smallest $i$ such that there exists an ordinal $\delta$ with $$s,t\in P^{\alpha_i\cdot \delta}\setminus P^{\alpha_i\cdot (\delta+1)}.$$   Since $\gamma=\alpha_l$, $P=P^{\alpha_l\cdot 0}\setminus P^{\alpha_l\cdot 1}$, and such an $i\leqslant l$ exists. Let us also note that $\varsigma_P$ admits an alternative description: $$\varsigma_P(s,t)=\min\{i\leqslant l: \tau_{P, \alpha_i}(s)=\tau_{P, \alpha_i}(t)\}.$$  Here we recall that $\tau_{P, \alpha_i}(t)=\max \{\delta: t\in P^{\alpha_i\cdot \delta}\}$.  We give a third description of $\varsigma_P$: Let $\varrho_i=\omega^{\omega^{\ee_{i+1}}}\cdot \ldots \cdot \omega^{\omega^{\ee_l}}$, where $\varrho_l=1$ by convention.  Then $\alpha_i\cdot \varrho_i=\gamma$ for each $i\leqslant l$.    For each $i\leqslant l$, we partition $P$ into $(P^{\alpha_i\cdot \delta}\setminus P^{\alpha_i\cdot (\delta+1)})_{\delta<\varrho_i}$ and define the equivalence relation $\approx_i$ on $P$ by letting $s\approx_i t$ if $s$ and $t$ lie in the same member of the partition $(P^{\alpha_i\cdot \delta}\setminus P^{\alpha_i\cdot (\delta+1)})_{\delta<\varrho_i}$.   Then for $(s,t)\in \Lambda_2(P)$, we let $$\varsigma_P(s,t)=\min\{i\leqslant l: s\approx_i t\}.$$   This set is non-empty, since for $i=l$, the partition is the trivial partition of $P$ into one set.

We collect the following easy fact about the separation function, which is a consequence of left distributivity of ordinal multiplication.

\begin{proposition} Fix $0\leqslant l<m<\omega$ and $\ee_0\geqslant \ldots \geqslant \ee_m$. Let $\delta=\omega^{\omega^{\ee_0}}\cdot \ldots \cdot \omega^{\omega^{\ee_m}}$ and $\gamma=\omega^{\omega^{\ee_0}}\cdot \ldots \cdot \omega^{\omega^{\ee_l}}$.  If $P$ is a tree with rank $\delta$, then for any $\eta< \omega^{\omega^{\ee_{l+1}}}\cdot \ldots \cdot \omega^{\omega^{\ee_m}}$, $$\varsigma_{P^{\gamma\cdot \eta}\setminus P^{\gamma\cdot (\eta+1)}}= \varsigma_P|_{\Lambda_2(P^{\gamma\cdot \eta}\setminus P^{\gamma\cdot (\eta+1)})}.$$

\label{division}
\end{proposition}

\begin{proof} Fix $\eta<\omega^{\omega^{\ee_{l+1}}}\cdot \ldots \cdot \omega^{\omega^{\ee_m}}$ and let $Q=P^{\gamma\cdot \eta}\setminus P^{\gamma\cdot (\eta+1)}$, so $\text{rank}(Q)=\gamma$.   For $0\leqslant i\leqslant m$, let $\alpha_i=\omega^{\omega^{\ee_0}}\cdot \ldots \cdot \omega^{\omega^{\ee_i}}$.   Let us recall that for $(s,t)\in \Lambda_2(Q)$ and $i\leqslant l$, $$\tau_{P, \alpha_i}(s)=\max\{\beta: s\in P^{\alpha_i\cdot \beta}\},$$ $$\tau_{Q, \alpha_i}(s)=\max\{\beta: s\in Q^{\alpha_i\cdot \beta}\},$$ $$\varsigma_P(s,t)=\min\{i\leqslant m: \tau_{P, \alpha_i}(s)=\tau_{P, \alpha_i}(t)\},$$ and $$\varsigma_Q(s,t)=\min\{i\leqslant l: \tau_{Q, \alpha_i}(s)=\tau_{Q, \alpha_i}(t)\}.$$  Let us also note that  for $(s,t)\in \Lambda_2(Q)$, since $Q=P^{\gamma\cdot \eta}\setminus P^{\gamma\cdot (\eta+1)}$ and $\alpha_l=\gamma$, it follows that  $\tau_{P,\alpha_l}(s)=\eta=\tau_{P, \alpha_l}(t)$. Therefore $\varsigma_P(s,t)\leqslant l$ and $$\varsigma_P(s,t)=\min\{i\leqslant l: \tau_{P, \alpha_i}(s)=\tau_{P, \alpha_i}(t)\}.$$   

For each $0\leqslant i\leqslant l$, let $\varrho_i=\omega^{\omega^{\ee_{i+1}}}\cdot \ldots \cdot \omega^{\omega^{\ee_l}}$, with the convention that $\varrho_l=1$.   Then $\alpha_i\cdot \varrho_i=\gamma$ for each $0\leqslant i\leqslant l$.   We claim that for each $s\in Q$ and $i\leqslant l$, $$\tau_{P, \alpha_i}(s)= \varrho_i\cdot \eta +\tau_{Q, \alpha_i}(s).$$ Once we prove this, we will have that for any $(s,t)\in \Lambda_2(Q)$, the sets $\{i\leqslant l: \tau_{P,\alpha_i}(s)=\tau_{P,\alpha_i}(t)\}$ and $\{i\leqslant l: \tau_{Q,\alpha_i}(s)=\tau_{Q,\alpha_i}(t)\}$ will be equal. We will then deduce that the minima of these sets, which are $\varsigma_P(s,t)$ and $\varsigma_Q(s,t)$ respectively, are equal, finishing the proof.

We return to the proof of the claim that for $i\leqslant l$ and $s\in Q$, $\tau_{P, \alpha_i}(s)=\varrho_i\cdot \eta+\tau_{Q, \alpha_i}(s)$.  Let us note that since $Q\cap P^{\gamma\cdot (\eta+1)}=\varnothing$, for $s\in Q$ and $\xi<\varrho_i$, $s\in P^{\alpha_i\cdot (\varrho_i\cdot \eta+\xi)}$ if and only if \begin{align*} s & \in P^{\alpha_i\cdot(\varrho_i\cdot \eta+\xi)}\setminus P^{\gamma\cdot (\eta+1)} =(P^{\alpha_i\cdot \varrho_i\cdot \eta}\setminus P^{\gamma\cdot (\eta+1)})^{\alpha_i\cdot \xi}= Q^{\alpha_i\cdot \xi}. \end{align*} Applying this with $\xi=\tau_{Q, \alpha_i}(s)$ and $\xi=\tau_{Q, \alpha_i}(s)+1$, we see that $s\in P^{\alpha_i\cdot (\varrho_i\cdot \eta+\tau_{Q, \alpha_i}(s))}$ and $s\notin P^{\alpha_i\cdot (\varrho_i\cdot \eta +\tau_{Q, \alpha_i}(s)+1)}$, which means $\tau_{P, \alpha_i}(s)=\varrho_i\cdot \eta+\tau_{Q, \alpha_i}(s)$.

\end{proof}

\section{Sums and  $\mathfrak{R}_1$}

If $\gamma=\omega^{\ee_0}+\ldots +\omega^{\ee_l}$ with $\ee_0\geqslant \ldots \geqslant \ee_l$, we will think of a tree $P$ with $\text{rank}(P)=\gamma$ as being made up of $l+1$ levels $P^{\gamma_i}\setminus P^{\gamma_{i+1}}$, $i=0, \ldots, l$,  where $\gamma_0=0$ and $\gamma_{i+1}=\omega^{\ee_0}+\ldots +\omega^{\ee_i}=\gamma_i+\omega^{\ee_i}$.   It follows from Proposition \ref{ted}$(x)$ that $\text{rank}(P^{\gamma_i}\setminus P^{\gamma_{i+1}})=\omega^{\ee_i}$ for each $i\leqslant l$.  Often we will want to stabilize some function by beginning with such a tree and for each $i\leqslant l$, choosing some subset $Q_i$ of $P^{\gamma_i}\setminus P^{\gamma_{i+1}}$. Furthermore, since the rank of $P^{\gamma_i}\setminus P^{\gamma_{i+1}}$ is the additively indecomposable ordinal $\omega^{\ee_i}$, $\varsigma_{P^{\gamma_i}\setminus P^{\gamma_{i+1}}}$ is defined. We will want to choose the subsets $Q_i$ in such a way that $\text{rank}(Q_i)=\omega^{\ee_i}$ and $\varsigma_{Q_i}=\varsigma_{P^{\gamma_i}\setminus P^{\gamma_{i+1}}}|_{\Lambda_2(Q_i)}$.   That is, we choose $Q=\cup_{i=0}^l Q_i$ in such a way that $Q$ has the same rank as $P$, the levels $Q_0, \ldots, Q_l$ are the levels of $Q$, and within each level, the separation function is preserved.   Similarly, we will often want to choose from such a tree a subtree which is contained within a specified subset of the levels of the original tree, also in  a way which preserves separation within each level. Our first two lemmas tell us how to do this. 

\begin{lemma} Let $\gamma$ be a non-zero ordinal and write $\gamma=\omega^{\ee_0}+\ldots +\omega^{\ee_l}$ with $\ee_0\geqslant \ldots \geqslant \ee_l$.  Let $\gamma_0=0$ and $\gamma_{i+1}=\gamma_i+\omega^{\ee_i}=\omega^{\ee_0}+\ldots +\omega^{\ee_i}$ for each $i\leqslant l$.   Suppose that $P$ is a tree with $\text{\emph{rank}}(P)=\gamma$ and $D\subset P^{\gamma_l}$ is such that $\text{\emph{rank}}(D)=\text{\emph{rank}}(P^{\gamma_l})$ and  $\varsigma_D=\varsigma_{P^{\gamma_l}}|_{\Lambda_2(D)}$.   Suppose also that for each $t\in \text{\emph{Leaves}}(D)$, there exist $s_t\in \text{\emph{Leaves}}(P^{\gamma_l})$ and $Q_t\subset P(s_t)$ such that $t\leqslant s_t$, $\text{\emph{rank}}(Q_t)=\text{\emph{rank}}(P(s_t))$,  and for each $i<l$, $Q^{\gamma_i}_t\setminus Q^{\gamma_{i+1}}_t=Q_t\cap (P^{\gamma_i}\setminus P^{\gamma_{i+1}})$ and $\varsigma_{Q^{\gamma_i}_t\setminus Q^{\gamma_{i+1}}_t}=\varsigma_{P^{\gamma_i}\setminus P^{\gamma_{i+1}}}|_{\Lambda_2(Q^{\gamma_i}_t\setminus Q^{\gamma_{i+1}}_t)}$.   Then with $$Q=D\cup \Bigl(\amalg_{t\in \text{\emph{Leaves}(D)}} Q_t\Bigr),$$ it follows that $\text{\emph{rank}}(Q)=\gamma$ and for each $i\leqslant l$, \begin{enumerate}[(i)]\item $Q^{\gamma_i}\setminus Q^{\gamma_{i+1}}=Q\cap (P^{\gamma_i}\setminus P^{\gamma_{i+1}})$, and \item $\varsigma_{Q^{\gamma_i}\setminus Q^{\gamma_{i+1}}}=\varsigma_{P^{\gamma_i}\setminus P^{\gamma_{i+1}}}|_{\Lambda_2(Q^{\gamma_i}\setminus Q^{\gamma_{i+1}})}$. \end{enumerate}

\label{grow1}

\end{lemma}

\begin{proof} First, we note that since $\gamma=\gamma_l+\omega^{\ee_l}$, $\text{rank}(P^{\gamma_l})=\omega^{\ee_l}$ by Proposition \ref{ted}$(iii)$.    Therefore $\text{rank}(D)=\text{rank}(P^{\gamma_l})=\omega^{\ee_l}$, and $\varsigma_D$ is defined.    For each $t\in \text{Leaves}(D)$, $\text{rank}(P(s_t))=\gamma_l$ by Proposition \ref{ted}$(ix)$.   Therefore $\text{rank}(Q_t)=\text{rank}(P(s_t))=\gamma_l$ by hypothesis.     By Proposition \ref{ted}$(xi)$, $$Q^{\gamma_l}=D$$ and for each $i\leqslant l$, $$Q^{\gamma_i}= D\cup \Bigl(\amalg_{t\in \text{Leaves}(D)} Q_t^{\gamma_i}\Bigr)$$ and $$Q^{\gamma_i}\setminus Q^{\gamma_{i+1}}= \amalg_{t\in \text{Leaves}(D)} Q_t^{\gamma_i}\setminus Q_t^{\gamma_{i+1}}.$$  Now for any $t\in \text{Leaves}(D)$ and $u\in Q_t$, since $t\leqslant s_t<u$ and $s_t\in \text{Leaves}(P^{\gamma_l})$,  $\tau_P(u)<\tau_P(s_t)=\gamma_l$ by Proposition \ref{ted}(iv) and (ix).    Therefore $u\notin P^{\gamma_l}$, and since $t\in \text{Leaves}(D)$ and $u\in Q_t$ were arbitrary, $Q\cap P^{\gamma_l}\subset D$.  Since $\gamma_{l+1}=\gamma$, $Q^{\gamma_{l+1}}\subset P^{\gamma}=\varnothing$.  From this and the hypothesis that $D\subset P^{\gamma_l}$, we deduce that \begin{align*} Q^{\gamma_l}\setminus Q^{\gamma_{l+1}} & = Q^{\gamma_l} =D = D\cap(P^{\gamma_l}\setminus P^{\gamma_{l+1}})= D\cap P^{\gamma_l}= Q\cap P^{\gamma_l}= Q\cap (P^{\gamma_l}\setminus P^{\gamma_{l+1}}).\end{align*}    Furthermore, \begin{align*} \varsigma_{Q^{\gamma_l}\setminus Q^{\gamma_{l+1}}} & = \varsigma_D = \varsigma_{P^{\gamma_l}\setminus P^{\gamma_{l+1}}}|_{\Lambda_2(D)} = \varsigma_{P^{\gamma_l}\setminus P^{\gamma_{l+1}}}|_{\Lambda_2(Q^{\gamma_l}\setminus Q^{\gamma_{l+1}})}.\end{align*}  This yields items $(i)$ and $(ii)$ in the case $i=l$.

Now fix $i<l$.   Using the previous paragraph and the hypotheses,   \begin{align*} Q^{\gamma_i}\setminus Q^{\gamma_{i+1}} & = \amalg_{t\in \text{Leaves}(D)} Q_t^{\gamma_i}\setminus Q_t^{\gamma_{i+1}}= \amalg_{t\in \text{Leaves}(D)} Q_t\cap (P^{\gamma_i}\setminus P^{\gamma_{i+1}}) \\ & = (P^{\gamma_i}\setminus P^{\gamma_{i+1}}) \cap \amalg_{t\in \text{Leaves}(D)} Q_t.\end{align*} Since $$D\cap (P^{\gamma_i}\setminus P^{\gamma_{i+1}})\subset P^{\gamma_l}\cap (P^{\gamma_i}\setminus P^{\gamma_{i+1}})=\varnothing,$$ \begin{align*} Q^{\gamma_i}\setminus Q^{\gamma_{i+1}} & = (P^{\gamma_i}\setminus P^{\gamma_{i+1}})\cap \amalg_{t\in \text{Leaves}(D)} Q_t \\ & = (P^{\gamma_i}\setminus P^{\gamma_{i+1}})\cap\Bigl(D\cup \bigl(\amalg_{t\in \text{Leaves}(D)}Q_t\bigr)\Bigr)= Q\cap (P^{\gamma_i}\setminus P^{\gamma_{i+1}}).   \end{align*} This completes $(i)$.  Next, since $$Q^{\gamma_i}\setminus Q^{\gamma_{i+1}}=\amalg_{t\in \text{Leaves}(D)} Q_t^{\gamma_i}\setminus Q_t^{\gamma_{i+1}},$$ it follows from Proposition \ref{ted}$(vi)$ that for any ordinal $\zeta$, $$(Q^{\gamma_i}\setminus Q^{\gamma_{i+1}})^\zeta = \amalg_{t\in \text{Leaves}(D)} (Q_t^{\gamma_i}\setminus Q^{\gamma_{i+1}}_t)^\zeta.$$  From this it follows that $\tau_{Q_t^{\gamma_i}\setminus Q_t^{\gamma_{i+1}}}=\tau_{Q^{\gamma_i}\setminus Q^{\gamma_{i+1}}}|_{Q_t}$, and $$\varsigma_{Q_t^{\gamma_i}\setminus Q_t^{\gamma_{i+1}}}=\varsigma_{Q^{\gamma_i}\setminus Q^{\gamma_{i+1}}}|_{\Lambda_2(Q_t^{\gamma_i}\setminus Q_t^{\gamma_{i+1}})}.$$   Therefore by our assumptions on the sets $Q_t$, for any $(u,v)\in \Lambda_2(Q^{\gamma_i}\setminus Q^{\gamma_{i+1}})$, if $t\in \text{Leaves}(D)$ is such that $t<u$, $$\varsigma_{Q^{\gamma_i}\setminus Q^{\gamma_{i+1}}}(u,v)=\varsigma_{Q_t^{\gamma_i}\setminus Q_t^{\gamma_{i+1}}}(u,v)=\varsigma_{P^{\gamma_i}\setminus P^{\gamma_{i+1}}}(u,v).$$  This concludes $(ii)$.

\end{proof}

\begin{lemma} Let $\gamma$ be a non-zero ordinal and write $\gamma=\omega^{\ee_0}+\ldots +\omega^{\ee_l}$ with $\ee_0\geqslant \ldots \geqslant \ee_l$.    Let $\gamma_0=0$ and $\gamma_{i+1}=\gamma_i+\omega^{\ee_i}$ for $i=0, \ldots, l$.    Fix a non-empty set $A\subset l+1$ and write $A=\{a(0), \ldots, a(p)\}$ with $a(0)<\ldots <a(p)$. Let $\beta=\omega^{\ee_{a(0)}}+\ldots +\omega^{\ee_{a(p)}}$, $\beta_0=0$, and $\beta_{i+1}=\beta_i+\omega^{\ee_{a(i)}}$ for $i=0, \ldots, p$.     Define $$Q=\bigcup_{i=0}^p P^{\gamma_{a(i)}}\setminus P^{\gamma_{a(i)+1}}.$$  Then $\text{\emph{rank}}(Q)=\beta$ and for each $i\leqslant p$, the following hold: \begin{enumerate}[(i)]\item $Q^{\beta_i}\setminus Q^{\beta_{i+1}}= P^{\gamma_{a(i)}}\setminus P^{\gamma_{a(i)+1}}$.   \item $\varsigma_{Q^{\beta_i}\setminus Q^{\beta_{i+1}}}=\varsigma_{P^{\gamma_{a(i)}}\setminus P^{\gamma_{a(i)+1}}}|_{\Lambda_2(Q^{\beta_i}\setminus Q^{\beta_{i+1}})}$. \end{enumerate}

 %Define $Q_p=P^{\gamma_{a(p)}\setminus \gamma_{a(p)+1}}$.  Assuming that for some $i<p$, $Q_p, \ldots, Q_{i+1}$ have been defined, let $$Q_i=\amalg_{t\in \text{\emph{Leaves}}(Q_{i+1})} P(t)\cap (P^{\gamma_{a(i)}}\setminus P^{\gamma_{a(i)+1}}).$$    Let $Q=\cup_{i\leqslant p} Q_i$.    

\label{grow2}
\end{lemma}

\begin{proof} If $p=0$, items $(i)$ and $(ii)$ are trivial, while it follows from Proposition \ref{ted}$(x)$ that  $\text{rank}(Q)=\text{rank}(Q_p)=\omega^{\ee_{a(0)}}=\beta$. For the remainder of the proof, assume $0<p$.  For convenience, for each $i\leqslant p$, let $Q_i=P^{\gamma_{a(i)}}\setminus P^{\gamma_{a(i)+1}}$.  Note that for $i<p$, $\text{Leaves}(Q_{i+1})=\text{Leaves}(\cup_{j=i+1}^p Q_j)$. This is tautological if $i+1=p$, and otherwise if $i+1<m\leqslant p$ and if $t_0\in \text{Leaves}(Q_m)$, $$\varnothing\neq P(t_0)\cap (P^{\gamma_{a(m-1)}}\setminus P^{\gamma_{a(m-1)+1}})\subset \Bigl(\cup_{j=i+1}^p Q_j\Bigr)(t_0),$$ so $t_0\notin \text{Leaves}(\cup_{j=i+1}^p Q_j)$.

Now we may use Proposition \ref{ted} parts $(v)$, $(x)$, and $(xi)$, we find that for each $i\leqslant p$ and $\zeta\leqslant \omega^{a(i)}$, $$Q^{\beta_i+\zeta}= \Bigl(\bigcup_{j=i+1}^p Q_j\Bigr)\cup Q_i.$$  Here, if $i=p$, the union $\cup_{j=i+1}^p Q_j$ is taken to be the empty set by convention.  This is done by induction  on the  appropriate set of $(i, \zeta)$, ordered lexicographically.   The base case $i=0$, $\zeta=0$ case is clear.  If we have this claim for some $i$ and $\zeta=0$, we use  Proposition \ref{ted}$(xi)$ with $R=\bigcup_{j=i+1}^p Q_j$ and $R_t=P(t)\cap (P^{\gamma_{a(i)}}\setminus P^{\gamma_{a(i)+1}})$ for each $t\in \text{Leaves}(Q_{i+1})$, noting that $$Q_i= \amalg_{t\in \text{Leaves}(Q_{i+1})} R_t$$   to obtain the claim for this $i$ and all $\zeta\leqslant \omega^{a(i)}$.  Now if for some $i<p$, we have the claim for all $\zeta\leqslant \omega^{a(i)}$, we deduce the $i+1$, $\zeta=0$ case of the claim by noting that  $$Q^{\beta_{i+1}+0} =Q^{\beta_{i+1}}=Q^{\beta_i+\omega^{a(i)}} = \Bigl(\bigcup_{j=i+1}^p Q_j\Bigr)\cup Q_i^{\omega^{a(i)}}= \bigcup_{j=i+1}^p Q_j= \Bigl(\bigcup_{j=i+2}^p Q_j\Bigr)\cup Q_{i+1}.$$  Using the $i=p$, $\zeta=0$ case of the claim, we deduce that  $$Q^{\beta_p}= Q^{\beta_{p-1}+\omega^{a(p-1)}}=Q_p=P^{\gamma_{a(p)}}\setminus P^{\gamma_{a(p)+1}}.$$ Since $\gamma_{a(p)+1}=\gamma_{a(p)}+\omega^{a(p)}$, $$\omega^{a(p)}=\text{rank}(P^{\gamma_{a(p)}}\setminus P^{\gamma_{a(p)+1}})=\text{rank}(Q^{\beta_p}), $$ an appeal to Proposition \ref{ted}$(iii)$ yields that $\text{rank}(Q)= \beta_p+\omega^{a(p)}=\beta_{p+1}=\beta$.

Now using the claim from the previous paragraph, we deduce that for each $i\leqslant p$, $$Q^{\beta_i}\setminus Q^{\beta_{i+1}}= \Biggl(\Bigl(\cup_{j=i+1}Q_j\Bigr)\cup Q_i\Biggr) \setminus \Biggl(\Bigl(\cup_{j=i+2}Q_j\Bigr)\cup Q_{i+1}\Biggr) = Q_i =P^{\gamma_{a(i)}}\setminus P^{\gamma_{a(i)+1}},$$ which is $(i)$.    For $(ii)$, we note that $(i)$ implies that $\tau_{Q^{\beta_i}\setminus Q^{\beta_{i+1}}}= \tau_{P^{\gamma_{a(i)}}\setminus P^{\gamma_{a(i)+1}}}$, which yields $(ii)$.

\end{proof}

\begin{rem}\upshape We will typically use Lemmas \ref{grow1} and \ref{grow2} in the specific case in which $\gamma=\omega^\ee\cdot(l+1)$ for some $\ee\in \textbf{Ord}$ and some $l<\omega$. In this case, $\gamma_i=\omega^\ee\cdot i$.

\end{rem}

Our next result yields that functions which depend on leaves may, after passing to a subtree which preserves the positions of its members in the derivatives, be made to be independent of the leaves.

\begin{lemma} For a well-founded tree, $P$ and  $0<n<\omega$,  let $$E_n(P)=\{(t_0,\ldots, t_{n-1}, t_n): t_0<\ldots <t_{n-1}\leqslant t_n, t_n\in \text{\emph{Leaves}}(P)\}.$$  Let $E_0(P)=\text{\emph{Leaves}}(P)$ and let $\Lambda_0(P)=\{\varnothing\}$.         Then for any $k<\omega$ and any function $f:E_n(P)\to k+1$, there exist $Q\subset P$ such that  $\text{\emph{rank}}(Q)=\text{\emph{rank}}(P)$ and $\tau_Q=\tau_P|_Q$ and a function $F:\Lambda_n(Q)\to k+1$ such that $$F(t_0, \ldots, t_{n-1})=f(t_0, \ldots, t_{n-1}, t_n)$$ for any $(t_0, \ldots, t_{n-1}, t_n)\in E_n(Q)$. Here, for the $n=0$ case,  we obey the convention that if $t_n\in \text{\emph{Leaves}}(Q)$, $(t_0, \ldots, t_{n-1})$ refers to $\varnothing\in \Lambda_0(Q)$.  Thus the conclusion in the $n=0$ case is simply that $f|_{\text{\emph{Leaves}}(Q)}\equiv F(\varnothing)$.

\label{tek}

\end{lemma}

\begin{proof} First note that if $Q\subset P$ satisfies $\tau_Q=\tau_P|_Q$, then for any $t\in Q$, $t\in \text{Leaves}(Q)$ if and only if $\tau_Q(t)=0$ if and only if $\tau_P(t)=0$ if and only if $t\in \text{Leaves}(P)$. From this it follows that $\text{Leaves}(Q)=Q\cap \text{Leaves}(P)$ if $\tau_Q=\tau_P|_Q$. Therefore $E_n(Q)\subset E_n(P)$ for such $Q$, and $f$ is defined on $E_n(Q)$.

For $n<\omega$ and $\xi\in \textbf{Ord}$, let  $T(\xi,n)$ be the following statement: For any $k<\omega$, any tree $P$ with $\text{rank}(P)=\xi$,  and any function $f:E_n(P)\to k+1$, there exist $Q\subset P$ such that  $\text{rank}(Q)=\text{rank}(P)$ and $\tau_Q=\tau_P|_Q$ and a function $F:\Lambda_n(Q)\to k+1$ such that $$F(t_0, \ldots, t_{n-1})=f(t_0, \ldots, t_{n-1}, t_n)$$ for any $(t_0, \ldots, t_{n-1}, t_n)\in E_n(Q)$. 

The lemma is equivalent to the assertion that $T(\xi,n)$ holds for all $(\xi,n)\in  \textbf{Ord}\times \omega$.  We prove $T(\xi,n)$ by induction on $(\xi,n)\in  \textbf{Ord}\times \omega$, ordered lexicographically.  To that end, fix some $(\xi,n)\in \textbf{Ord}\times \omega$ and assume that $T(\xi,m)$ holds for all $m<n$ and that $T(\zeta, m)$ holds for all $\zeta<\xi$ and $m<\omega$.

Case $1$, $\xi=0$: This is vacuous.

Case $2$, $\xi=1$: Note that if $P$ is a tree with $\text{rank}(P)=1$,  $P$ consists only of incomparable leaves. Let $Q=\{t\}\subset P$ be arbitrary. Then $\text{rank}(Q)=1=\text{rank}(P)$ and $\tau_Q=\tau_P|_Q\equiv 0$.  If $n=0$, let $F(\varnothing)=f(t)$.  Since $E_0(Q)=\text{Leaves}(Q)=\{t\}$, we reach the desired conclusion if $n=0$. If $n=1$, let $F(t)=f(t,t)$. Since $\Lambda_1(Q)=\{(t)\}$ and $E_1(Q)=\{(t,t)\}$, this defines $F$, and it is easy to see that the conclusion is satisfied by this choice of $F$.   If $n>1$, $E_n(Q)=\Lambda_n(Q)=\varnothing$, so the conclusion on $F$ is vacuous in this case.  Thus $T(1,n)$ holds.

Case $3$, $\xi$ is a limit ordinal: Assume $P$ with $\text{rank}(P)=\xi$, $k<\omega$, $f:E_n(P)\to k+1$ are given.    Write $P=\amalg_{t\in \text{Roots}(P)}P[t]$. Note that $E_n(P)$, $\Lambda_n(P)$ are equal to the disjoint unions $\cup_{t\in \text{Roots}(P)} E_n(P[t])$, $\cup_{t\in \text{Roots}(P)} \Lambda_n(P[t])$, respectively. Furthermore, $\tau_{P[t]}=\tau_P|_{P[t]}$ and $\text{rank}(P[t])<\xi$ for each $t\in \text{Roots}(P)$.  To see that $\text{rank}(P[t])<\xi$ for each $t\in \text{Roots}(P)$, we note that the alternative is that $t\in P^\zeta[t]$ for every $\zeta<\xi$, in which case $t\in P^\xi=\varnothing$.    For each $t\in \text{Roots}(P)$, apply $T(\text{rank}(P[t]), n)$  to $f|_{E_n(P[t])}$ to obtain $Q_t\subset P[t]$ and $F_t:\Lambda_n(Q_t)\to k+1$ satisfying the conclusions.   In the case $n=0$, for each $i\leqslant k$, let $$M_i=\{t\in \text{Roots}(P): F_t(\varnothing)=i\}.$$   Note that since \begin{align*} \xi & =\sup \Bigl\{\text{rank}(P[t]):t\in \cup_{i\leqslant k}M_i\Bigr\}=\max_{i\leqslant k} \sup\{\text{rank}(P[t]): t\in M_i\} \\ & = \max_{i\leqslant k} \sup \{\text{rank}(Q_t): t\in M_i\},\end{align*} there exists $i\leqslant k$ such that $\sup\{\text{rank}(Q_t): t\in M_i\}=\xi$.  Then let $F(\varnothing)=i$ and let $Q=\amalg_{t\in M_i}Q_t$.  Then by Proposition \ref{ted}$(vi)$, $$\text{rank}(Q)=\sup \{\text{rank}(Q_t): t\in M_i\}=\xi=\text{rank}(P).$$    For each $t\in M_i$, since $Q$ is defined by an incomparable union, $$\tau_Q|_{Q_t}=\tau_{Q_t}=\tau_{P[t]}|_{Q_t}=\tau_P|_{Q_t}.$$  Since this holds for any $t\in M_i$, $\tau_Q=\tau_P|_Q$.    For $s\in \text{Leaves}(Q)$, let $t\in M_i$ be such that $s\in \text{Leaves}(Q_t)$.    Then by our choice of $i$,  $$F(\varnothing)=i=F_t(\varnothing) =f(s).$$    This finishes Case $3$ if $n=0$.     If  $n>0$, let $Q=\amalg_{t\in \text{Roots}(P)}Q_t$ and define $F$ on $Q$ by letting $F|_{Q_t}=F_t$.     That $\text{rank}(Q)=\xi$ and $\tau_Q=\tau_P|_Q$ follow as in the $n=0$ case.     Note that $E_n(Q)$ and $\Lambda_n(Q)$ are the disjoint unions $\cup_{t\in \text{Roots}(P)} E_n(Q_t)$ and $\cup_{t\in \text{Roots}(P)} \Lambda_n(Q_t)$, respectively. Fix $(t_0, \ldots, t_{n-1}, t_n)\in E_n(Q)$ and note that there must exist some $t\in \text{Roots}(P)$ such that $(t_0, \ldots, t_{n-1}, t_n)\in E_n(Q_t)$, so $$F(t_0, \ldots, t_{n-1})=F_t(t_0, \ldots, t_{n-1})=f(t_0, \ldots, t_{n-1}, t_n).$$

Case $4$, $\xi=\zeta+1$, $\zeta>0$: Fix $P$ with $\text{rank}(P)=\zeta+1$, $k<\omega$, and $f:E_n(P)\to k+1$.   Note that since $\text{rank}(P)=\zeta+1$, $\text{rank}(P^\zeta)=1$ and $P^\zeta=\text{Leaves}(P^\zeta)$.    Fix $t\in P^\zeta=\text{Leaves}(P^\zeta)$ and note that $\tau_P(t)=\zeta$, whence $\text{rank}(P(t))=\zeta$ by Proposition \ref{ted}$(ix)$.       Note that $\tau_{P(t)}=\tau_P|_{P(t)}$ by Proposition \ref{ted}$(viii)$.     Apply $T(\zeta, n)$ to $P(t)$ and the restriction of $f$ to $E_n(P(t))$ to find $S\subset P(t)$ such that $\text{rank}(S)=\text{rank}(P(t))=\zeta$ and $\tau_S=\tau_{P(t)}|_S=\tau_P|_S$ and a function $F_0:\Lambda_n(S)\to k+1$ satisfying the conclusions of the lemma.   If $n=0$, we let $Q=\{t\}\cup S$, so $Q(t)=S$.  In this case, $Q^\zeta=\{t\}$, so $\text{rank}(Q)=\zeta+1=\xi$.  Moreover, $\text{Leaves}(S)=\text{Leaves}(Q)$, so $f|_{\text{Leaves}(Q)}=f|_{\text{Leaves}(S)}\equiv F_0(\varnothing)$.     Also, since $Q(t)=S$ has rank $\zeta$, $\tau_Q(t)=\zeta=\tau_P(t)$.  By Proposition \ref{ted}$(viii)$, $$\tau_Q|_{Q(t)}=\tau_{Q(t)}=\tau_S=\tau_P|_S,$$ whence $\tau_Q=\tau_P|_Q$.    This concludes the $n=0$ case.      Suppose now that  $n>0$ and define $g:E_{n-1}(S)\to k+1$ by letting $$g(t_0, \ldots, t_{n-2}, t_{n-1})=f(t, t_0, \ldots, t_{n-1}).$$  This is well-defined, since $\text{Leaves}(S)\subset \text{Leaves}(P)$ implies that $E_{n-1}(S)\subset E_{n-1}(P)$.   By $T(\zeta, n-1)$, there exist a subtree $T$ of $S$ with $\text{rank}(T)=\text{rank}(S)=\zeta$,  $\tau_T=\tau_S|_T=\tau_P|_T$, and a function $G:\Lambda_{n-1}(T)\to k+1$ such that $$G(t_0, \ldots, t_{n-2})=g(t_0, \ldots, t_{n-2}, t_{n-1})=f(t, t_0, \ldots, t_{n-1})$$ for every $(t_0, \ldots, t_{n-1})\in E_{n-1}(T)$.   Let $Q=\{t\}\cup T$.  Define $F:\Lambda_n(Q)\to k+1$ by letting $F(t_0, \ldots, t_{n-1})=F_0(t_0, \ldots, t_{n-1})$ if $t<t_0$, in which case $(t_0, \ldots, t_{n-1})\in \Lambda_n(T)\subset  \Lambda_n(S)$, and $F(t_0,  \ldots, t_{n-1})=G(t_1, \ldots, t_{n-1})$ if $t_0=t$, in which case $(t_1, \ldots, t_{n-1})\in \Lambda_{n-1}(T)$.   We note that since $Q(t)=T$, $Q^\zeta=\{t\}$ and $\text{rank}(Q)=\zeta+1=\xi$.     Furthermore, $\tau_Q(t)=\zeta=\tau_P(t)$ and for $t<s\in Q$, $$\tau_Q(s)=\tau_T(s)=\tau_P(s),$$ so $\tau_Q=\tau_P|_Q$.    Finally, for $(t_0, \ldots, t_{n-1}, t_n)\in E_n(Q)$, either $t_0=t$, in which case $$f(t_0, \ldots, t_n)=g(t_1, \ldots, t_n)=G(t_1, \ldots, t_{n-1})=F(t, t_1, \ldots, t_{n-1}),$$ or $t<t_0$, in which case $$f(t_0, \ldots, t_n)= F_0(t_0, \ldots, t_{n-1})=F(t_0, \ldots, t_{n-1}),$$ since $(t_0, \ldots, t_n)\in E_n(T)\subset E_n(S)$.

\end{proof}

\begin{rem}\upshape We will use the $n=0$ and $n=1$ cases of Lemma \ref{tek}. Note that $E_1(P)=\Pi(P)$.  We will often use the $n=0$ case of Lemma \ref{tek} with set rather than function notation. That is, if $M_0, \ldots, M_k\subset \text{Leaves}(P)$ are such that $\cup_{i\leqslant k}M_i=\text{Leaves}(P)$, then there exist $i\leqslant k$ and  $Q\subset P$ with $\text{rank}(Q)=\text{rank}(P)$ and $\tau_Q=\tau_P|_Q$ such that $\text{Leaves}(Q)=Q\cap M_i$.    To see this, we apply the $n=0$ case of Lemma \ref{tek} to the function $f:\text{Leaves}(P)\to k+1$ given by $f(t)=\min\{i\leqslant k: t\in M_i\}$ to obtain $Q\subset P$ with $\text{rank}(Q)=\text{rank}(P)$, $\tau_Q=\tau_P|_Q$, and $F:\Lambda_0(Q)=\{\varnothing\}\to k+1$. We then let $i=F(\varnothing)$.  Note that for each $t\in \text{Leaves}(Q)$, by definition of $f$ and the properties of $F$, $f(t)=i$, so $t\in M_i$.  Thus $\text{Leaves}(Q)\subset Q\cap M_i$. Since $M_i\subset \text{Leaves}(P)$ and $\tau_Q=\tau_P|_Q$, $Q\cap M_i\subset \text{Leaves}(P)\cap Q= \text{Leaves}(Q)$, so $\text{Leaves}(Q)= Q\cap M_i$. 

\label{dory}

\end{rem}

Some results in this section will use induction to prove a result on trees having rank equal to an additively indecomposable ordinal. Thus in order to apply an inductive hypothesis, we need to find within a given tree subtrees whose ranks are additively indecomposable. We also wish to choose these subtrees in  a way which preserves separation.      Our next two lemmas describe how to do this.

\begin{lemma} Let $\gamma,\ee$ be ordinals such that $\gamma$ is additively indecomposable and either $\gamma=1$ or $\gamma=\omega^{\omega^{\ee_0}}\cdot \ldots \cdot \omega^{\omega^{\ee_l}}$ and $\ee_0\geqslant \ldots \geqslant \ee_l\geqslant \ee$.    Suppose $P$ is a tree with $\text{\emph{rank}}(P)=\gamma\cdot \omega^{\omega^{\ee}}$. Then there exist a subset $R$ of $\text{\emph{Roots}}(P)$ and collections $\{\eta_t:t\in R\}\subset \omega^{\omega^\ee}$ and $\{s_t: t\in R\}\subset P$ such that $\sup_{t\in R}\eta_t=\omega^{\omega^\ee}$,  $s_t\in \text{\emph{Leaves}}( P[t]^{\gamma\cdot \eta_t})$ for each $t\in R$, and  \begin{enumerate}[(i)]\item if $\ee=0$, $\eta_t=n_t+1$ for some $n_t<\omega$, \item if $\ee=\delta+1$, $\eta_t=(\omega^{\omega^\delta})^{n_t}$ for some $n_t<\omega$, and \item if $\ee$ is a limit ordinal, $\eta_t=\omega^{\omega^{\ee_t}}$ for some $\ee_t<\ee$. \end{enumerate}

\label{grow3}
\end{lemma}

\begin{proof} In case $(i)$, note that $\omega^{\omega^\ee}=\omega$.   Let $R=\{t\in \text{Roots}(P): \text{rank}(P[t])>\gamma\}$ and for $t\in R$, let $n_t=\max\{n<\omega: \text{rank}(P[t])>\gamma\cdot (n+1)\}$ and let $\eta_t=n_t+1$. For $t\in R$, the set in the definition of $\eta_t$ is non-empty by the definition of $R$. Note that since $\text{rank}(P[t]) \leqslant \text{rank}(P)=\gamma\cdot \omega$, the set in the definition of $\eta_t$ has a maximum, otherwise $t\in P[t]^{\gamma\cdot \omega}=P[t]^{\text{rank}(P)}=\varnothing$.    Let $S=\text{Roots}(P)\setminus R$.   Since $\text{rank}(P[t])>\gamma\cdot \eta_t$, $P[t]^{\gamma\cdot \eta_t}$ is a non-empty, well-founded tree.  Therefore it has at least one leaf, $s_t$.    It remains to show that $\sup_{t\in R}\eta_t=\omega$. If it were not so, there would be a finite bound $\sup_{t\in R}\eta_t=n<\omega$.   Then \begin{align*} P^{\gamma \cdot (n+1)} & = \Bigl(\amalg_{t\in S}P[t]^{\gamma \cdot (n+1)}\Bigr)\amalg \Bigl(\amalg_{t\in R} P[t]^{\gamma \cdot (n+1)}\Bigr) =\varnothing. \end{align*} The first set is empty because $\text{rank}(P[t])\leqslant \gamma$ for each $t\in S$, while the second set is empty because $\text{rank}(P[t])\leqslant \gamma\cdot (\eta_t+1) \leqslant \gamma\cdot (n+1)$ for each $t\in R$.     But this implies that $\text{rank}(P)\leqslant \gamma\cdot (n+1)<\gamma\cdot \omega$, which contradicts our original hypothesis.

Cases $(ii)$ and $(iii)$ are similar to case $(i)$.   In case $(ii)$, we note that $\gamma\cdot \omega^{\omega^{\delta+1}}=\sup_{n<\omega} \gamma \cdot (\omega^{\omega^\delta})^n$. We let $R=\{t\in \text{Roots}(P): \text{rank}(P[t])>\gamma\}$, $S=\text{Roots}(P)\setminus R$,  $n_t=\max \{n<\omega: \text{rank}(P[t])>\gamma\cdot (\omega^{\omega^\delta})^n\}$, and $\eta_t=(\omega^{\omega^\delta})^{n_t}$. For $t\in R$, the set in the definition of $n_t$ is non-empty by the definition of $R$.     Since $\text{rank}(P[t])\leqslant \text{rank}(P)=\gamma\cdot (\omega^{\omega^\delta})^\omega$, the set in the definition of $n_t$ has a maximum.  We choose a leaf $s_t\in P[t]^{\gamma\cdot \eta_t}$ as before.  As in the previous paragraph, if $\sup_{t\in R} \eta_t<\omega^{\omega^\ee}$ (equivalently, if $\sup_{t\in R}n_t<\omega$), then we could find $n<\omega$ such that $\sup_{t\in R} \eta_t< (\omega^{\omega^\delta})^n$. As above, we could deduce that $P^{\gamma\cdot (\omega^{\omega^\delta})^{n+1}}=\varnothing$ and $\text{rank}(P)\leqslant \gamma\cdot (\omega^{\omega^\delta})^{n+1}<\gamma\cdot \omega^{\omega^\ee}$, which is not true. Therefore $\sup_{t\in R}\eta_t=\omega^{\omega^\ee}$.    

For case $(iii)$, we let $R=\{t\in \text{Roots}(P): \text{rank}(P[t])>\gamma\cdot \omega\}$,  $\ee_t=\max\{\eta: \text{rank}(P[t])>\gamma\cdot \omega^{\omega^\eta}\}$, and $\eta_t=\omega^{\omega^{\ee_t}}$. Note that for $t\in R$, it follows from the definition of $R$ that  the set in the definition of $\ee_t$ is non-empty.   We must justify that this set has a maximum. To do that, we first note that it has an upper bound of $\ee$, since $\text{rank}(P[t])\leqslant \text{rank}(P)=\gamma\cdot \omega^{\omega^\ee}$.    To see that this supremum is a maximum, we note that this is trivial if $\ee_t$ is zero or a successor, so assume $\ee_t$ is a limit ordinal. Note that for any ordinal $\beta$, $P[t]^\beta$ is either empty or contains $t$, whence $t\in P[t]^{\gamma\cdot \omega^{\omega^\eta}}$ for each $\eta<\ee_t$.   Then since $\{\gamma\cdot \omega^{\omega^\eta}: \eta<\ee_t\}$ is cofinal in $\gamma\cdot \omega^{\omega^{\ee_t}}$ and since $\gamma\cdot \omega^{\omega^{\ee_t}}$ is a limit ordinal,  by Proposition \ref{ted}$(i)$, $$t\in \bigcap_{\eta<\ee_t} P[t]^{\gamma\cdot \omega^{\omega^\eta}}=P[t]^{\gamma\cdot \omega^{\omega^{\ee_t}}}.$$  Thus $\text{rank}(P[t])>\gamma\cdot \omega^{\omega^{\ee_t}}$, and the supremum is a maximum.   This argument also shows that $\ee_t<\ee$, otherwise $t\in P[t]^{\gamma\cdot \omega^{\omega^\ee}}=\varnothing$. The choice of $s_t\in \text{Leaves}(P[t]^{\gamma\cdot \ee_t})$ is as in the previous cases.  We note that if $\sup_{t\in R}\ee_t=\eta<\ee$, then $$\text{rank}(P) \leqslant \max\Bigl\{\sup_{t\in \text{Roots}(P)\setminus R} \text{rank}(P[t]), \sup_{t\in R}\text{rank}(P[t])\Bigr\}\leqslant \max\{\gamma \cdot \omega, \gamma\cdot \omega^{\omega^\eta}\}<\gamma\cdot \omega^{\omega^\ee},$$ a contradiction.    Thus $\sup_{t\in R}\ee_t=\ee$ and $\sup_{t\in R}\eta_t=\sup_{t\in R} \omega^{\omega^{\ee_t}}= \omega^{\omega^\ee}$.

\end{proof}

\begin{lemma} Suppose $\gamma, \ee$ are ordinals such that $\gamma$ is additively indecomposable and either $\gamma=1$ or $\gamma=\omega^{\omega^{\ee_0}}\cdot \ldots \cdot \omega^{\omega^{\ee_l}}$ with $\ee_0\geqslant \ldots \geqslant \ee_l\geqslant \ee$.     Suppose  $P$ is a tree with $\text{\emph{rank}}(P)=\gamma\cdot \omega^{\omega^{\ee}}$ and $R\subset\text{\emph{Roots}}(P)$ is such that for each $t\in R$, we have an ordinal $\eta_t<\omega^{\omega^\ee}$ and $s_t\in \text{\emph{Leaves}}(P[t])^{\gamma\cdot \eta_t}$ such that $\sup_{t\in R}\eta_t=\omega^{\omega^\ee}$.   %Let us also suppose that \begin{enumerate}[(a)]\item if $\ee=0$, $\eta_t=n_t+1$ for some $n_t<\omega$, \item if $\ee=\delta+1$, $\eta_t=(\omega^{\omega^\delta})^{n_t}$ for some $n_t<\omega$, and \item if $\ee$ is a limit ordinal, $\eta_t=\omega^{\omega^{\ee_t}}$ for some $\ee_t<\ee$. \end{enumerate}

\begin{enumerate}[(i)]\item If $\ee>0$, suppose also that we have a subset $M$ of $R$ such that $\sup_{t\in M}\eta_t=\omega^{\omega^\ee}$ and for each $t\in M$,  $\eta_t$ is an additively indecomposable ordinal and we have a subtree $Q_t$ of $P(s_t)$ with $\text{\emph{rank}}(Q_t)=\text{\emph{rank}}(P(s_t))=\gamma\cdot \eta_t$ such that for each $\eta<\eta_t$, $$Q_t^{\gamma\cdot \eta}\setminus Q_t^{\gamma \cdot (\eta+1)}=Q_t\cap (P^{\gamma\cdot \eta}\setminus P^{\gamma \cdot (\eta+1)})$$ and for each $(u,v)\in \Lambda_2(Q_t^{\gamma \cdot \eta}\setminus Q_t^{\gamma\cdot (\eta+1)})$, $$\varsigma_{Q_t^{\gamma \cdot \eta}\setminus Q_t^{\gamma\cdot (\eta+1)}}(u,v)=\varsigma_{P^{\gamma\cdot \eta}\setminus P^{\gamma\cdot (\eta+1)}}(u,v).$$  Then $Q:=\amalg_{t\in M}Q_t$ has $\text{\emph{rank}}(Q)=\text{\emph{rank}}(P)$ and $\varsigma_Q=\varsigma_P|_{\Lambda_2(Q)}$.

\item If $\ee=0$, suppose also that $\eta_t=n_t+1$ for some $n_t<\omega$ and that we have a subset $M$ of $R$ and $(m_t)_{t\in M}\subset \omega$ such that $\sup_{t\in M}m_t=\omega$ and for each $t\in M$, there exist $q(0)<\ldots <q(m_t)\leqslant n_t$ and a subset $Q_t$ of $P(s_t)$ with $\text{\emph{rank}}(Q_t)= \gamma\cdot (m_t+1)$ such that for each $i\leqslant m_t$, $$Q_t^{\gamma\cdot i}\setminus Q_t^{\gamma\cdot (i+1)} = Q_t \cap (P^{\gamma\cdot q(i)}\setminus P^{\gamma\cdot (q(i)+1)})$$ and for each $(u,v)\in \Lambda_2(Q_t^{\gamma\cdot i}\setminus Q_t^{\gamma \cdot (i+1)})$, $$\varsigma_{Q_t^{\gamma\cdot i}\setminus Q_t^{\gamma\cdot (i+1)}}(u,v)= \varsigma_{P^{\gamma \cdot q(i)}\setminus P^{\gamma \cdot (q(i)+1)}}(u,v).$$   Then $Q:=\amalg_{t\in M} Q_t$ has $\text{\emph{rank}}(Q)=\text{\emph{rank}}(P)$ and $\varsigma_Q=\varsigma_P|_{\Lambda_2(Q)}$. 

\end{enumerate}

\label{grow4}

\end{lemma}

\begin{proof} $(i)$ Since $\text{rank}(Q_t)=\gamma\cdot \eta_t$, $$\text{rank}(Q)=\sup_{t\in M}\text{rank}(Q_t)=\sup_{t\in M} \gamma\cdot \eta_t=\gamma \cdot \omega^{\omega^\ee}=\text{rank}(P).$$

If $\gamma=1$, the condition that $\varsigma_Q=\varsigma_P|_{\Lambda_2(Q)}$ is trivial, since in this case $\text{rank}(Q)=\text{rank}(P)=\omega^{\omega^\ee}$, $\lambda(\gamma)=1$,  and $\varsigma_Q, \varsigma_P\equiv 0$. Suppose $1<\gamma$ and   $$\gamma=\omega^{\omega^{\ee_0}}\cdot \ldots \cdot \omega^{\omega^{\ee_l}}$$ with $\ee_0\geqslant \ldots \geqslant \ee_l\geqslant \ee$.         Then $\text{rank}(Q)=\text{rank}(P)=\omega^{\omega^{\ee_0}}\cdot \ldots \cdot \omega^{\omega^{\ee_l}}\cdot \omega^{\omega^\ee}$.   Note that for $(u,v)\in \Lambda_2(Q)$, $\varsigma_Q(u,v)\leqslant l$ if and only if there exists $\eta$ such that $u,v\in Q^{\gamma\cdot \eta}\setminus Q^{\gamma \cdot (\eta+1)}$, and otherwise $\varsigma_Q(u,v)=l+1$. Similarly, for $(u,v)\in \Lambda_2(P)$, $\varsigma_P(u,v)\leqslant l$ if and only if there exists $\eta$ such that $u,v\in P^{\gamma\cdot \eta}\setminus P^{\gamma \cdot (\eta+1)}$, and otherwise $\varsigma_P(u,v)=l+1$.

Now fix  $(u,v)\in \Lambda_2(Q)$. Since the union in the definition of $Q$ is incomparable, there exists $t\in M$ such that $u,v\in \Lambda_2(Q_t)$.  Moreover, since $Q$ is an incomparable union, there exists $\eta$ such that $u,v\in Q^{\gamma\cdot \eta}\setminus Q^{\gamma\cdot (\eta+1)}$ if and only if there exists $\eta$ such that $u,v\in Q^{\gamma \cdot \eta}_t\setminus Q_t^{\gamma \cdot (\eta+1)}$ if and only if there exists $\eta$ such that $u,v\in P^{\gamma\cdot \eta}\setminus P^{\gamma\cdot (\eta+1)}$, and if any (equivalently, all) of these three conditions holds, they hold with the same ordinal $\eta$ by the hypotheses of the lemma. In the case that such an ordinal $\eta$ exists, then $(u,v)\in \Lambda_2(Q_t^{\gamma \cdot \eta}\setminus Q_t^{\gamma \cdot (\eta+1)})\cap \Lambda_2(P^{\gamma\cdot \eta}\setminus P^{\gamma \cdot (\eta+1)})$ and $$\varsigma_Q(u,v)=\varsigma_{Q_t^{\gamma \cdot \eta}\setminus Q_t^{\gamma \cdot (\eta+1)}}(u,v)=\varsigma_{P^{\gamma \cdot \eta}\setminus P^{\gamma\cdot (\eta+1)}}(u,v)=\varsigma_P(u,v).$$   Here we are using Proposition \ref{division} for the first and last equality. Now if no such $\eta$ exists, then $\varsigma_Q(u,v)=l+1=\varsigma_P(u,v)$. Therefore $\varsigma_Q=\varsigma_P|_{\Lambda_2(Q)}$. 

$(ii)$ We have $$\text{rank}(P)\geqslant \text{rank}(Q)=\sup_{t\in M} \gamma\cdot (m_t+1)= \gamma\cdot \omega=\text{rank}(P).$$   We show that $\varsigma_Q=\varsigma_P|_{\Lambda_2(Q)}$.  Fix $(u,v)\in \Lambda_2(Q)$. The remainder of the argument is similar to $(i)$.  First,  $(u,v)\in \Lambda_2(Q_t)$ for some $t\in M$.  Then either there exists $i\leqslant m_t$ such that $u,v\in Q_t^{\gamma\cdot i}\setminus Q_t^{\gamma \cdot (i+1)}$, and therefore $u,v\in P^{\gamma \cdot q(i)}\setminus P^{\gamma \cdot (q(i)+1)}$. In this case, $$\varsigma_Q(u,v)=\varsigma_{Q_t^{\gamma \cdot i}\setminus Q_t^{\gamma \cdot (i+1)}}(u,v)= \varsigma_{P^{\gamma \cdot q(i)}\setminus P^{\gamma \cdot (q(i)+1)}}(u,v)=\varsigma_P(u,v).$$ If no such $i$ exists, then $\varsigma_Q(u,v)=\varsigma_P(u,v)=l+1$, the largest possible value of $\varsigma_P$ and $\varsigma_Q$.

\end{proof}

We are now ready to provide a complete characterization of $\mathfrak{R}_1$.   The  next result does not only yield that $\mathfrak{R}_1$ is simply the class of additively indecomposable ordinals, but provides further structural information regarding how we can choose a  monochromatic subtree to be positioned inside the original tree.

\begin{theorem} Let $\gamma$ be a non-zero ordinal and write $\gamma=\omega^{\ee_0}+\ldots +\omega^{\ee_l}$ with $\ee_0\geqslant \ldots \geqslant \ee_l$.  Let $\gamma_0=0$ and $\gamma_{i+1}=\gamma_i+\omega^{\ee_i}$ for $i=0, \ldots, l$.    For any tree $P$ with $\text{\emph{rank}}(P)=\gamma$, any $k<\omega$, and any $f:P\to k+1$, there exist a subtree $Q$ of $P$ and a function $F:l+1\to k+1$ such that $\text{\emph{rank}}(Q)=\text{\emph{rank}}(P)$ and for each $i\leqslant l$,  \begin{enumerate}[(i)]\item $Q^{\gamma_i}\setminus Q^{\gamma_{i+1}}=Q\cap (P^{\gamma_i}\setminus P^{\gamma_{i+1}})$, \item $\varsigma_{Q^{\gamma_i}\setminus Q^{\gamma_{i+1}}}=\varsigma_P|_{\Lambda_2(Q^{\gamma_i}\setminus Q^{\gamma_{i+1}})}$, \item $f|_{Q^{\gamma_i}\setminus Q^{\gamma_{i+1}}}\equiv F(i)$.   \end{enumerate}

\label{ds9}

\end{theorem}

\begin{proof} First let us note that in the case $l=0$, item $(iii)$ yields that $f|_Q$ is constant.    Furthermore, item $(i)$ is trivial in this case. We prove the result by induction.   

Let $\gamma$ be a non-zero ordinal and let us assume the result holds for all non-zero ordinals less than $\gamma$.    Write $\gamma=\omega^{\ee_0}+\ldots +\omega^{\ee_l}$ with $\ee_0\geqslant \ldots \geqslant \ee_l$ and let $\gamma_i$ be as in the statement of the theorem.    

Let us first consider the case $l>0$.    Since $\gamma=\gamma_l+\omega^{\ee_l}$, $\text{rank}(P^{\gamma_l})=\omega^{\ee_l}$.  Since $\omega^{\ee_l}<\gamma$, the inductive hypothesis applied to $f|_{P^{\gamma_l}}$ and the first sentence of the proof together yield the existence of some $j\leqslant k$ and $R\subset P^{\gamma_l}$ such that $\text{rank}(R)=\text{rank}(P^{\gamma_l})=\omega^{\ee_l}$, $\varsigma_R=\varsigma_{P^{\gamma_l}}|_{\Lambda_2(R)}$, and $f|_R\equiv j$.    Now for each $t\in \text{Leaves}(R)$, fix some $s_t\in \text{Leaves}(P^{\gamma_l})$ such that $t\leqslant s_t$.   Since $\gamma_l=\omega^{\ee_0}+\ldots +\omega^{\ee_{l-1}}<\gamma$, the inductive hypothesis applied to $f|_{P(s_t)}$ yields the existence of some $Q_t\subset P(s_t)$ and some $F_t:l\to k+1$ such that $\text{rank}(Q_t)=\omega^{\ee_0}+\ldots +\omega^{\ee_{l-1}}$ and for each $i<l$, $f|_{Q_t^{\gamma_i}\setminus Q_t^{\gamma_{i+1}}}\equiv F_t(i)$, $Q_t^{\gamma_i}\setminus Q_t^{\gamma_{i+1}}=Q_t\cap (P^{\gamma_i}(s_t)\setminus P^{\gamma_{i+1}}(s_t))=Q_t\cap (P^{\gamma_i}\setminus P^{\gamma_{i+1}})$, and for each $(u,v)\in \Lambda_2(Q_t^{\gamma_i}\setminus Q_t^{\gamma_{i+1}})$, $$\varsigma_{Q_t^{\gamma_i}\setminus Q_t^{\gamma_{i+1}}}(u,v)=\varsigma_{P^{\gamma_i}(s_t)\setminus P^{\gamma_{i+1}}(s_t)}(u,v)= \varsigma_{P^{\gamma_i}\setminus P^{\gamma_{i+1}}}(u,v).$$   For the last equality, we use Proposition \ref{ted}$(ix)$ to deduce that $(P^{\gamma_i}\setminus P^{\gamma_{i+1}})(s_t)=P^{\gamma_i}(s_t)\setminus P^{\gamma_{i+1}}(s_t)$, so $$\tau_{P^{\gamma_i}\setminus P^{\gamma_{i+1}}}(u)=\tau_{P^{\gamma_i}(s_t)\setminus P^{\gamma_{i+1}}(s_t)}(u)$$ for all $u\in P(s_t)\cap (P^{\gamma_i}\setminus P^{\gamma_{i+1}})$.     Now for each $G:l\to k+1$, let $$M_G=\{t\in \text{Leaves}(R): F_t=G\}.$$  By Lemma \ref{tek}, there exist $G:l\to k+1$ and $D\subset R$ such that $\tau_D=\tau_R|_D$, $\text{Leaves}(D)=M_G\cap D$, and $\text{rank}(D)=\text{rank}(R)=\omega^{\ee_l}$.    Now define $F:l+1\to k+1$ by letting $F(i)=G(i)$ for $i<l$ and $F(l)=j$.   Let $$Q=D\cup \Bigl(\amalg_{t\in \text{Leaves}(D)}Q_t\Bigr).$$      It follows from Lemma \ref{grow1} that $\text{rank}(Q)=\text{rank}(P)$ and that items $(i)$ and $(ii)$ are satisfied.   Since $Q^{\gamma_l}=D\subset R$, $f|_{Q^{\gamma_l}\setminus Q^{\gamma_{l+1}}}\equiv j=F(l)$.    For $i<l$, if $u\in Q^{\gamma_i}\setminus Q^{\gamma_{i+1}}$, then there exists $t\in \text{Leaves}(D)\subset M_G$ such that $u\in Q_t^{\gamma_i}\setminus Q_t^{\gamma_{i+1}}$.   Then $f(u)=F_t(i)=G(i)=F(i)$.   This yields $(iii)$ and concludes the $l>0$ case.

Now we treat the $l=0$ case, in which case $\gamma=\omega^\mu$ for some ordinal $\mu$.   As usual, we will treat the $\mu=0$, $\mu$  a successor, and $\mu$ a limit cases.   The $\mu=0$ case is simple, since for any $P$ with $\text{rank}(P)=1$, $k<\omega$, and $f:P\to k+1$, we fix $t\in P$ and let $Q=\{t\}$ and $F(0)=f(t)$.   The conclusions are easily seen to be satisfied in this case.

Now suppose that $\mu=\nu+1$ and let $\delta=\omega^\nu$.  Fix $P$ with $\text{rank}(P)=\omega^\mu$, $k<\omega$, and $f:P\to k+1$.    By writing $\gamma=\delta \cdot \omega=\delta\cdot \omega^{\omega^0}$, Case $(i)$ of Lemma \ref{grow3} yields a subset $R$ of $\text{Roots}(P)$, $(s_t)_{t\in R}\subset P$, and $(n_t)_{t\in R}\subset \omega$ such that $s_t\in \text{Leaves}(P[t]^{\delta \cdot (n_t+1)})$ and $\sup_{t\in R} n_t=\omega$.    For each $t\in R$, let $T_t=P(s_t)$ and note that $\text{rank}(T_t)=\delta \cdot (n_t+1)$.     By applying the inductive hypothesis to the restriction of $f$ to $T_t$, we deduce the existence of some $F_t:n_t+1\to k+1$ and $S_t\subset T_t$ satisfying the conclusions of the theorem.     Now for $t\in R$ and $i\leqslant k$, let $$A_{t,i}=\{j\leqslant n_t: F_t(j)=i\}.$$  Note that there must exist some $j\leqslant k$ such that $\sup_{t\in R} |A_{t,j}|=\omega$.    We may find a subset, and indeed a countable subset, $M$ of $R$ such that, with $m_t=|A_{t,j}|-1$, $\sup_{t\in M} m_t=\omega$ and $m_t>0$ for all $t\in M$.    By Lemma \ref{grow2}, for each $t\in M$, we may find a subset $Q_t$ of $S_t$ such that $\text{rank}(Q_t)=\delta \cdot (m_t+1)$ and, writing $A_{t,j}=\{a(0), \ldots, a(m_t)\}$ with $a(0)<\ldots <a(m_t)$, for each $i\leqslant m_t$, $$Q_t^{\delta \cdot i}\setminus Q_t^{\delta \cdot (i+1)}=Q_t\cap (S_t^{\delta \cdot a(i)}\setminus S_t^{\delta \cdot (a(i)+1)})$$ and $$\varsigma_{Q_t^{\delta \cdot i}\setminus Q_t^{\delta \cdot (i+1)}}=\varsigma_{S_t^{\delta \cdot a(i)}\setminus S_t^{\delta \cdot (a(i)+1)}}|_{\Lambda_2(Q_t^{\delta \cdot i}\setminus Q_t^{\delta \cdot (i+1)})}.$$  It follows from the definition of $A_{t,j}$ and the properties of $F_t$ and $S_t$ that $f|_{Q_t}\equiv j$. Define $F(0)=j$.    Let $Q=\amalg_{t\in M}Q_t$.  Obviously item $(iii)$ is satisfied. The remaining requirements are satisfied by Case $(ii)$ of Lemma \ref{grow4}.

Finally, let us suppose that $\mu$ is a limit ordinal, say $\mu=\omega^{\delta_0}+\ldots +\omega^{\delta_m}$ with $\delta_0\geqslant \ldots \geqslant \delta_m$.  Let $\nu=\omega^{\delta_0}+\ldots +\omega^{\delta_{m-1}}$ if $m>0$, and otherwise let $\nu=0$.    Let $\delta=\delta_m$, so that $\omega^\mu=\omega^\nu\cdot \omega^{\omega^\delta}$.  Let us also note that $\delta>0$, otherwise $\omega^\delta=1$ and $\mu$ would be a successor ordinal.   Fix a tree $P$ with $\text{rank}(P)=\omega^\mu$, $k<\omega$, and $f:P\to k+1$.   By Lemma \ref{grow3}, there exist a subset $R\subset \text{Roots}(P)$, $\{s_t: t\in R\}\subset P$, $\{\eta_t: t\in R\}\subset \omega^{\omega^\delta}$ such that $\sup_{t\in R}\eta_t=\omega^{\omega^\delta}$ and for each $t\in R$, $s_t\in \text{Leaves}(P[t]^{\omega^\nu \cdot \eta_t})$.  Furthermore, since we are in either case $(ii)$ or case $(iii)$ of Lemma \ref{grow3}, the ordinals $\eta_t$ may be chosen in such a way that $\omega^\nu \cdot \eta_t$ is an additively  indecomposable ordinal.  For each $t\in R$, let $T_t=P(s_t)$.   By applying the inductive hypothesis to the restriction of $f$ to $T_t$, we deduce the existence of some $F_t:1\to k+1$ and $Q_t\subset T_t$ such that $\text{rank}(Q_t)=\text{rank}(T_t)=\omega^\nu\cdot \eta_t$, $\varsigma_{Q_t}=\varsigma_{T_t}|_{\Lambda_2(Q_t)}$, and $f|_{Q_t}\equiv F_t(0)$.  For each $i\leqslant k$, let $R_i=\{t\in R: F_t(0)=i\}$.     Since $$\text{rank}(P)=\sup_{t\in R}\omega^\nu\cdot \eta_t=\sup_{t\in R} \text{rank}(Q_t),$$ there exists $j\leqslant k$ such that $$\sup_{t\in R_j}\text{rank}(Q_t)=\text{rank}(P).$$  Let $Q=\amalg_{t\in R_j} Q_t$ and let $F(0)=j$.    Then by Case $(i)$ of Lemma \ref{grow4}, $\text{rank}(Q)=\text{rank}(P)$, and items $(i)$ and $(ii)$ are satisfied by this $Q$. It is immediate from the construction that $f|_Q\equiv j=F(0)$.

\end{proof}

\begin{theorem}     \begin{enumerate}[(i)]\item Let $\gamma$ be a non-zero ordinal and write $\gamma=\omega^{\ee_0}+\ldots +\omega^{\ee_l} $ with $\ee_0\geqslant \ldots \geqslant \ee_l$.    Let $\gamma_0=0$ and for $i<l$, let $\gamma_{i+1}=\gamma_i+\omega^{\ee_i}$.  Suppose $P$ is a tree with $\text{\emph{rank}}(P)= \gamma$, $k<\omega$, and $f:P\to k+1$ is a function.  Suppose $Q\subset P$ and $F:l+1\to k+1$ are as in Theorem \ref{ds9}.   Fix $j\leqslant k$ and let  $A=\{i\leqslant l: F(i)=j\}$ and $\alpha =\omega^{\ee_{a(0)}}+\ldots +\omega^{\ee_{a(p)}}$, where by convention $\alpha=0$ if $A=\varnothing$ and otherwise $A=\{a(0), \ldots, a(p)\}$ with $a(0)<\ldots <a(p)$.    Then there exists a subtree $R$ of $P$ with $\text{\emph{rank}}(R)=\alpha$  such that $f|_R\equiv j$.       \item For $k,l<\omega$, let $F:l+1\to k+1$ be a function.   Fix  $\ee_0\geqslant \ldots \geqslant \ee_l$ and let $\gamma_i$ be as in $(i)$.  Fix a tree $P$ with $\text{\emph{rank}}(P)=\omega^{\ee_0}+\ldots +\omega^{\ee_l}$ and define $f:P\to k+1$ by letting $f|_{P^{\gamma_i}\setminus P^{\gamma_{i+1}}}=F(i)$.     Then if $j\leqslant k$, $A=\{i\leqslant l: F(i)=j\}$, $\alpha=\omega^{\ee_{a(0)}}+\ldots +\omega^{\ee_{a(p)}}$, and $Q$ is any subtree of $P$ such that $f|_Q\equiv j$, then $\text{\emph{rank}}(Q) \leqslant \alpha$. Here we obey the same conventions on $\alpha$ and $A$ as in $(i)$.   \end{enumerate}

\label{megatron}
\end{theorem}

\begin{proof}$(i)$ We choose the subtree $R$ from $Q$ using the set $A$ and Lemma \ref{grow2}.

$(ii)$    If $A=\varnothing$, the result follows by our established convention on $\alpha$ and the fact that $Q=\varnothing$ if $A=\varnothing$.     Assume $A\neq \varnothing$ and write $A=\{a(0), \ldots, a(p)\}$ with $a(0)<\ldots <a(p)$.    It is clear that $Q\subset \cup_{i\in A}P^{\gamma_i}\setminus P^{\gamma_{i+1}}=\cup_{i=0}^p P^{\gamma_{a(i)}}\setminus P^{\gamma_{a(i)+1}}$.   Let $\beta_0=0$ and $\beta_{i+1}=\beta_i+\omega^{\ee_{a(i)}}$ for $i=0, \ldots, p$.     An easy induction argument on $j$ yields that for each $j\leqslant p$,  $$Q^{\beta_j}\subset \bigcup_{i=j}^p P^{\gamma_{a(i)}}\setminus P^{\gamma_{a(i)+1}}.$$   This is because $\text{rank}(P^{\gamma_{a(i)}}\setminus P^{\gamma_{a(i)+1}})=\omega^{\ee_{a(i)}}$ and for any ordinal $\zeta$ and $j\leqslant p$, any leaf in $(P^{\gamma_{a(j)}}\setminus P^{\gamma_{a(j)+1}})^\zeta$ is also a leaf in $$\Bigl(\bigcup_{i=j+1}^p P^{\gamma_{a(i)}}\setminus P^{\gamma_{a(i)+1}}\Bigr)\cup (P^{\gamma_{a(j)}}\setminus P^{\gamma_{a(j)+1}})^\zeta.$$  An argument similar to that in Lemma \ref{grow2} yields the induction argument on $j$.    Now apply the argument with $j=p$ and note that $\alpha=\beta_{p+1}$, $Q^\alpha=Q^{\beta_{p+1}}=\varnothing$. 

\end{proof}

We have already discussed the easy inclusion of the next corollary, and the reverse inclusion is an immediate consequence of Theorem \ref{ds9}. 

\begin{corollary} The class $\mathfrak{R}_1$ is the class of additively indecomposable ordinals.

\label{up}
\end{corollary}

\begin{rem}\upshape That $\mathfrak{R}_1$ is simply the class of additively indecomposable ordinals is related to certain  partition problems in the spirit of those introduced by Erd\H{o}s and Rado \cite{ER}, and the topological partition problem solved by Baumgartner \cite{B}. For topological spaces $X,Y$, the relation $X\rightarrow (Y)_2^1$ denotes that if $X$ is partitioned into two subsets, one of the subsets contains a homeomorphic copy of $Y$.  Baumgartner showed that for a countable ordinal $\alpha$, the topological relation $\alpha\rightarrow (\alpha)_2^1$ holds if and only if and only if $\alpha=\omega^{\omega^\beta}$ for some $\beta$. We note that the Cantor-Bendixson index of the ordinal interval $\omega^\xi$ is $\xi$, so the Cantor-Bendixson index of $\omega^{\omega^\beta}$ is $\omega^\beta$. Thus the countable ordinals satisfying the topological relation $\alpha\rightarrow (\alpha)_2^1$ are precisely those whose Cantor-Bendixson index is an additively indecomposable ordinal. This is in close analogy to the above characterization of $\mathfrak{R}_1$ as the class of the additively indecomposable ordinals.

 Recall that the definition of a tree is a partially ordered set $(P, \leqslant )$ such that for each $t\in P$, the ancestor set $A_P[t]=\{s\in P: s\leqslant t\}$ is well-ordered. Since we are concerned only with well-founded trees, we have limited our study to those trees such that $A_P[t]$ is finite and linearly ordered for all $t\in P$, but we now recall the more general definition so that we may compare Theorem \ref{ds9}.      Given a tree $(P, \leqslant )$, we may define $P(0)$ to be the set of minimal members.  Assuming $P(\zeta)$ has been defined for each $\zeta<\xi$, we define $P(\xi)$ to be the set of minimal members of $P\setminus \cup_{\zeta<\xi} P(\zeta)$ if this set is non-empty, and otherwise $P(\xi)=\varnothing$.   We then define the \emph{height} of $P$ to be the minimum $\xi$ such that $P(\xi)=\varnothing$.    If $P$ is a well-ordered set (that is, a tree with a single branch), the height of $P$ is simply the order type.  Let $\mathfrak{S}_1$ denote the class of non-zero ordinals $\zeta$ such that for any finite partition of $\zeta$, there exists at least one member of the partition with order type $\zeta$. Erd\H{o}s and Rado showed in \cite{ER} that $\mathfrak{S}_1$ is simply the class of additively indecomposable ordinals, which coincides with $\mathfrak{R}_1$.   Let $\mathfrak{T}_1$ denote the class of non-zero ordinals $\zeta$ such that for any finite partition of a tree with height $\zeta$, one of the members of the partition has height $\zeta$.    Since every well-ordered set is a tree, and in fact a tree with a single branch, $\mathfrak{T}_1\subset \mathfrak{S}_1$. It is not immediate from the definition that $\mathfrak{S}_1\subset \mathfrak{T}_1$, since if $\zeta$ is a limit ordinal, a tree of height $\zeta$ need not have a well-ordered subset with order type $\zeta$. Indeed, for a limit ordinal $\zeta$, if $P=\amalg_{\xi<\zeta}P_\xi$ is a totally incomparable union of trees such that $P_\xi$ is well-ordered with order type $\xi$, $P$ has no well-ordered subset of order type $\zeta$.   However, for a limit ordinal $\zeta$, any tree $P$ with height $\zeta$ has well-ordered subsets of order type $\xi$ for each $\xi<\zeta$. From this one can easily deduce from the fact that $\mathfrak{S}_1$ is the class of additively indecomposable ordinals that $\mathfrak{T}_1$ is also the class of additively indecomposable ordinals. Thus $\mathfrak{R}_1=\mathfrak{S}_1=\mathfrak{T}_1$.

One may ask whether the characterization of $\mathfrak{R}_1$ implies the characterization of $\mathfrak{T}_1$, or if the characterization of $\mathfrak{T}_1$ implies the characterization of $\mathfrak{R}_1$. However, they appear to be distinct.   The height of a tree is distinct from the rank of the tree, as the height is obtained by successive removals of the minimal remaining members, while the rank is defined by successive removals of the maximal remaining members. The height and rank of a tree are equal if either one is finite, and in this case they are each equal to the maximum cardinality of a linearly ordered subset. However, one can see that any well-founded tree (that is, one for which the rank is an ordinal), the height cannot exceed $\omega$.  Of course, one may ask whether these notions are related through the reverse order of a given order. That is, given a partially ordered set $(P, \leqslant )$, we may define $\leqslant^*$ on $P$ by $s\leqslant^* t$ if and only if $t\leqslant s$. Since this reverses the roles of maximal and minimal members of a partially ordered set, one may wonder if the apparently distinct facts that $\mathfrak{R}_1$ and $\mathfrak{T}_1$ are each equal to the class of additively indecomposable ordinals may be related by considering the reverse of the order on a given tree.  However, given a tree $(P, \leqslant)$, $(P, \leqslant^*)$ need not be a tree.   Thus it seems that these facts are genuinely distinct.

Further distinguishing Ramsey results involving rank from those involving height is the solution to the problem for pairs rather than singletons. We wish to thank François Dorais for bringing the following argument to our attention. Recall that a cardinal number $\kappa$ is \emph{weakly compact} if it is uncountable and for any function $f:[\kappa]^2\to 2$, there exists a subset $S$ of $\kappa$ with cardinality $\kappa$ such that $f|_{[S]^2}$ is constant. Recall that $[\kappa]^2$ denotes the set of two element subsets of $\kappa$.  Recall that weakly compact cardinals are large cardinals, whose existence cannot be proved from the standard axioms of set theory.   Let us denote by $\mathfrak{S}_2$ the class of all non-zero ordinals $\xi$ such that for any well-ordered set $X$ with order type $\xi$ and any  function $f:[X]^2\to 2$, there exists a subset $Y$ of $X$ with order type $\xi$ such that $[Y]^2$ is constant.  First, if $\xi\in \mathfrak{S}_2$, then $\xi$ must be a cardinal. Indeed, if $\xi$ is an ordinal whose cardinality is $\kappa$, we may fix a bijection $g:\xi\to \kappa$ and define $f(\{\alpha, \beta\})=0$ if the order of $\{\alpha, \beta\}$ is preserved by $g$ and $f(\{\alpha, \beta\})=1$ if the order of $\{\alpha, \beta\}$ is reversed by $g$.   Then if $Y\subset \xi$ is such that $f|_{[Y]^2}$ is constant, then $g|_Y$ is increasing if $f|_{[Y]^2}\equiv  0$, which means the order type of $\xi$ is not more than, and is therefore equal to, that of $\kappa$, or $g|_Y$ is decreasing if $f|_{[Y]^2}\equiv 1$, in which case $Y$ is finite. Now that we have demonstrated that the members of $\mathfrak{S}_2$ must be cardinals, it follows from the definition of weakly compact cardinals and Theorem \ref{Ramsey}$(ii)$ that $\mathfrak{S}_2$ consists precisely of $1,2, \omega$, and the weakly compact cardinals.  This is a drastically different class than $\mathfrak{R}_2$, the multiplicatively indecomposable ordinals. 

\end{rem}

We conclude this section with an application of our characterization of $\mathfrak{R}_1$, which will be necessary for our results on pairs.  As we have already explained, we think of a tree of rank $\omega^\ee\cdot (n+1)$ as consisting of $n+1$ levels $P^{\omega^\ee \cdot i}\setminus P^{\omega^\ee \cdot (i+1)}$, $i=0, \ldots, n$.   Our goal will be to begin with a tree of rank $\omega^\ee\cdot (n+1)$ and a coloring of linearly ordered pairs in this tree and to pass to a subtree $Q$ of rank $\omega^\ee\cdot (p+1)$, where $p$ depends upon $n$ and the number of colors used, in such a way that every pair in $\Lambda_2(Q)$ whose members lie on different levels of $Q$ receive the same color.    In what follows, $n+1$ is endowed with its usual order, so $\Lambda_2(n+1)$ is identifiable with $[n+1]^2$, the two element subsets of $n+1$. 

\begin{lemma} Let $\gamma=\omega^{\ee_0}+\ldots +\omega^{\ee_l}$ with $\ee_0\geqslant \ldots \geqslant \ee_l$.  Let $\gamma_0=0$ and $\gamma_{i+1}=\gamma_i+\omega^{\ee_i}$ for $i=0, \ldots, l$.      For a tree $P$ with $\text{\emph{rank}}(P)=\gamma$ and a function $f:\Lambda_2(P)\to k+1$, there exist $Q\subset P$ with $\text{\emph{rank}}(Q)=\text{\emph{rank}}(P)$ and a function $G:\Lambda_2(l+1)\to k+1$ such that for all $i\leqslant l$, \begin{enumerate}[(i)]\item $Q^{\gamma_i}\setminus Q^{\gamma_{i+1}}=Q\cap (P^{\gamma_i}\setminus P^{\gamma_{i+1}})$, \item if $(s,t)\in \Lambda_2(Q)$ and $i,j\leqslant l$ are such that $i<j$, $s\in Q^{\gamma_j}\setminus Q^{\gamma_{j+1}}$, and  $t\in Q^{\gamma_i}\setminus Q^{\gamma_{i+1}}$, then $f(s,t)= G(i,j),$  \item $\varsigma_{Q^{\gamma_i}\setminus Q^{\gamma_{i+1}}}(s,t)=\varsigma_{P^{\gamma_i}\setminus P^{\gamma_{i+1}}}(s,t)$ for all $(s,t)\in \Lambda_2(Q^{\gamma_i}\setminus Q^{\gamma_{i+1}})$. \end{enumerate} 

\label{kotr}

\end{lemma}

\begin{proof} In the proof, for sets $S,T$ and $b\in S^T$, we use $b(t)$ to denote the coordinates of $b$. That is, $b=(b(t))_{t\in T}$.

We prove by induction on $l<\omega$ that if $\gamma$ is an ordinal which can be written $\gamma=\omega^{\ee_0}+\ldots +\omega^{\ee_l}$ with $\ee_0\geqslant \ldots \geqslant\ee_l$, then for any $P$ with $\text{rank}(P)=\gamma$, $k<\omega$, and $f:\Lambda_2(P)\to k+1$, there exist $Q\subset P$ and $G:\Lambda_2(l+1)\to k+1$ as in the lemma.   If $l=0$, for any $P,k,f$, we let $Q=P$.  Then item $(ii)$ is vacuous in this case, and the other items are trivially verified.

Assume $0<l<\omega$ and the result holds for $m=l-1$.    Let $\gamma=\omega^{\ee_0}+\ldots +\omega^{\ee_l}$ with $\ee_0\geqslant \ldots \geqslant \ee_l$.  Let $\gamma_i$ be as in the  lemma.    For each $t\in \text{Leaves}(P^{\gamma_l})$, the inductive hypothesis applied to $P(t)$ yields the existence of some $R_t\subset P(t)$ with $\text{rank}(R_t)=\gamma_l=\text{rank}(P(t))$ and $G_t:\Lambda_2(m+1)\to k+1$ as in the lemma.    Now for each function $H:\Lambda_2(m+1)\to k+1$, let $$M_H=\{t\in \text{Leaves}(P^{\gamma_l}): G_t=H\}.$$   Then by Lemma \ref{tek}, there exist $H:\Lambda_2(m+1)\to k+1$ and a subset $E$ of $P^{\gamma_l}$ such that $\tau_E=\tau_{P^{\gamma_l}}|_E$, $\text{rank}(E)=\omega^{\ee_l}$, and $\text{Leaves}(E)=M_H\cap E$.

Now let us recall the definition of the ancestor set $A_E[t]=\{s\in E: s\leqslant t\}$.   For each $t\in \text{Leaves}(E)$, let us define $$b_t:R_t\to (k+1)^{A_E[t]}$$ by $$b_t(u)=(f(s,u))_{s\in A_E[t]}.$$  Since $(k+1)^{A_E[t]}$ is finite, by Theorem \ref{ds9}, there exist $B_t:m+1\to (k+1)^{A_E[t]}$ and $Q_t\subset R_t$ such that $\text{rank}(Q_t)=\text{rank}(R_t)$ and for each $i\leqslant m$, $Q_t^{\gamma_i}\setminus Q_t^{\gamma_{i+1}}=Q_t\cap (R^{\gamma_i}_t\setminus R^{\gamma_{i+1}}_t)$, $\varsigma_{Q^{\gamma_i}_t\setminus Q^{\gamma_{i+1}}_t}=\varsigma_{R_t^{\gamma_i}\setminus R_t^{\gamma_{i+1}}}|_{\Lambda_2(Q_t^{\gamma_i}\setminus Q_t^{\gamma_{i+1}})}$, and $b_t|_{Q_t^{\gamma_i}\setminus Q_t^{\gamma_{i+1}}}\equiv B_t(i)$.     Before proceeding, let us give some explanation for what we have accomplished so far. The tree $Q_t$, having the same rank as $P(t)$, which is $\gamma_l=\omega^{\ee_0}+\ldots +\omega^{\ee_{l-1}}$, has $l$ levels.   Passing from $R_t$ to $Q_t$ has made it such that, for each ancestor $s$ of $t$ and each $u\in Q_t$, $f(s,u)$ does not depend on $u$, but only on the $i\leqslant m$ such that $u\in Q_t^{\gamma_i}\setminus Q_t^{\gamma_{i+1}}$.  Namely, if $B_t(i)=(B_t(i)(s))_{s\in A_E[t]}\in (k+1)^{A_E[t]}$, then for each $u\in Q_t^{\gamma_i}\setminus Q_t^{\gamma_{i+1}}$, $f(s,u)=B_t(i)(s)$.

Now let us recall that $\Pi(E)=\{(s,t)\in E\times \text{Leaves}(E): s\leqslant t\}$.  Let us define a function $c:\Pi(E)\to (k+1)^{m+1}$ in the following way: For $t\in \text{Leaves}(E)$ and $i\leqslant m$, let $B_t(i)=(B_t(i)(s))_{s\in A_E[t]}$ as in the previous paragraph.     Then for $s\leqslant t$, define $$c(s,t)= (B_t(i)(s))_{i\leqslant m}.$$    By Lemma \ref{tek}, there exist a subset $F$ of $E$ and a function $C:F\to (k+1)^{m+1}$ such that $\tau_F=\tau_E|_F$, $\text{rank}(F)=\text{rank}(E)=\omega^{\ee_l}$, and for each $(s,t)\in \Pi(F)\subset \Pi(E)$, $C(s)=c(s,t)$.  Note  that $\tau_F=\tau_E|_F$ implies that $\varsigma_F=\varsigma_E|_{\Lambda_2(F)}=\varsigma_P|_{\Lambda_2(F)}$.      By Theorem \ref{ds9}, there exist $a\in (k+1)^{m+1}$ and a subset $D$ of $F$ such that $\text{rank}(D)=\text{rank}(F)$, $\varsigma_D=\varsigma_F|_{\Lambda_2(D)}$, and $C|_D\equiv a$.

  Now for each $t\in \text{Leaves}(D)$, fix $s_t\in \text{Leaves}(F)\subset M_H$ such that $t\leqslant s_t$.     Let $$Q=D\cup \Bigl(\amalg_{t\in \text{Leaves}(D)} Q_{s_t}\Bigr).$$   Define $G:\Lambda_2(l+1)\to k+1$ as follows: Let $G(i,j)=a(i)$ if $j=l$, and let $G(i,j)=H(i,j)$ if $j<l$.

 We last undertake the  process of verifying the requirements. Let us note that $$\varsigma_D=\varsigma_F|_{\Lambda_2(D)}=\varsigma_E|_{\Lambda_2(D)}=\varsigma_{P^{\gamma_l}}|_{\Lambda_2(D)}$$ and for each $t\in \text{Leaves}(D)$ and $i\leqslant m$, $$Q_t^{\gamma_i}\setminus Q_t^{\gamma_{i+1}}=Q_t\cap (P^{\gamma_i}\setminus P^{\gamma_{i+1}})$$ and $$\varsigma_{Q_t^{\gamma_i}\setminus Q_t^{\gamma_{i+1}}}=\varsigma_{P^{\gamma_i}\setminus P^{\gamma_{i+1}}}|_{\Lambda_2(Q_t^{\gamma_i}\setminus Q_t^{\gamma_{i+1}})}.$$  Then the fact that $\text{rank}(Q)=\text{rank}(P)$ and items $(i)$ and $(iii)$ follow from Lemma \ref{grow1}.     Now suppose that $(s,u)\in \Lambda_2(Q)$ and $i< j< l$ are such that $s\in Q^{\gamma_j}\setminus Q^{\gamma_{j+1}}$ and $u\in Q^{\gamma_i}\setminus Q^{\gamma_{i+1}}$.    Then  there exists $t\in \text{Leaves}(D)$ such that $s\in    Q^{\gamma_j}_{s_t}\setminus Q^{\gamma_{j+1}}_{s_t}$ and $u\in Q^{\gamma_i}_{s_t}\setminus Q^{\gamma_{i+1}}_{s_t}$.   It follows from the properties of $Q_t$ that $s\in R_{s_t}^{\gamma_j}\setminus R_{s_t}^{\gamma_{j+1}}$ and $u\in R_{s_t}^{\gamma_i}\setminus R_{s_t}^{\gamma_{i+1}}$.   Since $s_t\in  M_H$,  $f(s,u)=H(i,j)=G(i,j)$.

Last, suppose that $(s,u)\in \Lambda_2(Q)$ and $i<l$ are such that $s\in Q^{\gamma_l}\setminus Q^{\gamma_{l+1}}=D$ and $u\in Q^{\gamma_i}\setminus Q^{\gamma_{i+1}}$.   Then there exists $t\in \text{Leaves}(D)$ such that $s\leqslant t\leqslant s_t<u$.     Since $(s,s_t)\in \Pi(F)$,  $u\in Q^{\gamma_i}\setminus Q^{\gamma_{i+1}}$, and  $(s,u)\in \Lambda_2(Q)$, it follows that $u\in Q_{s_t}^{\gamma_i}\setminus Q_{s_t}^{\gamma_{i+1}}$ and $$f(s,u)=b_t(u)(s)=B_t(i)(s)=c(s,t)(i)=C(s)(i)=a(i)=G(i,l).$$

\end{proof}

\begin{corollary} Fix $p,k< \omega$.   Let $r=r(p,1,k)< \omega$ be the number from Theorem \ref{Ramsey}$(ii)$. Then for any additively indecomposable ordinal  $\gamma$, any tree $P$  with $\text{\emph{rank}}(P)=\gamma \cdot (r+1)$, and any function $f:\Lambda_2(P)\to k+1$,  there exist $j\leqslant k$, $R\subset P$ with $\text{\emph{rank}}(R)=\gamma\cdot (p+1)$, and a subset $\{q(0), \ldots, q(p)\}\subset r+1$ with $q(0)<\ldots <q(p)$ such that for $(s,t)\in \Lambda_2(R)$, \begin{enumerate}[(i)]\item if $\tau_{R, \gamma}(s)>\tau_{R,\gamma}(t)$, $f(s,t)=j$,  \item for each $i\leqslant p$, $R^{\gamma\cdot i}\setminus R^{\gamma \cdot (i+1)}=R\cap (P^{\gamma\cdot q(i)}\setminus P^{\gamma \cdot (q(i)+1)})$, \item for each $i\leqslant p$, $\varsigma_{R^{\gamma \cdot i}\setminus R^{\gamma\cdot (i+1)}}=\varsigma_{P^{\gamma\cdot q(i)}\setminus P^{\gamma\cdot (q(i)+1)}}|_{\Lambda_2(R^{\gamma \cdot i}\setminus R^{\gamma\cdot (i+1)})}$.  \end{enumerate}

\label{CD}
\end{corollary}

\begin{proof} Let $\gamma$, $P$, $f$ be as in the statement.  By Lemma \ref{kotr}, there exist $Q\subset P$ and $G:\Lambda_2(r+1)\to k+1$ such that $\text{rank}(Q)=\gamma\cdot (r+1)$, for each $i\leqslant p$, $Q^{\gamma\cdot i}\setminus Q^{\gamma\cdot (i+1)}=Q\cap (P^{\gamma\cdot i}\setminus P^{\gamma\cdot (i+1)})$ and $\varsigma_{Q^{\gamma\cdot i}\setminus Q^{\gamma\cdot (i+1)}}(s,t)=\varsigma_{P^{\gamma\cdot i}\setminus P^{\gamma\cdot (i+1)}}(s,t)$ for all $(s,t)\in \Lambda_2(Q^{\gamma\cdot i}\setminus Q^{\gamma\cdot (i+1)})$, and such that if $(s,t)\in \Lambda_2(Q)$ with $\tau_{Q, \gamma}(s)>\tau_{Q, \gamma}(t)$, then $f(s,t)=G(\tau_{Q,\gamma}(t), \tau_{Q, \gamma}(s))$. 

By the definition of $r$, there exist $j\leqslant k$ and  a subset $S$ of $r+1$ with cardinality $p+1$, which we denote by $\{q(0), \ldots, q(p)\}$ with $q(0)<\ldots <q(p)$, such that $G|_{\Lambda_2(S)}\equiv j$.   By Lemma \ref{grow2}, there exists a subset $R$ of $Q$ with $\text{rank}(R)=\gamma \cdot (p+1)$ such that for each $i\leqslant p$, $$R^{\gamma\cdot i}\setminus R^{\gamma\cdot (i+1)}= R\cap (Q^{\gamma \cdot q(i)}\setminus Q^{\gamma \cdot (q(i)+1)})$$ and $\varsigma_{R^{\gamma\cdot i}\setminus R^{\gamma\cdot (i+1)}}(s,t)=\varsigma_{Q^{\gamma \cdot q(i)}\setminus Q^{\gamma \cdot (q(i)+1)}}(s,t)$ for all $(s,t)\in \Lambda_2(R^{\gamma\cdot i}\setminus R^{\gamma\cdot (i+1)})$.      We verify property $(i)$.  We note that since $R^{\gamma\cdot i}\setminus R^{\gamma\cdot (i+1)}= R\cap (Q^{\gamma \cdot q(i)}\setminus Q^{\gamma \cdot (q(i)+1)})$, $q(\tau_{R, \gamma}(s))=\tau_{Q, \gamma}(s)$ for each $s\in R$. From this and the properties of $Q$ and $G$, if $(s,t)\in \Lambda_2(R)$ with $\tau_{R, \gamma}(s)>\tau_{R, \gamma}(t)$, then $\tau_{Q, \gamma}(s)>\tau_{Q,\gamma}(t)$ and $$f(s,t)=G(\tau_{Q, \gamma}(t), \tau_{Q, \gamma}(s))=G(q(\tau_{R,\gamma}(t)), q(\tau_{R, \gamma}(s)))=j.$$

\end{proof}

\section{Contractions and sparseness}

 Let $\gamma=\omega^{\omega^{\ee_0}}\cdot \ldots\cdot \omega^{\omega^{\ee_l}}$, $\ee_0\geqslant \ldots \geqslant \ee_l$ be an infinite,  additively indecomposable ordinal. Given a non-empty subset $A$ of $l+1=\lambda(\gamma)$ and a tree $P$ with $\text{rank}(P)=\gamma$,   we say $Q\subset P$ is an $A$-\emph{contraction} of $P$ provided that, with $A=\{a(0), \ldots, a(q)\}$ and $a(0)<\ldots <a(q)$,  $$\text{rank}(Q)=\omega^{\omega^{\ee_0}\cdot 1_A(0)+\omega^{\ee_1}\cdot 1_A(1)+ \ldots +\omega^{\ee_l}\cdot 1_A(l)}=\omega^{\omega^{\ee_{a(0)}}}\cdot \ldots \cdot\omega^{\omega^{\ee_{a(q)}}}$$ and $$\varsigma_P(s,t)=a(\varsigma_Q(s,t))$$    for all $(s,t)\in \Lambda_2(Q)$.    We say $Q\subset P$ is a $\varnothing$-\emph{contraction} of $P$ if $\text{rank}(Q)=1$.

Let us say that an additively indecomposable ordinal $\gamma$ is \emph{acceptable} if for any tree $P$ with $\text{rank}(P)=\gamma$ and any  subset $A$ of $\lambda(\gamma)$, $P$ has an $A$-contraction.   It is evident that $1$ is acceptable, since any tree of rank $1$ is a $\varnothing$-contraction of itself.

\begin{lemma} Suppose $\gamma$ is an additively indecomposable ordinal which is acceptable. If $\gamma>1$, write $\gamma=\omega^{\omega^{\ee_0}}\cdot \ldots \cdot \omega^{\omega^{\ee_l}}$ with $\ee_0\geqslant \ldots \geqslant \ee_l$.    Fix $B\subset \lambda(\gamma)$.     If $B\neq \varnothing$, write $B=\{b(0),\ldots ,b(m)\}$ with $b(0)<\ldots <b(m)$ and let $\beta=\omega^{\omega^{\ee_{b(0)}}}\cdot\ldots\cdot \omega^{\omega^{\ee_{b(m)}}}.$ If $B=\varnothing$,  let $\beta=1$.     Then for any ordinal $\zeta$ and any tree $P$ with $\text{\emph{rank}}(P)=\gamma \cdot \zeta$, there exists a subtree $R$ of $P$ such that $\text{\emph{rank}}(R)=\beta\cdot  \zeta$, $R^{\beta\cdot \delta}\setminus R^{\beta\cdot(\delta+1)}=R\cap (P^{\gamma \cdot \delta}\setminus P^{\gamma\cdot(\delta+1)})$ for each $\delta<\zeta$,  and such that for any $(s,t)\in \Lambda_2(R)$, exactly one of the following holds: \begin{enumerate}[(i)]\item $\tau_{P, \gamma}(s)=\tau_{R, \beta}(s)>\tau_{R,\beta}(t)=\tau_{P,\gamma}(t)$.     \item With $\eta=\tau_{P, \gamma}(s)$, $\tau_{P,\gamma}(t)=\tau_{R,\beta}(s)=\tau_{R, \beta}(t)=\eta$ and  $$b(\varsigma_{R^{\beta\cdot \eta}\setminus R^{\beta\cdot  (\eta+1)}}(s,t))=\varsigma_{P^{\gamma \cdot \eta}\setminus P^{\gamma\cdot (\eta+1)}}(s,t).$$   \end{enumerate}

\label{proto}
\end{lemma}

\begin{proof} First note that the condition that $$R^{\beta \cdot \eta}\setminus R^{\beta\cdot (\eta+1)}=R\cap (P^{\gamma\cdot \eta}\setminus P^{\gamma\cdot (\eta+1)})$$ for all $\eta<\zeta$ implies that $\tau_{R,\beta}=\tau_{P, \gamma}|_R$.   Thus once we establish the first equality in the preceding sentence, we will have $\tau_{P,\gamma}(s)=\tau_{R, \beta}(s)$ and $\tau_{P, \gamma}(t)=\tau_{R,\beta}(t)$ in items $(i)$ and $(ii)$.

 We prove the result by induction on $\zeta$. Note that in all cases, at most one of the items $(i)$ and $(ii)$ holds, so we must show that at least one of them holds.  If $\zeta=0$, the result is vacuous.

Suppose  $\zeta=1$ and $\text{rank}(P)=\gamma \cdot 1=\gamma$. Since $\gamma$ is acceptable, $P$ admits a $B$-contraction $R$.   By the properties of $B$-contraction, $\text{rank}(R)=\beta$ and $$b(\varsigma_R(s,t))=\varsigma_P(s,t)$$ for all $(s,t)\in \Lambda_2(R)$. Note that in this case, since $\gamma=\text{rank}(P)$ and $\beta=\text{rank}(R)$,   $\tau_{P, \gamma}(s)=0=\tau_{R, \beta}(s)$ for all $s\in R$. Thus $(ii)$ is satisfied for all $(s,t)\in \Lambda_2(R)$ in this case.

Suppose that for some $\zeta>1$, $\text{rank}(P)=\gamma \cdot (\zeta+1)= \gamma\cdot\zeta+\gamma$, and the result holds for $\zeta$.  Since $\text{rank}(P^{\gamma \cdot \zeta})=\gamma$ and since $\gamma$ is acceptable, there exists a $B$-contraction $Q$ of $P^{\gamma \cdot \zeta}$.   For each $t\in \text{Leaves}(Q)$, fix some $s_t\in \text{Leaves}(P^{\gamma\cdot \zeta})$ such that $t\leqslant s_t$.    Now since $\text{rank}(P(s_t))= \gamma\cdot  \zeta$, by the inductive hypothesis, there exists a subtree $R_t\subset P(s_t)$ such that $\text{rank}(R_t)=\beta\cdot \zeta$, $$R_t^{\beta\cdot \eta}\setminus R_t^{\beta\cdot (\eta+1)}=R_t\cap (P(s_t)^{\gamma\cdot \eta}\setminus P(s_t)^{\gamma\cdot (\eta+1)})= R_t\cap (P^{\gamma\cdot \eta}\setminus P^{\gamma\cdot (\eta+1)})$$ for all $\eta<\zeta$, and for each $(u,v)\in \Lambda_2(R_t)$, either $$\tau_{R_t,\beta}(u)=\tau_{P(s_t), \gamma}(u)=\tau_{P, \gamma}(u)>\tau_{R_t, \beta}(v)=\tau_{P(s_t), \gamma}(v)=\tau_{P,\gamma}(v) $$ or for some $\eta<\zeta$, $$\tau_{R_t, \beta}(u)=\tau_{P, \gamma}(u)=\tau_{R_t, \beta}(v)=\tau_{P, \gamma}(v)=\eta$$ and $$b(\varsigma_{R_t^{\beta\cdot \eta}\setminus R_t^{\beta\cdot (\eta+1)}}(u,v))=\varsigma_{P(s_t)^{\gamma\cdot \eta}\setminus P(s_t)^{\gamma\cdot (\eta+1)}}(u,v)=\varsigma_{P^{\gamma\cdot \eta}\setminus P^{\gamma\cdot (\eta+1)}}(u,v).$$ Here we repeatedly use the fact that $\tau_{P(s_t)}=\tau_P|_{P(s_t)}$, which follows from Proposition \ref{ted}$(viii)$.         Let $$R=Q\cup \Bigl(\amalg_{t\in \text{Leaves}(Q)} R_t\Bigr).$$   It follows from Proposition \ref{ted}$(xi)$ that $\text{rank}(R)=\beta \cdot (\zeta+1)$, $Q=R^{\beta\cdot \zeta}\setminus R^{\beta\cdot (\zeta+1)}$, and for each $\eta\leqslant \zeta$, $$R^{\beta\cdot \eta}\setminus R^{\beta\cdot(\eta+1)}=\amalg_{t\in \text{Leaves}(Q)} R_t^{\beta\cdot \eta}\setminus R_t^{\beta\cdot (\eta+1)}.$$   From this and the properties of the trees $R_t$, if $\eta=\zeta$,  $$R^{\beta\cdot \eta}\setminus R^{\beta\cdot (\eta+1)}=Q=R\cap (P^{\gamma\cdot \eta}\setminus P^{\gamma\cdot (\eta+1)}),$$ and if $\eta<\zeta$,  \begin{align*} R^{\beta\cdot \eta}\setminus R^{\beta\cdot (\eta+1)} & = \amalg_{t\in \text{Leaves}(Q)} R_t^{\beta \cdot \eta}\setminus R_t^{\beta\cdot (\eta+1)}=\Bigl(\amalg_{t\in \text{Leaves}(Q)}R_t\Bigr)\cap  P^{\gamma\cdot \eta}\setminus P^{\gamma\cdot(\eta+1)} \\ & =R\cap (P^{\gamma\cdot \eta}\setminus P^{\gamma\cdot (\eta+1)}).\end{align*}          Now fix $(u,v)\in \Lambda_2(R)$. Note that $u<v$ implies that $\tau_{R, \beta}(u)\geqslant \tau_{R, \beta}(v)$.  By the first sentence of the proof, $\tau_{R,\beta}(u)=\tau_{P, \gamma}(u)$ and $\tau_{R, \beta}(v)=\tau_{P, \gamma}(v)$.    If $\tau_{R, \beta}(u)>\tau_{R, \beta}(v)$, $(i)$ is satisfied.  If $\tau_{R, \beta}(u)=\tau_{R, \beta}(v)$, then consider two cases. If $\tau_{R, \beta}(u)=\zeta$, then $u,v\in Q=R^{\beta\cdot \zeta}\setminus R^{\beta\cdot (\zeta+1)}$.  Since $Q$ is a $B$-contraction of $P^{\gamma\cdot \zeta}$, $$b(\varsigma_Q(u,v))=\varsigma_{P^{\gamma\cdot \zeta}}(u,v).$$ Thus $(ii)$ is satisfied in this case.  Now if $\eta=\tau_{R, \beta}(u)=\tau_{R, \beta}(v)<\zeta$, then there exists a unique leaf $t\in \text{Leaves}(Q)$ such that $u,v\in R_t$.  Since $\tau_{R_t, \beta}=\tau_{R, \beta}|_{R_t}$ by Proposition \ref{ted}$(viii)$, $(u,v)\in \Lambda_2(R_t^{\beta\cdot \eta}\setminus R_t^{\beta\cdot (\eta+1)})$ and $\varsigma_{R_t^{\beta\cdot \eta}\setminus R_t^{\beta\cdot (\eta+1)}}(u,v)=\varsigma_{R^{\beta\cdot \eta}\setminus R^{\beta\cdot (\eta+1)}}(u,v)$.    By the properties of $R_t$, $$b(\varsigma_{R^{\beta\cdot \eta}\setminus R^{\beta\cdot (\eta+1)}}(u,v))=\varsigma_{P^{\gamma\cdot \eta}\setminus P^{\gamma\cdot (\eta+1)}}(u,v).$$  This finishes the successor case.

Now suppose $\zeta$ is a limit ordinal and the result holds for all smaller ordinals.   Let $P$ be a tree with rank $\gamma\cdot \zeta$.    Let $$S=\{t\in \text{Roots}(P): \text{rank}(P[t])>\gamma\}$$ and for each $t\in \text{Roots}(S)$, let $$\eta_t= \sup\{\eta: \text{rank}(P[t])>\gamma\cdot \eta\}.$$  Arguing as in Lemma \ref{grow3}, this supremum is a maximum, $\eta_t<\zeta$, and $\sup_{t\in S}\eta_t=\zeta$.  For each $t\in S$, we may fix $s_t\in \text{Leaves}(P[t])^{\gamma\cdot \eta_t}$ and let $Q_t=P(s_t)$.   By the inductive hypothesis, for each $t\in S$, there exists $R_t\subset Q_t$ with $\text{rank}(R_t)=\beta\cdot \eta_t$ such that for each $\eta<\eta_t$, $$R_t^{\beta\cdot \eta}\setminus R_t^{\beta\cdot (\eta+1)}=R_t\cap (Q_t^{\gamma\cdot \eta}\setminus Q_t^{\gamma\cdot (\eta+1)})=R_t\cap (P^{\gamma\cdot \eta}\setminus P^{\gamma\cdot (\eta+1)}),$$ and for each $(u,v)\in \Lambda_2(R_t)$, either $\tau_{R_t, \beta}(u)>\tau_{R_t, \beta}(v)$ or there exists $\eta<\eta_t$ such that $$\tau_{R_t, \beta}(u)=\tau_{R_t, \beta}(v)=\tau_{Q_t, \gamma}(u)=\tau_{P, \gamma}(u)=\tau_{Q_t, \gamma}(v)=\tau_{P, \gamma}(v)=\eta$$ and $$b(\varsigma_{R_t^{\beta\cdot \eta}\setminus R_t^{\beta\cdot (\eta+1)}}(u,v))=\varsigma_{Q_t^{\gamma\cdot \eta}\setminus Q_t^{\gamma\cdot (\eta+1)}}(u,v)=\varsigma_{P^{\gamma\cdot \eta}\setminus P^{\gamma\cdot (\eta+1)}}(u,v).$$ Let $R=\amalg_{t\in S}R_t$.   Then $$\text{rank}(R)=\sup_{t\in S}\text{rank}(R_t)=\sup_{t\in S} \gamma\cdot \eta_t=\gamma\cdot \zeta$$   and for each $\eta<\zeta$, $$R^{\beta\cdot \eta}\setminus R^{\beta\cdot (\eta+1)}=\amalg_{t\in S} R_t^{\beta\cdot \eta}\setminus R_t^{\beta\cdot (\eta+1)}=\amalg_{t\in S}R_t\cap (P^{\gamma\cdot \eta}\setminus P^{\gamma\cdot (\eta+1)})=R\cap (P^{\gamma\cdot \eta}\setminus P^{\gamma\cdot (\eta+1)}).$$   Now fix $(u,v)\in \Lambda_2(R)$. Then there exists $t\in S$ such that $(u,v)\in \Lambda_2(R_t)$. Moreover, $\tau_{R_t, \beta}(u)=\tau_{R, \beta}(u)$ and $\tau_{R_t, \beta}(v)=\tau_{R,\beta}(v)$.  It follows from the properties of $R_t$ that either $\tau_{R, \beta}(u)>\tau_{R, \beta}(v)$ or, with $\eta=\tau_{R, \beta}(u)=\tau_{R,\beta}(v)=\tau_{P, \gamma}(u)=\tau_{P, \gamma}(v)$,  $$b(\varsigma_{R^{\beta\cdot \eta}\setminus R^{\beta\cdot (\eta+1)}}(u,v))=\varsigma_{P^{\gamma\cdot \eta}\setminus P^{\gamma\cdot (\eta+1)}}(u,v).$$

\end{proof}

\begin{corollary} Suppose that $\gamma$ is an acceptable, additively indecomposable ordinal and if $\gamma>1$, write $\gamma=\omega^{\omega^{\ee_0}}\cdot \ldots \cdot \omega^{\omega^{\ee_l}}$ with  $\ee_0\geqslant \ldots \geqslant \ee_l$. Then for any ordinal $\ee\leqslant \ee_l$, $\gamma\cdot \omega^{\omega^\ee}$ is acceptable. 

\label{ju}

\end{corollary}

\begin{proof} Note that $\lambda(\gamma\cdot \omega^{\omega^\ee})=\lambda(\gamma)+1$.  Fix $A\subset \lambda(\gamma \cdot \omega^{\omega^\ee})$. If $A=\varnothing$, any tree $P$ whose rank is $\gamma\cdot \omega^{\omega^\ee}$ has a singleton subset, which is a $\varnothing$-contraction.  Assume $A\neq \varnothing$. Let $B=A\cap \lambda(\gamma)$ and write $$A=\{a(0), \ldots, a(n)\}$$ with $a(0)<\ldots <a(n)$.    If $B\neq \varnothing$, write $$B=\{b(0), \ldots, b(m)\}$$ with $b(0)<\ldots <b(m)$.  Let $$\alpha= \omega^{\omega^{\ee_{a(0)}}}\cdot \ldots \cdot \omega^{\omega^{\ee_{a(n)}}}.$$       Fix a tree $P$ with $\text{rank}(P)=\gamma\cdot \omega^{\omega^\ee}$.  Let $\ee_{l+1}=\ee$.  Let us consider two cases.

Case $1$, $\lambda(\gamma)\notin A$: In this case,  $A=B$. Let us define $\beta$ by  $$\beta=\alpha=\omega^{\omega^{\ee_0} \cdot 1_A(0) +\ldots +\omega^{\ee_{l+1}}\cdot 1_A(l+1)}=\omega^{\omega^{\ee_0} \cdot 1_B(0) +\ldots +\omega^{\ee_{l+1}}\cdot 1_B(l+1)}.$$  Let $Q=P\setminus P^\gamma$ and note that $\text{rank}(Q)=\gamma$ and $\varsigma_Q(u,v)=\varsigma_P(u,v)$ for any $(u,v)\in Q$ by Proposition \ref{division}.  Since $\gamma$ is acceptable, there exists a $B$-contraction $R$ of $Q$. We verify that $R$ is also an $A$-contraction of $P$.  We note that by the definition of $B$-contraction, $\text{rank}(R)=\beta=\alpha$. Now for any $(u,v)\in \Lambda_2(R)$, $$a(\varsigma_R(u,v))=b(\varsigma_R(u,v))=\varsigma_Q(u,v)=\varsigma_P(u,v).$$

Case $2$,  $\lambda(\gamma)\in A$: Let  $$\beta=\omega^{\omega^{\ee_0} \cdot 1_B(0) +\ldots +\omega^{\ee_{l+1}}\cdot 1_B(l+1)}=\omega^{\omega^{\ee_0} \cdot 1_B(0) +\ldots +\omega^{\ee_l}\cdot 1_B(l)}.$$     Note that $\text{rank}(P)=\gamma\cdot \omega^{\omega^\ee}$ and $\alpha=\beta\cdot \omega^{\omega^\ee}$.  By Lemma \ref{proto} applied with $\zeta=\omega^{\omega^\ee}$,  there exists $R\subset P$ such that $\text{rank}(R)=\beta \cdot \omega^{\omega^\ee}$,  $R^{\beta\cdot \eta}\setminus R^{\beta \cdot (\eta+1)}=R\cap (P^{\gamma\cdot \eta}\setminus P^{\gamma\cdot (\eta+1)})$ for each $\eta<\omega^{\omega^\ee}$,  and for each $(s,t)\in \Lambda_2(R)$, either $\tau_{R, \beta}(s)=\tau_{P, \gamma}(s)>\tau_{R, \beta}(t)=\tau_{P, \gamma}(t)$ or there exists $\eta<\omega^{\omega^\ee}$ such that $\eta=\tau_{R, \beta}(s)=\tau_{P, \gamma}(s)=\tau_{R, \beta}(t)=\tau_{P, \gamma}(t)$ and $$b(\varsigma_{R^{\beta \cdot \eta}\setminus R^{\beta\cdot (\eta+1)}}(s,t))= \varsigma_{P^{\gamma\cdot \eta}\setminus P^{\gamma\cdot (\eta+1)}}(s,t).$$  We now check that $R$ is an $A$-contraction of $P$.   Fix $(s,t)\in \Lambda_2(R)$.  First suppose that $\tau_{R,\beta}(s)>\tau_{R, \beta}(t)$, in which case $\tau_{P, \gamma}(s)>\tau_{P, \gamma}(t)$. Then $\varsigma_R(s,t)=n$, the largest possible value of $\varsigma_R$, and $\varsigma_P(s,t)=l+1$, the largest possible value of $\varsigma_P$. Furthermore, since $l+1=\lambda(\gamma)\in A$, $a(n)=\max A=l+1$, so $$a(\varsigma_R(s,t))=a(n)=l+1=\varsigma_P(s,t).$$  Now suppose that $\tau_{R, \beta}(s)=\tau_{R, \beta}(t)=\eta$. Note that this implies that $B\neq \varnothing$, since $B=\varnothing$ implies $\beta=1$ and $\tau_{R, \beta}(s)>\tau_{R, \beta}(t)$ for any $(s,t)\in \Lambda_2(R)$ by Proposition \ref{ted}$(iv)$.   Moreover, since $\tau_{R, \beta}(s)=\tau_{R, \beta}(t)$, $\varsigma_R(s,t)<n$.     Then by Proposition \ref{division}, $$\varsigma_R(s,t)=\varsigma_{R^{\beta\cdot \eta}\setminus R^{\beta \cdot (\eta+1)}}(s,t)$$ and $$\varsigma_P(s,t)=\varsigma_{P^{\gamma\cdot \eta}\setminus P^{\gamma\cdot (\eta+1)}}(s,t).$$  Furthermore, by our choice of $R$ together with the fact that $\varsigma_R(s,t)\leqslant n$, $$a(\varsigma_R(s,t))=b(\varsigma_{R^{\beta\cdot \eta}\setminus R^{\beta\cdot (\eta+1)}}(s,t))=\varsigma_{P^{\gamma\cdot \eta}\setminus P^{\gamma\cdot (\eta+1)}}(s,t)=\varsigma_P(s,t).$$

\end{proof}

%\begin{proof}   Let $l=\lambda(\gamma)$.    Fix $A\subset \lambda(\gamma\omega^{\omega^\ee})=\lambda(\gamma)+1=l+1$.  If $A\subset l=\lambda(\gamma)$, then since $\gamma$ is acceptable and $\text{rank}(P\setminus P^\gamma)=\gamma$, we may fix an $A$-contraction of $P\setminus P^\gamma$. It is easy to see that this set is also an $A$-contraction of $P$. 

%Now if $l\in A$, let $B=A\cap l$.   Apply Lemma \ref{proto} with $\zeta=\omega^{\omega^\ee}$ to find $R\subset P$ satisfying the conclusions of Lemma \ref{proto}.    It is straightforward to verify that this $R$ is an $A$-contraction of $P$.  Since $A\subset l+1$ was arbitrary, we are done. 

%\end{proof}

\begin{corollary} For any additively indecomposable ordinal $\gamma$, any tree $Q$ with $\text{\emph{rank}}(Q)=\gamma$, and any  subset $A$ of $\lambda(\gamma)$, $Q$ has an $A$-contraction.

\label{henderson}
\end{corollary}

\begin{proof} Corollary \ref{henderson} is equivalent to saying that for any $l<\omega$ and any additively indecomposable ordinal $\gamma$ with $\lambda(\gamma)=l$, $\gamma$ is acceptable.    We prove this by induction on $l<\omega$.   The $l=0$ case is simply the aforementioned fact that $\gamma=1$ is acceptable.

Now suppose that for some $l< \omega$, every additively indecomposable ordinal $\gamma$ with $\lambda(\gamma)=l$ is $\gamma$ is acceptable.   Suppose that $\nu$ is an additively indecomposable ordinal  with $\lambda(\nu)=l+1$ and write $$\nu=\omega^{\omega^{\ee_0}}\cdot \ldots \omega^{\omega^{\ee_{l-1}}}\cdot \omega^{\omega^\ee}$$ with $\ee_0\geqslant \ldots \geqslant \ee_{l-1}\geqslant \ee$.   Let $$\gamma=\omega^{\omega^{\ee_0}}\cdot \ldots \cdot \omega^{\omega^{\ee_{l-1}}},$$ with the convention that $\gamma=1$ if $l=0$.    Then $\lambda(\gamma)=l$, whence $\gamma$ is acceptable by the inductive hypothesis. Therefore  $\nu=\gamma\cdot \omega^{\omega^\ee}$ is acceptable by Corollary \ref{ju}.

\end{proof}

Let us say that an additively indecomposable ordinal $\gamma$ is \emph{tractable} provided that for any tree $P$ with $\text{rank}(P)=\gamma$, any $k<\omega$, and any function $f:\Lambda_2(P)\to k+1$, there exist a subtree $Q$ of $P$ with $\text{rank}(Q)=\gamma$ and a function $F:\lambda(\gamma)\to k+1$ such that for any $(s,t)\in \Lambda_2(Q)$, $$\varsigma_P(s,t)=\varsigma_Q(s,t)$$ and  $$F(\varsigma_Q(s,t))=f(s,t).$$  Note that if $\gamma=1$, $\lambda(\gamma)=0=\varnothing$ and $\Lambda_2(P)=\varnothing$, so the existence of the function $F$ satisfying the conclusion is vacuous.  From this it easily follows that $1$ is tractable, since taking $Q=P$ satisfies the conclusions in the case $\text{rank}(P)=1$.

\begin{lemma} Fix $k,l< \omega$.    Suppose  $\gamma$ is a tractable additively indecomposable ordinal with $\lambda(\gamma)=l+1$.

\begin{enumerate}[(i)]\item   If $N< \omega$,  $P$ is a tree with $\text{\emph{rank}}(P)=\gamma\cdot (N+1)$, and $f:\Lambda_2(P)\to k+1$ is a function, then there exist a subtree $Q$ of $P$ with $\text{\emph{rank}}(Q)=\gamma \cdot(N+1)$ and for each $i\leqslant N$ a function $F_i:l+1\to k+1$ such that for each $i\leqslant N$ and any $(s,t)\in \Lambda_2(Q^{\gamma \cdot i}\setminus Q^{\gamma \cdot (i+1)})$, $$Q\cap (P^{\gamma\cdot  i}\setminus P^{\gamma \cdot(i+1)})=Q^{\gamma \cdot i}\setminus Q^{\gamma\cdot(i+1)},$$   $$\varsigma_{P^{\gamma\cdot i}\setminus P^{\gamma\cdot (i+1)}}(s,t)=\varsigma_{Q^{\gamma \cdot i}\setminus Q^{\gamma \cdot (i+1)}}(s,t),$$ and  $$f(s,t)=F_i(\varsigma_{Q^{\gamma\cdot  i}\setminus Q^{\gamma \cdot (i+1)}}(s,t)).$$

\item If $n,N<\omega$ are such that   $ n \cdot (k+1)^{\lambda(\gamma)}< N<\omega$, $P$ is a tree with $\text{\emph{rank}}(P)=\gamma\cdot (N+1)$, and $f:\Lambda_2(P)\to k+1$ is a function, then there exist $0\leqslant s(0)<\ldots <s(n)\leqslant  N$, a function $F:\lambda(\gamma)\to k+1$,  and a subtree $R$ of $P$ with $\text{\emph{rank}}(R)=\gamma \cdot (n+1)$ such that for each $i\leqslant n$ and any $(s,t)\in \Lambda_2(R^{\gamma\cdot  i}\setminus R^{\gamma\cdot (i+1)})$,  $$R^{\gamma \cdot i}\setminus R^{\gamma\cdot (i+1)}=R\cap ( P^{\gamma\cdot (s(i))}\setminus P^{\gamma\cdot (s(i)+1)}),$$   $$\varsigma_{R^{\gamma\cdot i}\setminus R^{\gamma \cdot (i+1)}}(s,t)= \varsigma_{P^{\gamma \cdot s(i)}\setminus P^{\gamma \cdot (s(i)+1)}}(s,t), $$ and  $$f(s,t)=F(\varsigma_{R^{\gamma\cdot  i}\setminus R^{\gamma\cdot (i+1)}}(s,t)).$$  \end{enumerate}

\label{medieval times}

\end{lemma}

\begin{rem}\upshape In the statement above, we have assumed $\lambda(\gamma)=l+1$ for notational convenience. The assumption that $\lambda(\gamma)=l+1$ excludes the $\lambda(\gamma)=0$, $\gamma=1$ case. However,  the analogous statements are true in the $\gamma=1$ case as well. For $(i)$, we let $Q=P$, and for $(ii)$, we let $s(i)=N-n+i$ for each $i\leqslant n$.   We then let $R=\cup_{i=0}^n P^{s(i)}\setminus P^{s(i)+1}=\cup_{i=0}^n P^{N-n+i}\setminus P^{N-n+i+1}=P^{N-n}$.    These clearly satisfy the conditions $Q^i\setminus Q^{i+1}=Q\cap (P^i\setminus P^{i+1})$ and $R^i\setminus R^{i+1}=R\cap (P^{s(i)}\setminus P^{s(i)+1})$ for all appropriate $i$.   Furthermore, since for each appropriate $i$, $\text{rank}(Q^i\setminus Q^{i+1})=\text{rank}(R^i\setminus R^{i+1})=1$ and $\Lambda_2(Q^i\setminus Q^{i+1})=\Lambda_2(R^i\setminus R^{i+1})=\varnothing$, the remaining conclusions are vacuous in the $\gamma=1$ case.

\end{rem}

\begin{proof}$(i)$ We induct on $N$. If $N=0$, this is equivalent to $\gamma$ being tractable. Suppose that for some $0<N$, the statement holds for $N-1$ and $\text{rank}(P)=\gamma\cdot (N+1)$.   By applying the inductive hypothesis to $P(t)$ for each $t\in \text{Leaves}(P^{\gamma \cdot N})$, we deduce the existence of some $Q_t\subset P(t)$ with $\text{rank}(Q_t)=\gamma\cdot N$ and $F^t_i:\lambda(\gamma)\to k+1$, $i=0, \ldots, N-1$, such that for each $i< N$ and $(s,u)\in \Lambda_2(Q^{\gamma \cdot i}_t\setminus Q^{\gamma\cdot (i+1)}_t)$, $$Q^{\gamma \cdot i}_t\setminus Q^{\gamma\cdot (i+1)}_t= Q_t\cap (P(t)^{\gamma\cdot  i}\setminus P(t)^{\gamma \cdot (i+1)})= Q_t \cap (P^{\gamma\cdot  i}\setminus P^{\gamma\cdot (i+1)}),$$   $$\varsigma_{P(t)^{\gamma \cdot i}\setminus P(t)^{\gamma\cdot (i+1)}}(s,u)=\varsigma_{P^{\gamma \cdot i}\setminus P^{\gamma\cdot (i+1)}}(s,u)= \varsigma_{Q_t^{\gamma\cdot i}\setminus Q_t^{\gamma\cdot (i+1)}}(s,u),$$ and $$f(s,u)=F^t_i(\varsigma_{Q_t^{\gamma i}\setminus Q_t^{\gamma(i+1)}}(s,u)).$$  Here we are using Proposition \ref{ted}$(viii)$ to deduce that $$Q_t\cap (P(t)^{\gamma\cdot i}\setminus P(t)^{\gamma\cdot (i+1)})= Q_t\cap (P^{\gamma\cdot i}\setminus P^{\gamma\cdot (i+1)})$$ and $$\varsigma_{P(t)^{\gamma\cdot i}\setminus P(t)^{\gamma\cdot (i+1)}}(s,u)=\varsigma_{P^{\gamma\cdot i}\setminus P^{\gamma\cdot (i+1)}}(s,u).$$  Now for each tuple $\mathcal{F}=(F_i)_{i<N}$ of functions from $\lambda(\gamma)$ to $k+1$, let $$M_\mathcal{F}=\{t\in \text{Leaves}(P^{\gamma\cdot N}): (F^t_i)_{i<N}=(F_i)_{i< N}\}.$$   By Lemma \ref{tek}, there exist $\mathcal{F}=(F_i)_{i< N}$ and  $T\subset P^{\gamma\cdot N}$ such that $\text{rank}(T)=\text{rank}(P^{\gamma\cdot N})=\gamma$,  $\tau_T=\tau_{P^{\gamma\cdot N}}|_T$, and $\text{Leaves}(T)=M_\mathcal{F}\cap T$.    Since $\tau_T=\tau_{P^{\gamma\cdot N}}|_T$, it follows that $\varsigma_T=\varsigma_{P^{\gamma\cdot N}}|_{\Lambda_2(T)}$.   Applying the tractability of $\gamma$ to the restriction of $f$ to $T$ yields the existence of some $F_N:\lambda(\gamma)\to k+1$ and some $D\subset T$ with $\text{rank}(D)=\gamma$ such that $\varsigma_D(s,t)=\varsigma_T(s,t)=\varsigma_P(s,t)=\varsigma_{P^{\gamma\cdot N}}(s,t)$ and $$f(s,t)=F_N(\varsigma_T(s,t))=F_N(\varsigma_D(s,t))=F_N(\varsigma_P(s,t))$$ for all $(s,t)\in \Lambda_2(D)$.   For each $t\in \text{Leaves}(D)$, fix $s_t\in \text{Leaves}(T)$ such that $t\leqslant s_t$.   Let $$Q=D\cup \Bigl(\amalg_{t\in \text{Leaves}(D)} Q_{s_t}\Bigr).$$     Then the tree $Q$ and the  collection $(F_i)_{i\leqslant N}$ satisfy the conclusions of $(i)$.  Let us verify. First, it follows from Proposition \ref{ted}$(xi)$ that $\text{rank}(Q)=\gamma\cdot (N+1)$ and for each $i\leqslant N$,  $$Q^{\gamma\cdot i}=D\cup \Bigl(\amalg_{t\in \text{Leaves}(D)}Q_{s_t}^{\gamma\cdot i}\Bigr).$$ This  yields that $$Q^{\gamma\cdot i}\setminus Q^{\gamma\cdot (i+1)}=Q\cap (P^{\gamma\cdot i}\setminus P^{\gamma\cdot (i+1)})$$ for each $i\leqslant N$.    Now fix $i\leqslant N$ and $(s,u)\in \Lambda_2(Q^{\gamma\cdot i}\setminus Q^{\gamma\cdot (i+1)})$.   Suppose first that $i=N$.  Then $(s,u)\in \Lambda_2(D)$, whence $$\varsigma_Q(s,u)=\varsigma_D(s,u)=\varsigma_{P^{\gamma\cdot N}}(s,u)=\varsigma_P(s,u)$$ by Proposition \ref{division} and our choice of $D$ and $Q$, and $$f(s,u)=F_N(\varsigma_Q(s,u))=F_N(\varsigma_P(s,u)).$$      Now suppose that $i<N$. In this case, there exists a unique $t\in  \text{Leaves}(D)$ such that $(s,u)\in \Lambda_2(Q_{s_t}^{\gamma\cdot i}\setminus Q_{s_t}^{\gamma\cdot (i+1)})$.  Since $\tau_{Q_{s_t}}=\tau_Q|_{Q_{s_t}}$, $$\varsigma_{Q^{\gamma\cdot i}\setminus Q^{\gamma\cdot (i+1)}}(s,u)=\varsigma_{Q_{s_t}^{\gamma\cdot i}\setminus Q_{s_t}^{\gamma\cdot (i+1)}}(s,u)=\varsigma_{P^{\gamma\cdot i}\setminus P^{\gamma\cdot (i+1)}}(s,u).$$ Furthermore, by our choice of $F^{s_t}_i$ and the fact that $s_t\in \text{Leaves}(T)\subset M_\mathcal{F}$, $F_i^{s_t}=F_i$ and $$f(s,u)=F_i(\varsigma_{Q_{s_t}^{\gamma\cdot i}\setminus Q_{s_t}^{\gamma\cdot (i+1)}}(s,u)).$$

$(ii)$ Suppose $P$ is a tree with $\text{rank}(P)=\gamma\cdot (N+1)$, $k<\omega$, and  $f:\Lambda_2(P)\to k+1$ is a function. By $(i)$, we may first fix $Q\subset P$ with $\text{rank}(Q)=\gamma\cdot (N+1)$ and a sequence $(F_i)_{i\leqslant N}$ of functions from $\lambda(\gamma)$ to $k+1$ such that for each $i\leqslant N$ and for each $(s,t)\in \Lambda_2(Q^{\gamma\cdot i}\setminus Q^{\gamma\cdot (i+1)})$, $$Q^{\gamma\cdot i}\setminus Q^{\gamma\cdot (i+1)}=Q\cap (P^{\gamma\cdot i}\setminus P^{\gamma\cdot (i+1)}),$$ $$\varsigma_{Q^{\gamma\cdot i}\setminus Q^{\gamma\cdot (i+1)}}(s,t)=\varsigma_{P^{\gamma\cdot i}\setminus P^{\gamma\cdot (i+1)}}(s,t),$$ and $$f(s,t)=F_i(\varsigma_{Q^{\gamma\cdot i}\setminus Q^{\gamma\cdot (i+1)}}(s,t)).$$   Now for each function $F:\lambda(\gamma)\to k+1$, let $S_F=\{i\leqslant N: F_i=F\}$. Since $N>n\cdot (k+1)^{\lambda(\gamma)}$, there exists $F:\lambda(\gamma)\to k+1$ such that $|S_F|\geqslant n+1$.   Let us enumerate a subset $A_F=\{s(0), \ldots, s(n)\}$ of $S_F$ so that $s(0)<\ldots <s(n)$.  Then by Lemma \ref{grow2}, there exists $R\subset Q$ with $\text{rank}(R)=\gamma\cdot (n+1)$ such that for each $i\leqslant n$, $$R^{\gamma\cdot i}\setminus R^{\gamma\cdot (i+1)}=R\cap (Q^{\gamma\cdot s(i)}\setminus Q^{\gamma\cdot (s(i)+1)})=R\cap (P^{\gamma\cdot s(i)}\setminus P^{\gamma\cdot (s(i)+1)})$$ and for each $(s,t)\in \Lambda_2(R^{\gamma\cdot i}\setminus R^{\gamma\cdot (i+1)})$, $$\varsigma_{R^{\gamma\cdot i}\setminus R^{\gamma\cdot (i+1)}}(s,t)=\varsigma_{Q^{\gamma\cdot s(i)}\setminus Q^{\gamma\cdot (s(i)+1)}}(s,t)=\varsigma_{P^{\gamma\cdot s(i)}\setminus P^{\gamma\cdot (s(i)+1)}}(s,t)$$ and  $$f(s,t)=F_{s(i)}(\varsigma_{Q^{\gamma\cdot s(i)}\setminus Q^{\gamma\cdot (s(i)+1)}}(s,t))=F(\varsigma_{Q^{\gamma\cdot s(i)}\setminus Q^{\gamma\cdot (s(i)+1)}}(s,t))=F(\varsigma_{R^{\gamma\cdot i}\setminus R^{\gamma\cdot (i+1)}}(s,t)).$$

\end{proof}

Given a tree $P$ such that $\text{rank}(P)$ is an  additively indecomposable ordinal, a subtree $Q$ of $P$, and a collection $A\subset \lambda(\text{rank}(P))$, we say the triple $(P,Q,A)$ is \emph{sparse} provided that $$\{\varsigma_P(s,t): (s,t)\in \Lambda_2(Q)\}\subset A.$$  

\begin{lemma} If $(P,Q,A)$ is sparse and  $\text{\emph{rank}}(P)=1$, $\text{\emph{rank}}(Q)\leqslant 1$. If $(P,Q,A)$ is sparse and $\text{\emph{rank}}(P)=\omega^{\omega^{\ee_0}}\cdot \ldots \cdot \omega^{\omega^{\ee_l}}$ with $\ee_0\geqslant \ldots \geqslant \ee_l$, then $$\text{\emph{rank}}(Q)\leqslant  \omega^{\omega^{\ee_0}\cdot 1_A(0)}\cdot \ldots\cdot  \omega^{\omega^{\ee_l}\cdot 1_A(l)}.$$   

\label{drain}
\end{lemma}

\begin{proof} We prove the result by induction on $l<\omega$ for trees $P$ with $\lambda(\text{rank}(P))=l$. For $l=0$, a tree $P$ has $\lambda(\text{rank}(P))=0$ if and only if $\text{rank}(P)=1$. In this case $A=\varnothing$ and $\text{rank}(Q)\leqslant \text{rank}(P)\leqslant 1$, as desired.

Now assume we have the result for some $l<\omega$ and we have a tree $P$ with $$\text{rank}(P)=\omega^{\omega^{\ee_0}}\cdot \ldots \cdot \omega^{\omega^{\ee_l}},$$ $\ee_0\geqslant \ldots \geqslant \ee_l$, $Q\subset P$, and $A\subset l+1=\lambda(\text{rank}(P))$ such that $$\{\varsigma_P(s,t):(s,t)\in \Lambda_2(Q)\}\subset A.$$   Let $\gamma=1$ if $l=0$ and otherwise let $$\gamma=\omega^{\omega^{\ee_0}}\cdot \ldots \cdot \omega^{\omega^{\ee_{l-1}}}$$ and $B=A\cap l$.  We first claim that for each $\delta<\omega^{\omega^{\ee_l}}$, $(P^{\gamma\cdot \delta}\setminus P^{\gamma\cdot (\delta+1)}, Q\cap (P^{\gamma\cdot \delta}\setminus P^{\gamma\cdot (\delta+1)}), B)$ is sparse. First let us note that for each such $\delta$, $\text{rank}(P^{\gamma\cdot \delta}\setminus P^{\gamma\cdot (\delta+1)})=\gamma$ and $\lambda(\gamma)=l\supset B$.     It follows from  Proposition \ref{division} that for each such $\delta$ and $(s,t)\in \Lambda_2(P^{\gamma\cdot \delta}\setminus P^{\gamma\cdot (\delta+1)})$, $\varsigma_P(s,t)=\varsigma_{P^{\gamma\cdot \delta}\setminus P^{\gamma\cdot (\delta+1)}}(s,t)$. Therefore for any $\delta<\omega^{\omega^{\ee_l}}$ and $(s,t)\in \Lambda_2(Q\cap (P^{\gamma\cdot \delta}\setminus P^{\gamma\cdot (\delta+1)}))$, $$\varsigma_{P^{\gamma\cdot \delta}\setminus P^{\gamma\cdot (\delta+1)}}(s,t)=\varsigma_P(s,t)\in A.$$   It follows  from the fact that $\text{rank}(P^{\gamma\cdot \delta}\setminus P^{\gamma\cdot (\delta+1)})=\gamma$ and $\lambda(\gamma)=l$ that for any $\delta<\omega^{\omega^{\ee_l}}$ and $(s,t)\in \Lambda_2(P^{\gamma\cdot \delta}\setminus P^{\gamma\cdot (\delta+1)})$, $$\varsigma_{P^{\gamma\cdot \delta}\setminus P^{\gamma\cdot (\delta+1)}}(s,t) <\lambda(\gamma)=l.$$  Combining the last two facts yields that for any $\delta<\omega^{\omega^{\ee_l}}$ and any $(s,t)\in \Lambda_2(Q\cap (P^{\gamma\cdot \delta}\setminus P^{\gamma\cdot (\delta+1)}))$, $$\varsigma_{P^{\gamma\cdot \delta}\setminus P^{\gamma\cdot (\delta+1)}}(s,t)\in A\cap l=B.$$  This yields that for each $\delta<\omega^{\omega^{\ee_l}}$, $(P^{\gamma\cdot \delta}\setminus P^{\gamma\cdot (\delta+1)}, Q\cap (P^{\gamma\cdot \delta}\setminus P^{\gamma\cdot (\delta+1)}), B)$ is sparse.

 By the inductive hypothesis, this yields that for each $\delta<\omega^{\omega^{\ee_l}}$, $$\text{rank}(Q\cap (P^{\gamma\cdot \delta}\setminus P^{\gamma\cdot (\delta+1)}))\leqslant \omega^{\omega^{\ee_0}\cdot 1_B(0)}\cdot \ldots \cdot \omega^{\omega^{\ee_{l-1}}\cdot 1_B(l-1)}=:\beta.$$  Now we separate into two cases.

In the first case, $l\in A$. In this case, $$\alpha:=\omega^{\omega^{\ee_0}\cdot 1_A(0)}\cdot \ldots \cdot \omega^{\omega^{\ee_l}\cdot 1_A(l)}= \beta\cdot \omega^{\omega^{\ee_l}}.$$   We claim that for each $\delta$,  $$Q^{\beta\cdot \delta}\subset P^{\gamma\cdot \delta}.$$  Applying this with $\delta=\omega^{\omega^{\ee_l}}$ will yield that $$Q^\alpha=Q^{\beta\cdot \omega^{\omega^{\ee_l}}}\subset P^{\gamma\cdot \omega^{\omega^{\ee_l}}}=P^{\text{rank}(P)}=\varnothing,$$ and $\text{rank}(Q)\leqslant \alpha$ as desired.  We prove this result by induction on $\delta$.  The $\delta=0$ case is trivial. Assume $\delta$ is a limit ordinal and the result holds for all smaller ordinals.  Then since $\sup \{\beta\cdot \zeta:\zeta<\delta\}=\beta \cdot \delta$ and $\sup \{\gamma\cdot \zeta:\zeta<\delta\}=\gamma\cdot \delta$, an appeal to Proposition \ref{ted}$(i)$ yields that $$Q^{\beta\cdot \delta}=\bigcap_{\zeta<\delta}Q^{\beta\cdot \zeta}\subset \bigcap_{\zeta<\delta} P^{\gamma\cdot \zeta}=P^{\gamma\cdot \delta}.$$  This gives the limit ordinal case. For the successor case,  suppose that $Q^{\beta\cdot \delta}\subset P^{\gamma \cdot \delta}$.     Then $$Q^{\beta \cdot \delta}\setminus  P^{\gamma\cdot (\delta+1)}\subset Q\cap (P^{\gamma\cdot \delta}\setminus P^{\gamma\cdot (\delta+1)}).$$   Since $\text{rank}(Q\cap (P^{\gamma\cdot \delta}\setminus P^{\gamma\cdot (\delta+1)}))\leqslant \beta$, $\text{rank}(Q^{\beta\cdot \delta}\setminus P^{\gamma\cdot (\delta+1)})\leqslant \beta$.  From this it follows that $Q^{\beta\cdot (\delta+1)}\subset P^{\gamma\cdot (\delta+1)}$. Indeed, if it were not so, there would exist some $t\in Q^{\beta\cdot (\delta+1)}\setminus P^{\gamma\cdot (\delta+1)}$, and $Q^{\beta\cdot \delta}[t]\subset Q^{\beta\cdot \delta}\setminus P^{\gamma\cdot (\delta+1)}$.   To see this, note that obviously $Q^{\beta\cdot \delta}[t]\subset Q^{\beta \cdot \delta}$, while $P^{\gamma\cdot (\delta+1)}$ being downward closed in $P$ and $t\notin P^{\gamma\cdot (\delta+1)}$ imply that $$Q^{\beta\cdot \delta}[t]\cap P^{\gamma\cdot (\delta+1)}\subset P[t]\cap P^{\gamma\cdot (\delta+1)}=\varnothing.$$          Now since $t\in Q^{\beta\cdot (\delta+1)}=(Q^{\beta\cdot \delta})^\beta$, $$\text{rank}(Q^{\beta\cdot \delta}\setminus P^{\gamma\cdot (\delta+1)})\geqslant \text{rank}(Q^{\beta\cdot \delta}[t])>\beta.$$  The last inequality uses Proposition \ref{ted}$(viii)$.    But $\text{rank}(Q^{\beta\cdot \delta}\setminus P^{\gamma\cdot (\delta+1)})>\beta$ is a contradiction, so no $t\in Q^{\beta\cdot (\delta+1)}\setminus P^{\gamma\cdot (\delta+1)}$ can exist, and $Q^{\beta \cdot (\delta+1)}\subset P^{\gamma\cdot (\delta+1)}$.    This completes the inductive proof and the first case.

In the second case, $l\notin A$. In this case, $\beta=\alpha$.  We claim that $$Q=Q\cap \bigcup_{\delta<\omega^{\omega^{\ee_l}}} P^{\gamma\cdot \delta}\setminus P^{\gamma\cdot (\delta+1)}= \bigcup_{\delta<\omega^{\omega^{\ee_l}}} Q\cap (P^{\gamma\cdot \delta}\setminus P^{\gamma\cdot (\delta+1)})= \amalg_{\delta<\omega^{\omega^{\ee_l}}}Q\cap (P^{\gamma\cdot \delta}\setminus P^{\gamma\cdot (\delta+1)}).$$   That is, the union in the third expression is a totally incomparable one.    The first two equalities are obvious, so we prove the third.  If $$\bigcup_{\delta<\omega^{\omega^{\ee_l}}} Q\cap (P^{\gamma\cdot \delta}\setminus P^{\gamma\cdot (\delta+1)})$$ is not an incomparable union, then there would exist distinct $\delta, \eta<\omega^{\omega^{\ee_l}}$, $s\in Q\cap (P^{\gamma\cdot \delta}\setminus P^{\gamma\cdot (\delta+1)})$, $t\in Q\cap (P^{\gamma\cdot \eta}\setminus P^{\gamma\cdot (\eta+1)})$ such that $s<t$. Then $\tau_{P, \gamma}(s)>\tau_{P, \gamma}(t)$, which yields that $\varsigma_P(s,t)=l$, the largest possible value of $\varsigma_P$.   This contradicts the sparseness, since $A\ni \varsigma_P(s,t)=l\notin A$.  Therefore $$Q=\amalg_{\delta<\omega^{\omega^{\ee_l}}} Q\cap (P^{\gamma\cdot \delta}\setminus P^{\gamma\cdot (\delta+1)}).$$   Then since $\text{rank}(Q\cap (P^{\gamma\cdot \delta}\setminus P^{\gamma\cdot (\delta+1)}))\leqslant \beta$ for each $\delta$, and $Q$ is an incomparable union of trees of rank at most $\beta$, $Q$ has rank at most $\beta=\alpha$.  This finishes the second case.

\end{proof}

\section{The main theorem}

We are now prepared to prove the main theorem, which we now recall in slightly different terminology.

\begin{theorem}  Let $\gamma$ be an additively indecomposable ordinal.  For any tree $P$ with  $\text{\emph{rank}}(P)= \gamma$, any $k<\omega$, and any $f:\Lambda_2(P)\to k+1$, there exist a subtree $Q$ of $P$ with $\text{\emph{rank}}(Q)=\gamma$ and $F:\lambda(\gamma)\to k+1$ such that for any $(s,t)\in \Lambda_2(Q)$, $F(\varsigma_Q(s,t))=f(s,t)$  and $\varsigma_Q(s,t)=\varsigma_P(s,t)$.   

That is, every additively indecomposable ordinal is tractable. 

\label{big show}

\end{theorem}

\begin{proof} We prove by induction on $\xi$ that $\omega^\xi$ is tractable. The $\xi=0$, $\omega^\xi=1$ case is trivial.    Assume $0<\xi$,  $$\xi=\omega^{\ee_0}+  \ldots +\omega^{\ee_l}$$ with $\ee_0\geqslant \ldots \geqslant \ee_l$, is such that $\omega^\zeta$ is tractable for each $\zeta<\xi$.        If $l=0$, let $\gamma=1$, and otherwise let  $$\gamma=\omega^{\omega^{\ee_0}}\cdot \ldots \cdot \omega^{\omega^{\ee_{l-1}}}.$$ Note that $\omega^\xi=\gamma\cdot \omega^{\omega^{\ee_l}}$.     For convenience, let $\ee=\ee_l$. We prove $\omega^\xi$ is tractable. We consider three cases.

Case $1$, $\ee=0$: We prove $\gamma\cdot \omega^{\omega^0}=\gamma\cdot \omega$ is tractable. Let $P$ be a tree with $\text{rank}(P)=\gamma\cdot \omega$, $k<\omega$, and $f:\Lambda_2(P)\to k+1$.  By Lemma \ref{grow3}, we may find $R\subset \text{Roots}(P)$, $\{n_t: t\in R\}\subset \omega$, and $\{s_t: t\in R\}$ such that $\sup_{t\in R}n_t=\omega$ and $s_t\in \text{Leaves}(P^{\gamma\cdot (n_t+1)}[t])$. Let $R_t=P(s_t)$, so $\text{rank}(R_t)=\gamma\cdot (n_t+1)$.   For each $q<\omega$, let $N_q=q\cdot (k+1)^{\lambda(\gamma)}+1$.     Let $r_q=r(N_q, 1, k)$, where $r(p,1,k)$ is the number from Theorem \ref{Ramsey}$(ii)$.    We may recursively select distinct roots $t_1, t_2, \ldots$, such that  $t_q\in R$ and $n_{t_q}\geqslant r_q$ for all $q<\omega$.    Indeed, let $t_1\in R$ be such that $n_{t_1}>r_1$.   If $t_1, \ldots, t_q$ have been chosen, fix $t_{q+1}\in R$ such that $$n_{t_{q+1}}>\max\{r_{q+1}, n_{t_1}, \ldots, n_{t_q}\}.$$   Now let $S_q=R_{t_q}\setminus R_{t_q}^{\gamma\cdot (r_q+1)}$. Since $\text{rank}(R_{t_q})\geqslant \gamma\cdot (r_q+1)$, $\text{rank}(S_q)=\gamma\cdot (r_q+1)$. Note that $\tau_{S_q}=\tau_{R_{t_q}}|_{S_q}=\tau_P|_{S_q}$. From this it follows that for each $i\leqslant r_q$, $$S_q^{\gamma\cdot i}\setminus S_q^{\gamma\cdot (i+1)}=S_q\cap (P^{\gamma\cdot i}\setminus P^{\gamma\cdot (i+1)})$$ and $$\varsigma_{S_q^{\gamma\cdot i}\setminus S_q^{\gamma\cdot (i+1)}}(s,u)=\varsigma_{P^{\gamma\cdot i}\setminus P^{\gamma\cdot (i+1)}}(s,u)$$ for all $(s,u)\in \Lambda_2(S_q^{\gamma\cdot i}\setminus S_q^{\gamma\cdot (i+1)})$.  By Corollary \ref{CD}, there exist a subset $T_q$ of $S_q$ with $\text{rank}(T_q)=\gamma\cdot (N_q+1)$, $j_q\leqslant k$,  and numbers $a(0)<\ldots <a(N_q)\leqslant r_q$ such that for each $i\leqslant N_q$, $T_q^{\gamma\cdot i}\setminus T_q^{\gamma\cdot (i+1)}=T_q\cap (S_q^{\gamma\cdot a(i)}\setminus S_q^{\gamma\cdot (a(i)+1)})$, $$\varsigma_{T_q^{\gamma\cdot i}\setminus T_q^{\gamma\cdot (i+1)}}(s,u)=\varsigma_{S_q^{\gamma\cdot a(i)}\setminus S_q^{\gamma\cdot (a(i)+1)}}(s,u)=\varsigma_{P^{\gamma\cdot a(i)}\setminus P^{\gamma\cdot (a(i)+1)}}(s,u)$$ for all $(s,u)\in \Lambda_2(T_q^{\gamma\cdot i}\setminus T_q^{\gamma\cdot (i+1)})$, and if $(s,u)\in \Lambda_2(T_q)$ are such that $\tau_{T_q, \gamma}(s)>\tau_{T_q,\gamma}(u)$, then $f(s,u)=j_q$.     By the definition of $N_q$ and Lemma \ref{medieval times}$(ii)$ applied to the restriction of $f$ to $\Lambda_2(T_q)$,  there exist $Q_q\subset T_q$,  a function $F_q:\lambda(\gamma)\to k+1$, and numbers $b(0)<\ldots <b(q)\leqslant N_q$ such that $\text{rank}(Q_q)=\gamma\cdot (q+1)$ and for each $i\leqslant q$, $Q_q^{\gamma\cdot i}\setminus Q_q^{\gamma\cdot (i+1)}=Q_q\cap (T_q^{\gamma\cdot b(i)}\setminus T_q^{\gamma\cdot (b(i)+1)})$ and for each $(s,u)\in \Lambda_2(Q_q^{\gamma\cdot i}\setminus Q_q^{\gamma\cdot (i+1)})$, $$\varsigma_{Q_q^{\gamma\cdot i}\setminus Q_q^{\gamma\cdot (i+1}}(s,u)= \varsigma_{T_q^{\gamma\cdot b(i)}\setminus T_q^{\gamma\cdot (b(i)+1)}}(s,u),$$ and $f(s,u)=F_q(\varsigma_{Q_q^{\gamma\cdot i}\setminus Q_q^{\gamma\cdot (i+1)}}(s,u))$.   For each $G:\lambda(\gamma)\to k+1$ and $j\leqslant k$, let $$M_{G,j}=\{q<\omega: F_q=G \text{\ and\ }j_q=j\}.$$  Then since $\omega=\cup_{G,j}M_{G,j}$, there exist $G:\lambda(\gamma)\to k+1$ and $j\leqslant k$ such that $\sup M_{G,j}=\omega$.     Let $M=M_{G,j}$ and let $Q=\amalg_{q\in M} Q_q.$  Define $F:\lambda(\gamma\cdot \omega)=l+1\to k+1$ by $F(i)=G(i)$ if $i<l$ and $F(l)=j$. We claim that these $Q$ and $F$ are as in the definition of tractable.  We note that $\text{rank}(Q)=\text{rank}(P)$ and $\varsigma_Q=\varsigma_P|_{\Lambda_2(Q)}$ by Lemma \ref{grow4}$(ii)$.  For $(s,u)\in \Lambda_2(Q)$, there exists $q\in M$ such that $(s,u)\in \Lambda_2(Q_q)$.  Then either there exists $i\leqslant q$ such that $(s,u)\in \Lambda_2(Q^{\gamma\cdot i}_q\setminus Q^{\gamma\cdot (i+1)}_q)$, in which case $\varsigma_Q(s,u)<l$ and $$f(s,u)= G(\varsigma_{Q^{\gamma\cdot i}_q\setminus Q^{\gamma\cdot (i+1)}_q}(s,u))=G(\varsigma_Q(s,u))=F_q(\varsigma_Q(s,u))=F(\varsigma_Q(s,u)),$$ or no such $i$ exists, in which case $$f(s,u)=j=F(l)=F(\varsigma_Q(s,u)).$$  Here we have used that $\tau_{Q_q}=\tau_Q|_{Q_q}$, so that $\varsigma_{Q_q^{\gamma\cdot i}\setminus Q_q^{\gamma\cdot (i+1)}}=\varsigma_{Q^{\gamma\cdot i}\setminus Q^{\gamma\cdot (i+1)}}|_{\Lambda_2(Q_q^{\gamma\cdot i}\setminus Q_q^{\gamma\cdot (i+1)})}$, and $\varsigma_Q(s,u)=l$ implies $\tau_{Q_q, \gamma}(s)>\tau_{Q_q, \gamma}(t)$, which implies $\tau_{T_q, \gamma}(s)>\tau_{T_q, \gamma}(u)$.    %  First, $$\text{rank}(Q)=\sup_{q\in M}\text{rank}(Q_q)=\sup_{q\in M} \gamma\cdot (q+1)=\gamma\cdot \omega=\text{rank}(P).$$   Now fix $(s,u)\in \Lambda_2(Q)$ and fix $q\in M$ such that $(s,u)\in \Lambda_2(Q_q)$.  We  note that $\varsigma_Q(s,u)<l$ if and only if $\varsigma_P(s,u)<l$.  To see this, note that $\varsigma_Q(s,u)<l$ if and only if there exists $i\leqslant q$ such that $s,u\in Q^{\gamma\cdot i}\setminus Q^{\gamma\cdot (i+1)}$,  in which case $s,u\in P^{\gamma\cdot a(b(i))}\setminus P^{\gamma\cdot (a(b(i))+1)}$ and by Proposition \ref{division}, $$\varsigma_Q(s,u)=\varsigma_{Q^{\gamma\cdot i}\setminus Q^{\gamma\cdot (i+1)}}(s,u)=\varsigma_{P^{\gamma\cdot a(b(i))}\setminus P^{\gamma\cdot (a(b(i))+1)}}(s,u)=\varsigma_P(s,u).$$  Conversely, $\varsigma_Q(s,u)=l$, which is the largest possible value of $\varsigma_Q$,  if and only if for some $i>m$, $\tau_{Q, \gamma}(s)=i>\tau_{Q,\gamma}(u)=m$. Then $s\in P^{\gamma\cdot a(b(i))}\setminus P^{\gamma\cdot (a(b(i))+1)}$ and $u\in P^{\gamma\cdot a(b(m))}\setminus P^{\gamma\cdot(a(b(m))+1)}$, whence $\tau_{P,\gamma}(s)>\tau_{P,\gamma}(u)$ and $\varsigma_P(s,u)=l$.       This not only shows that $\varsigma_Q(s,u)<l$ if and only if $\varsigma_P(s,u)<l$, but that either both values $\varsigma_Q(s,u), \varsigma_P(s,u)$ lie in $l$ and are equal, or both are equal to $l$. This yields that $\varsigma_Q(s,u)=\varsigma_P(s,u)$ for all $(s,u)\in \Lambda_2(Q)$.  

Case $2$, $\ee$ is a successor: Write $\ee=\delta+1$. By the inductive hypothesis,  $\gamma\cdot (\omega^{\omega^\delta})^n$ is tractable for each $n<\omega$. Let $P$ be a tree with $\text{rank}(P)=\gamma\cdot \omega^{\omega^\ee}$, $k<\omega$, and $f:\Lambda_2(P)\to k+1$.       By case $(ii)$ of Lemma \ref{grow3}, we may find $R\subset \text{Roots}(P)$, $\{n_t: t\in R\}\subset \omega$ such that $\sup_{t\in R}n_t=\omega$, and a collection $\{s_t: t\in R\}\subset P$ such that $s_t\in \text{Leaves}(P^{\gamma\cdot (\omega^{\omega^\delta})^{n_t}}[t])$. Let $R_t=P(s_t)$, so $\text{rank}(R_t)=\gamma\cdot (\omega^{\omega^\delta})^{n_t}$ and $\tau_{R_t}=\tau_P|_{R_t}$.      Let us also note that for any $(s,u)\in \Lambda_2(R_t)$, $\varsigma_P(s,u)<l$ if and only if $\varsigma_{R_t}(s,u)<l$, and in this case $\varsigma_P(s,u)=\varsigma_{R_t}(s,u)$.  To see this, note that $\varsigma_P(s,u)<l$ (resp. $\varsigma_{R_t}(s,u)<l$) if and only if there exists $\eta<(\omega^{\omega^\delta})^{n_t}$ such that $s,u\in P^{\gamma\cdot \eta}\setminus P^{\gamma\cdot (\eta+1)}$  if and only if there exists $\eta<(\omega^{\omega^\delta})^{n_t}$ such that $s,u\in R_t^{\gamma\cdot \eta}\setminus R_t^{\gamma\cdot (\eta+1)}$.  In this case, Proposition \ref{division} yields that $\varsigma_P(s,u)=\varsigma_{R_t}(s,u)$.  If no such $\eta$ exists, then $\varsigma_P(s,u)=l$.

  By the inductive hypothesis, $\gamma\cdot (\omega^{\omega^\delta})^{n_t}$ is tractable.    Let us note that $\lambda(\gamma\cdot (\omega^{\omega^\delta})^{n_t})=\lambda(\gamma)+n_t=l+n_t$.  By the inductive hypothesis, there exist $F_t:l+n_t\to k+1$ and $T_t\subset R_t$ with $\text{rank}(T_t)=\text{rank}(R_t)$ such that $\varsigma_{T_t}=\varsigma_{R_t}|_{\Lambda_2(T_t)}$ and $f(s,u)=F_t(\varsigma_{T_t}(s,u))$ for each $(s,u)\in \Lambda_2(T_t)$.  For each $G:l\to k+1$, let $$S_G=\{t\in R: G=F_t|_l\}.$$  Since $\sup_{t\in R}n_t=\omega$ and $R=\cup_G S_G$, there exists $G:l\to k+1$ such that $\sup_{t\in S_G}n_t=\omega$. Let $S=S_G$.   For each $t\in S$ and $j\leqslant k$, let $$B_{t,j}=\{i: l\leqslant i<l+n_t, F_t(i)=j\}=\{i\in \lambda(\text{rank}(T_t)): F_t(i)=j\}\setminus l.$$   Let us note that there exists $j\leqslant k$ such that $\sup_{t\in S}|B_{t,j}|=\omega$, since for each $t\in S$, $n_t=\sum_{j\leqslant k}|B_{t,j}|$ and $\sup_{t\in S}n_t=\omega$.    Let us now fix $j\leqslant k$ such that $\sup_{t\in S}|B_{t,j}|=\omega$.    Define $F:l+1\to k+1$ by $F(i)=G(i)$ if $i<l$ and $F(l)=j$.    Now let us fix distinct $t_1, t_2, \ldots\in S$ such that $|B_{t_i, j}|>i$ for all $i<\omega$. We may do this by choosing $t_1$ such that $|B_{t_1, j}|>1$ and, once $t_1, \ldots, t_i$ are chosen, choose $t_{i+1}\in S$ such that $$|B_{t_{i+1}, j}|>\max\{i+1, |B_{t_1, j}|, \ldots, |B_{t_i, j}|\}.$$

	For $i<\omega$, let $m_i=|B_{t_i, j}|$, $A_i=l\cup B_{t_i,j}$, and $Q_i$ be an $A_i$-contraction of $T_{t_i}$.  Note that by the definition of $A_i$-contraction and since $l\subset A_i$,  $$\text{rank}(Q_i)=\gamma\cdot (\omega^{\omega^\delta})^{m_i}.$$   Furthermore, for $(s,u)\in \Lambda_2(Q_i)$, $\varsigma_{Q_i}(s,u)<l$ if and only if $\varsigma_{T_{t_i}}(s,u)<l$ if and only if $\varsigma_{R_{t_i}}(s,u)<l$ if and only if $\varsigma_P(s,u)<l$, and in this case the four quantities $$\varsigma_{Q_i}(s,u), \varsigma_{T_{t_i}}(s,u), \varsigma_{R_{t_i}}(s,u), \varsigma_P(s,u)$$ are equal. The only part of this claim which does not follow immediately from the construction is the equality $\varsigma_{Q_i}(s,u)=\varsigma_{T_{t_i}}(s,u)$ if $\varsigma_{Q_i}(s,u)<l$. Let us verify the claim for $\varsigma_{Q_i}(s,u)$ and $\varsigma_{T_{t_i}}(s,u)$.    Let  $A_i=\{a(0), \ldots, a(r)\}$ for some $r$ and $a(0)<\ldots <a(r)$. By the definition of $A_i$ contraction together with the fact that $l\subset A_i$, $a(n)=n$ for each $n<l$.   Therefore  $a(\varsigma_{Q_i}(s,u))= \varsigma_{T_{t_i}}(s,u)$ if either one of these quantities is less than $l$.

				Now let $Q=\amalg_{i<\omega} Q_i$. We claim that $Q$ and $F$ are as in the definition of tractable, which will finish Case $2$.  Let us note that $\text{rank}(Q)=\text{rank}(P)$ and $\varsigma_Q=\varsigma_P|_{\Lambda_2(Q)}$ by Lemma \ref{grow4}$(i)$.    We  show that $f(s,u)=F(\varsigma_Q(s,u))$ for each $(s,u)\in \Lambda_2(Q)$. To that end, assume $(s,u)\in \Lambda_2(Q)$, so $(s,u)\in \Lambda_2(Q_i)$ for some $i<\omega$.   If $\varsigma_Q(s,u)<l$, there exists $\eta<(\omega^{\omega^\delta})^{m_i}$ such that $(s,u)\in Q^{\gamma\cdot \eta}_i\setminus Q^{\gamma\cdot (\eta+1)}_i$.    Since  $l\subset A_i$ and $Q_i$ is an $A_i$-contraction of $T_{t_i}$, \begin{align*} f(s,u) & =G(\varsigma_{T^{\gamma\cdot \eta}_{t_i}\setminus T^{\gamma\cdot (\eta+1)}_{t_i}}(s,u))=G(\varsigma_{Q_i^{\gamma\cdot \eta}\setminus Q_i^{\gamma\cdot (\eta+1)}}(s,u)) \\ & =G(\varsigma_{Q^{\gamma\cdot \eta}\setminus Q^{\gamma\cdot (\eta+1)}}(s,u))=G(\varsigma_Q(s,u))=F(\varsigma_Q(s,u)).\end{align*}    If $\varsigma_Q(s,u)=l$, then since $$\varsigma_{T_{t_i}}(s,u) = a(\varsigma_{Q_i}(s,u))\in A_i\setminus l=B_{t_i,j},$$ $f(s,u)=j=F(l)$.

% First note that $$\text{rank}(Q)=\sup_{i<\omega} \text{rank}(Q_i)=\sup_{i<\omega} \gamma\cdot (\omega^{\omega^\delta})^{m_i}=\gamma\cdot (\omega^{\omega^\delta})^\omega=\gamma\cdot \omega^{\omega^\ee}.$$   Next we argue that $\varsigma_Q=\varsigma_P|_{\Lambda_2(Q)}$.   Fix $(s,u)\in \Lambda_2(Q)$ and fix $i<\omega$ such that $(s,u)\in \Lambda_2(Q_i)$.    Then by Proposition \ref{division} together with the previously established facts, one of the four quantities $$\varsigma_Q(s,u), \varsigma_{Q_i}(s,u), \varsigma_{R_{t_i}}(s,u), \varsigma_P(s,u)$$ is less than $l$ if and only if they all are, and in this case all four are equal.  Furthermore, this happens if there exists some $\eta, \zeta$ such that $s,u\in Q^{\gamma\cdot \eta}\setminus Q^{\gamma\cdot (\eta+1)}$ and $s,u\in P^{\gamma\cdot \zeta}\setminus P^{\gamma\cdot (\zeta+1)}$. If both of these equivalent conditions fails, then $\varsigma_Q(s,u)=\varsigma_P(s,u)=l$, the largest possible value of each. This yields that $\varsigma_Q=\varsigma_P|_{\Lambda_2(Q)}$.  We last show that $f(s,u)=F(\varsigma_Q(s,u))$ for each $(s,u)\in \Lambda_2(Q)$. Still with $i$ such that $(s,u)\in \Lambda_2(Q_i)$ and $A_i=\{a(0), \ldots, a(r)\}$,  if $\varsigma_Q(s,u)<l$, then $$f(s,u)=F_{t_i}(\varsigma_{R_{t_i}}(s,u))=G(\varsigma_{R_{t_i}}(s,u))=G(\varsigma_Q(s,u)).$$   If $\varsigma_Q(s,u)= l$, then $\varsigma_{R_{t_i}}(s,u)=a(\varsigma_Q(s,u))\in B_{t_i, j}$, so $$f(s,u)=F_{t_i}(\varsigma_{R_{t_i}}(s,u))=j=F(l)=F(\varsigma_Q(s,u)).$$  

Case $3$, $\ee$ is a limit ordinal: Let $P$ be a tree with $\text{rank}(P)=\gamma\cdot \omega^{\omega^\ee}$, $k<\omega$, and $f:\Lambda_2(P)\to k+1$ be a function.    By case $(iii)$ of Lemma \ref{grow3}, we may find $R\subset \text{Roots}(P)$, $\{\ee_t: t\in R\}\subset \ee$ such that $\sup_{t\in R} \ee_t=\ee$, and a collection $\{s_t:t\in R\}\subset P$ such that $s_t\in \text{Leaves}(P^{\gamma\cdot \omega^{\omega^{\ee_t}}}[t])$.   Let $R_t=P(s_t)$, so $\text{rank}(R_t)=\gamma\cdot \omega^{\omega^{\ee_t}}$,  $\tau_{R_t}=\tau_P|_{R_t}$, and $\varsigma_{R_t}=\varsigma_P|_{\Lambda_2(R_t)}$. Here we are using the fact that $\lambda(\gamma\cdot \omega^{\omega^{\ee_t}})=\lambda(\gamma\cdot \omega^{\omega^\ee})$ and for a pair $(s,u)\in \Lambda_2(R_t)$, either there exists $\eta<\omega^{\omega^{\ee_t}}$ such that $s,u\in R_t^{\gamma\cdot \eta}\setminus R_t^{\gamma\cdot (\eta+1)}$, in which case $$\varsigma_{R_t}(s,u)= \varsigma_{R_t^{\gamma\cdot \eta}\setminus R_t^{\gamma\cdot \eta}}(s,u)=\varsigma_{P^{\gamma\cdot \eta}\setminus P^{\gamma \cdot (\eta+1)}}(s,u)=\varsigma_P(s,u) $$ by Proposition \ref{division}, or $\varsigma_{R_t}(s,u)=l=\varsigma_P(s,u)$.     Now by tractability of $\gamma\cdot \omega^{\omega^{\ee_t}}$, we may find $Q_t\subset R_t$ and $F_t:\lambda(\gamma\cdot \omega^{\omega^{\ee_t}})=\lambda(\gamma\cdot \omega^{\omega^\ee})\to k+1$ such that $\text{rank}(Q_t)=\gamma\cdot \omega^{\omega^{\ee_t}}$, $\varsigma_{Q_t}=\varsigma_{R_t}|_{\Lambda_2(Q_t)}=\varsigma_P|_{\Lambda_2(Q_t)}$,  and $f(s,u)=F_t(\varsigma_{R_t}(s,u))=F_t(\varsigma_P(s,u))$ for each  $(s,u)\in\Lambda_2(Q_t)$.    Now for each $F:\lambda(\gamma\cdot \omega^{\omega^\ee})\to k+1$, let $$M_F=\{t\in R: F_t=F\}.$$  Since $\sup_{t\in R} \ee_t=\ee$, there exists $F:\lambda(\gamma\cdot \omega^{\omega^\ee})\to k+1$ such that $$\sup_{t\in M_F} \ee_t=\ee.$$  Let $Q=\amalg_{t\in M_F} Q_t$.    Then this $Q$ and $F$ are as in the definition of tractable.  Since $P$, $k$, and $f$ were arbitrary, this concludes Case $3$. Let us verify that the conclusions are satisfied.  It follows from Lemma \ref{grow4}$(i)$ that $\text{rank}(Q)=\text{rank}(P)$ and $\varsigma_Q=\varsigma_P|_{\Lambda_2(Q)}$.    Note that for each $t\in M_F$ and $(s,u)\in \Lambda_2(Q_t)$, $\varsigma_Q(s,u)=\varsigma_{Q_t}(s,u)$.  This is because $\tau_{Q_t, \gamma}=\tau_{Q, \gamma}|_{Q_t}$, and so either $\varsigma_Q(s,u)=l=\varsigma_{Q_t}(s,u)$ if $\tau_{Q, \gamma}(s)>\tau_{Q, \gamma}(u)$, or $\varsigma_{Q, \gamma}(s)=\eta=\varsigma_{Q, \gamma}(u)$ for some $\eta<\omega^{\omega^{\ee_t}}$, in which case $$\varsigma_Q(s,u)=\varsigma_{Q^{\gamma\cdot \eta}\setminus Q^{\gamma\cdot (\eta+1)}}(s,u)=\varsigma_{Q_t^{\gamma\cdot \eta}\setminus Q_t^{\gamma\cdot (\eta+1)}}(s,u)=\varsigma_{Q_t}(s,u).$$  Thus for $(s,u)\in \Lambda_2(Q)$, there exists $t\in M_F$ such that $(s,u)\in \Lambda_2(Q_t)$, and $f(s,u)=F(\varsigma_{Q_t}(s,u))=F(\varsigma_Q(s,u))$.

\end{proof}

Theorem \ref{big show} is optimal in the following sense.

\begin{theorem} Fix  $l,k<\omega$. \begin{enumerate}[(i)]\item Suppose $\ee_0\geqslant \ldots \geqslant \ee_l$, $P$ is a tree with $\text{\emph{rank}}(P)= \omega^{\omega^{\ee_0}}\cdot \ldots \cdot \omega^{\omega^{\ee_l}}=:\gamma$, $f:\Lambda_2(P)\to k+1$ is a function.  If $Q\subset P$ and $F:l+1\to k+1$ are as in Theorem \ref{big show}, $j\leqslant k$,  $A=\{i\leqslant l: F(i)=j\}$, and $\alpha =\omega^{\omega^{\ee_0}\cdot 1_A(0)}\cdot \ldots \cdot \omega^{\omega^{\ee_l}\cdot 1_A(l)},$ then there exists a subtree $R$ of $P$ with $\text{\emph{rank}}(R)=\alpha$ and such that $f|_{\Lambda_2(R)}\equiv j$.       \item Let $F:l+1\to k+1$ be a function.   Fix  $\ee_0\geqslant \ldots \geqslant \ee_l$ and a tree $P$ with $\text{\emph{rank}}(P)=\omega^{\omega^{\ee_0}}\cdot \ldots \cdot \omega^{\omega^{\ee_l}}$ and define $f:\Lambda_2(P)\to k+1$ by $f(s,t)=F(\varsigma_P(s,t)).$  Then if $j\leqslant k$, $A=\{i\leqslant l: F(i)=j\}$, $\alpha=\omega^{\omega^{\ee_0}\cdot 1_A(0)} \ldots \omega^{\omega^{\ee_l}\cdot 1_A(l)},$ and $Q$ is any subtree of $P$ such that $f|_{\Lambda_2(Q)}\equiv j$, it follows that $\text{\emph{rank}}(Q) \leqslant \alpha$. \end{enumerate}

\label{optimal prime}
\end{theorem}

\begin{proof}$(i)$ Let $R$ be an $A$-contraction of $Q$. Such an $A$-contraction exists by Corollary \ref{henderson}. 

$(ii)$ This follows from Lemma \ref{drain}, noting that $(P,Q,A)$ must be sparse.

\end{proof}

The following is an immediate consequence of Theorem \ref{big show}.  It follows from the fact that the infinite, multiplicatively indecomposable ordinals are those infinite, additively indecomposable ordinals with precisely one term in the product decomposition.

\begin{corollary} If $\xi$ is any ordinal, $k<\omega$, $P$ is a well-founded tree with $\text{\emph{rank}}(P)\geqslant \omega^{\omega^\xi}$, and $f:\Lambda_2(P)\to k+1$ is a function, then there exists a subtree $Q$ of $P$ with $\text{\emph{rank}}(Q)=\omega^{\omega^\xi}$ such that $f|_{\Lambda_2(Q)}$ is constant.

\label{cheese down}

\end{corollary}

\begin{proof} If $\text{rank}(P)=\omega^{\omega^\xi}$, let $R=P$. If $\text{rank}(P)>\omega^{\omega^\xi}$, fix $t\in \text{Leaves}(P^{\omega^{\omega^\xi}})$ and let $R=P(t)$. In either case, $\text{rank}(R)=\omega^{\omega^\xi}$. By applying Theorem \ref{big show} to the restriction of $f$ to $\Lambda_2(R)$, we obtain $F:1=\lambda(\omega^{\omega^\xi})\to k+1$ and $Q\subset R$ with $\text{rank}(Q)=\omega^{\omega^\xi}$ such that $f(s,t)=F(\varsigma_R(s,t))=F(0)$ for all $(s,t)\in \Lambda_2(Q)$.    That is, $f|_{\Lambda_2(Q)}\equiv F(0)$.

\end{proof}

\begin{corollary} $\mathfrak{R}_2$ is the class of multiplicatively indecomposable ordinals.

\label{bobw}
\end{corollary}

\begin{rem}\upshape We note that the class $\mathfrak{R}_2$ is precisely the class of ordinals $\delta$ such that for each $\alpha, \beta<\delta$, $\alpha\cdot \beta<\delta$.  The proofs above illustrate that this is not a coincidence, but the cause of our result. Decomposing a tree into subtrees whose ranks are left divisors of the rank of the entire tree is, in some sense, the only source of non-monochromaticity.

\end{rem}

\begin{proof}[Proof of Corollary \ref{bobw}] It is evident that $\{0,1,2\}\subset \mathfrak{R}_2$.    We know that for each ordinal $\xi$, $\omega^{\omega^\xi}\in \mathfrak{R}_2$.    By the remark preceding the theorem, it it enough to prove that if $\delta$ is an ordinal such that there exist $\alpha, \beta<\delta$ with $\alpha\cdot \beta\geqslant \delta$, then $\delta\notin \mathfrak{R}_2$. We repeat a previous example.

Let $P$ be a tree with $\text{rank}(P)=\delta$, noting that some such tree exists. Define $\varsigma:\Lambda_2(P)\to 2$ by $\varsigma(s,t)=0$ if there exists $\zeta$ such that $s,t\in P^{\alpha\cdot \zeta}\setminus P^{\alpha\cdot (\zeta+1)}$, and $\varsigma(s,t)=1$ otherwise.   We claim that if $Q\subset P$ is such that $\varsigma|_{\Lambda_2(Q)}$ is constant, then $\text{rank}(Q)\leqslant \max\{\alpha, \beta\}<\delta$, which will complete the proof.    First suppose that $Q\subset P$ is such that $\varsigma|_{\Lambda_2(Q)}\equiv 0$.   Then $$Q=\amalg_{\zeta<\beta} Q\cap (P^{\alpha\cdot \zeta}\setminus P^{\alpha \cdot (\zeta+1)}).$$  If it were not so, there would exist $(s,t)\in \Lambda_2(Q)$ such that $\varsigma(s,t)=1$.    Since $\text{rank}(Q\cap (P^{\alpha \cdot \zeta}\setminus P^{\alpha \cdot (\zeta+1)}))\leqslant \alpha$, $\text{rank}(Q)\leqslant \alpha$ in this case.

Now suppose $\varsigma|_{\Lambda_2(Q)}\equiv 1$.  Then $\text{rank}(Q\cap (P^{\alpha \cdot \zeta}\setminus P^{\alpha \cdot (\zeta+1)}))\leqslant 1$ for each $\zeta$.    From this we deduce, as in the proof of Lemma \ref{drain}, that $Q^\zeta\subset P^{\alpha \cdot \zeta}$ for each ordinal $\zeta$.  The proof is by induction, with the $\zeta=0$ and $\zeta$ a limit case trivial.     For the successor case, if $Q^\zeta\subset P^{\alpha \cdot \zeta}$, then since $$Q^\zeta\setminus P^{\alpha \cdot (\zeta+1)}\subset Q\cap (P^{\alpha \cdot \zeta}\setminus P^{\alpha \cdot (\zeta+1)})$$ cannot contain two comparable members $s,t$ (since if $s<t$ for $s,t$ in this set, $\varsigma(s,t)=0$), $Q^{\zeta+1}\subset P^{\alpha\cdot (\zeta+1)}$.      Now since $\text{rank}(P)=\delta\leqslant \alpha \cdot \beta$, $Q^\beta \subset P^{\alpha\cdot \beta}=\varnothing$, so $\text{rank}(Q)\leqslant \beta$.

\end{proof}

\end{document}